\documentclass[11pt]{article}
\usepackage{macros}

\title{Towards the Relative Langlands Duality for Orthosymplectic Pairs}
\author{Dor Mezer}
\date{\today}

\begin{document}
\maketitle
\begin{abstract}
    In this paper we prove a conjectured equivalence of categories, showing that the S-dual of $\SO_{2n}\times \Sp_{2n}$ acting on $\CC_+^{2n}\otimes \CC_-^{2n}$ is equal to $\SO_{2n+1}\times \SO_{2n}\circlearrowright T^*\SO_{2n+1}$. This result is a particular case of a non-polarized version of the (local) relative Langlands duality of \cite{BZSV}. Similar results for the pairs $(\SO_{2n+1}, \Sp_{2n})$ and $(\GL_n, \GL_m)$ were proved in \cite{BFKT} and \cite{Fu} respectively, whereas the converse result was proved in \cite{BFT}. As a consequence of our main result, we prove that Langlands functoriality of the Derived Satake isomorphism for the pair $\Sp_{2n},\SO_{2n}$ is given by the theta correspondence.
    Our approach works (with appropriate modifications) in the general even orthosymplectic case of $\mathfrak{osp}(2m|2n)$.
\end{abstract}
\tableofcontents
\section{Introduction}
\subsection{The topic of this paper}
The main result of this paper is Theorem \ref{thm: main} below, which gives an equivalence of categories. We will call them the ``geometric category" and the ``spectral cateogry", in line with the common terminology in the Langlands program and that of \cite{BZSV} in particular.

We start with the describing the spectral category.
Given a Lie algebra $\mathfrak{g}$ we will denote by $D_\mathrm{perf}^{[]}(\mathfrak{g}^*)$ the category of perfect modules of the dg-algebra $\Sym(\mathfrak{g}[-2])$. Similarly, given a reductive group $G$, we will denote by $D_\mathrm{perf}^{[]}(T^*G)$ the category of perfect modules for the dg-algebra $\CC[G]\otimes \Sym(\mathfrak{g}[-2])$. See \ref{sec: derived Satake} as well as \cite[§6]{BZSV} for general discussion of shearing. This algebra carries a natural action of $G\times G$. There is also the natural moment map morphism $$D_\mathrm{perf}^{[]}(\mathfrak{g}^*)\otimes D_\mathrm{perf}^{[]}(\mathfrak{g}^*)\to D_\mathrm{perf}^{[]}(T^*G),$$
Corresponding to the usual moment map $T^*G\to \mathfrak{g}^*\times \mathfrak{g}^*$, which under the isomorphism $T^*G\cong G\times \mathfrak{g}^*$ corresponds to $(g,\xi)\mapsto (-\xi, \mathrm{Ad}_g(\xi))$. It is compatible with the (weak) $G\times G$ action on both sides.
More generally, given $H_1,H_2$ subgroups of $G$, we can restrict the moment map to
$$D_\mathrm{perf}^{[]}(\mathfrak{h}_1^*)\otimes D_\mathrm{perf}^{[]}(\mathfrak{h}_2^*)\to D_\mathrm{perf}^{[]}(T^*G),$$
and moreover take $H_1\times H_2$ invariants to get the equivariant moment map morphism
$$D_{\mathrm{perf}, H_1}^{[]}(\mathfrak{h}_1^*)\otimes D_{\mathrm{perf}, H_2}^{[]}(\mathfrak{h}_2^*)\to D_{\mathrm{perf}, H_1, H_2}^{[]}(T^*G).$$

Let us now describe the geometric category.
We consider the symplectic $4n^2$-dimensional vector space $\CC_-^{2n}\otimes \CC_+^{2n}$, and its loop vector space $(\CC_-^{2n}\otimes \CC_+^{2n})_\mathcal{K}$, equipped with the symplectic form given by the residue of the natural $\mathcal{K}$-valued symplectic form defined on it. We consider the completed Weyl algebra $\mathcal{W}$ of this (loop) symplectic vector space, and the category $\operatorname{\mathcal{W}-\mathrm{mod}}$ of renormalized discrete modules over it, along with its natural t-structure (see e.g. \cite[§2.1]{BFKT}).
There is a categorical action of $\operatorname{D-mod}_{-1/2}^*(\Sp_{4n^2,\mathcal{K}})$ on $\operatorname{\mathcal{W}-\mathrm{mod}}$, constructed in \cite[Lemma 2.4.1]{BDFRT}. Here the subscript $-1/2$ stands for considering D-modules twisted by the inverse of the square root of the determinant line bundle on $\Sp_{4n^2,\mathcal{K}}$. Note that by the determinant line bundle we mean the one whose fibers are isomorphic to the dual spaces to the relative determinant of the Lagrangian associated to $g\in \Sp_{4n^2,\mathcal{K}}$. This line bundle is ample. See Appendix \ref{appendix: det line bundle} for some further details and discussions on it.

We may restrict this action to $\SO_{2n,\mathcal{K}}\times \Sp_{2n, \mathcal{K}}$ to get an action of twisted D-modules on $\SO_{2n,\mathcal{K}}\times \Sp_{2n, \mathcal{K}}$. The twisting is by inverse of the square root of the restriction of the determinant line bundle on $\Sp_{4n^2,\mathcal{K}}$. However, this restriction has a square root (equivariant with respect to $\SO_{2n,\O}\times \Sp_{2n, \O}\times \SO_{2n,\O} \times \Sp_{2n, \O}$), and so we may untwist by this square root to identify the above category of twisted D-modules with untwisted D-modules (this equivalence is compatible with convolution), and in particular we have an action of $\operatorname{D-mod}^*(\SO_{2n,\mathcal{K}})\otimes \operatorname{D-mod}^*(\Sp_{2n, \mathcal{K}})$.
Taking (strong) $\SO_{2n,\O}\times \Sp_{2n,\O}$-invariants and restricting to locally compact objects (i.e. the full subcategory consisting of objects which are compact after forgetting the equivariant structure), we obtain a categorical action of
$\operatorname{D-mod}_{\SO_{2n,\O}}(\Gr_{\SO_{2n}})^c \otimes \operatorname{D-mod}_{\Sp_{2n,\O}}(\Gr_{\Sp_{2n}})^c$ on $\operatorname{\mathcal{W}-\mathrm{mod}}^{\SO_{2n,\O}\times \Sp_{2n,\O},\mathrm{lc}}$.
Here we use the notation of Section \ref{sec: D-modules}, so that when we write $\operatorname{D-mod}_{\SO_{2n,\O}}(\Gr_{\SO_{2n}})^c \otimes \operatorname{D-mod}_{\Sp_{2n,\O}}(\Gr_{\Sp_{2n}})^c$ we actually mean the full subcategories of locally compact objects.
By the derived Satake isomorphism of \cite{BF} (see Section \ref{sec: derived Satake}), we have an equivalence of monoidal categories
$$\operatorname{D-mod}_{\SO_{2n,\O}}(\Gr_{\SO_{2n}})^c \otimes \operatorname{D-mod}_{\Sp_{2n,\O}}(\Gr_{\Sp_{2n}})^c\cong D_{\mathrm{perf},\SO_{2n}}^{[]}(\mathfrak{so}^*_{2n})\otimes D_{\mathrm{perf},\SO_{2n+1}}^{[]}(\mathfrak{so}^*_{2n+1}).$$

We also let $E_0:=\mathcal{W} / (V\otimes V')_\O\in \operatorname{\mathcal{W}-\mathrm{mod}}^{\SO_{2n,\O}\times \Sp_{2n,\O},\mathrm{lc}}$ (see Section \ref{sec: notation} below).

The following is our main theorem in this paper.
\begin{theorem}\label{thm: main}
    There is an equivalence of categories:
    $$\operatorname{\mathcal{W}-\mathrm{mod}}^{\SO_{2n,\O}\times \Sp_{2n,\O},\mathrm{lc}}\cong D_{\mathrm{perf},\SO_{2n}\times \SO_{2n+1}}^{[]}(T^*\SO_{2n+1}),$$
    which is compatible with the actions of $D_{\mathrm{perf},\SO_{2n}}^{[]}(\mathfrak{so}^*_{2n})\otimes D_{\mathrm{perf},\SO_{2n+1}}^{[]}(\mathfrak{so}^*_{2n+1})$ on both sides, and under which, the object $E_0$ corresponds to the structure sheaf on $T^*\SO_{2n+1}$.
\end{theorem}
\begin{remark}\label{remark: reduced form of the Spectral category}
    Note that we have $$D_{\mathrm{perf},\SO_{2n}\times \SO_{2n+1}}^{[]}(T^*\SO_{2n+1})\cong D_{\mathrm{perf},\SO_{2n}}^{[]}(\mathfrak{so}^*_{2n+1}).$$
\end{remark}
\begin{remark}
    One may take ind-completions to the equivalence of Theorem \ref{thm: main} to get the ``large version", which is an equivalence of categories
    $$\mathrm{Ind}(\operatorname{\mathcal{W}-\mathrm{mod}}^{\SO_{2n,\O}\times \Sp_{2n,\O},\mathrm{lc}})\cong \mathrm{QCoh}_{\SO_{2n}\times \SO_{2n+1}}^{[]}(T^*\SO_{2n+1}),$$
    preserving compact objects, with similar compatibilities.
\end{remark}
\subsection{Relative Langlands duality}
Let $G$ be a reductive group and $X$ a spherical variety for $G$, and let $M:=T^*X$.
In \cite{BZSV}, it is conjectured that for any such pair, a certain S-dual pair $\check{G}\circlearrowright \check{M}$ is attached, where $\check{G}$ is the Langlands dual group of $G$, and $\check{M}$ is a hyperspherical variety for $\check{G}$, such that the following equivalence of categories holds (with certain compatibilities):
\begin{equation}\label{eq: relative Satake}\operatorname{D-mod}_{G_\O}(X_\mathcal{K})\cong D^{[]}_{\mathrm{perf}, \check{G}}(\check{M}).\end{equation}
See \cite[Conjecture 7.5.1]{BZSV} for full details.

This is expected to generalize beyond the polarizable case, i.e. to the generality where $M$ is a hyperspherical variety, not necessarily polarizable (so that $X$ is not at all defined).

Note that it is a non-trivial matter to express the left hand side of the above equivalence in terms of $M$ without using $X$, and as far as we know it is an open problem to define such a category for a general hyperspherical variety $M$. However, in certain cases, such a category is known to be expected to play this role.
\begin{remark}
While in general this generalization may give rise to an ``anomaly", which requires to include a twist while also replacing the notion of the Langlands dual group $\check{G}$, in our case the anomaly will cancel, so that no such modification is needed.
\end{remark}

The case considered in this paper is such a situation, with $M=\CC_-^{2n}\otimes \CC_+^{2n}, G=\Sp_{2n}\times \SO_{2n}$ and $\check{M}=T^*\SO_{2n+1}$.
In this framework, Theorem \ref{thm: main} can be viewed as a special case of a non-polarizable version of the above conjecture. This has been expected, e.g. by \cite[Table 1, row 6]{FiU}.

In \cite{BDFRT} and \cite{Nak} the authors give a definition of the S-dual for a certain class of varieties $M$. While this definition is incomplete and covers only part of the general conjectural framework, it does cover the case where $M$ is a symplectic representation of $G$, which includes our situation.
By this definition, the S-dual variety $\check{M}$ is given as the spectrum of the inner endomorphism algebra of the basic object in the geometric side, so that in presence of an equivalence of local categories (as in \eqref{eq: relative Satake} or Theorem \ref{thm: main}) the variety appearing in the spectral side must be the S-dual.
In view of this, Theorem \ref{thm: main} implies the following:
\begin{theorem}\label{thm: reformulation}
    The S-dual of $\Sp_{2n}\times \SO_{2n}\circlearrowright\CC_-^{2n}\otimes \CC_+^{2n}$ is $\SO_{2n+1}\times \SO_{2n}\circlearrowright T^*\SO_{2n+1}$.
\end{theorem}

The converse S-duality, along with the equivalence of local categories, was proved in \cite{BFT}. Other similar cases of S-dualities, along with the equivalences of local categories, were proved in \cite{BFT2}, \cite{BFKT}, and \cite{Fu} and \cite{CW}.

The approach taken in this paper works (with appropriate modifications) in the general even orthosymplectic case of $\mathfrak{osp}(2m|2n)$. This will be the subject of our next work. We also plan to investigate the odd orthosymplectic case of $\mathfrak{osp}(2m+1|2n)$. The conjectured S-duals are listed in \cite[Table 1]{FiU}.
\subsection{Theta correspondence as Langlands functoriality}\label{sec: functoriality}
A consequence of Theorem \ref{thm: main} is that the theta correspondence between $\Sp_{2n}$ and $\SO_{2n}$ realizes (local geometric spherical) Langlands functorialities corresponding to the natural inclusion $\SO_{2n}\to \SO_{2n+1}$.
The appearance of the Theta correspondence in realizing Langlands functorialities, and the functorial lifts between symplectic and orthogonal groups in particular, has a long history (for example, see \cite[§6]{Ral} for the classical theory and \cite{Lys} for a geometric result).
However, while these results show that the theta correspondence gives a functorial lift, our result proves that the theta correspondence gives {\bf the} functorial lift in our setting, i.e. the one corresponding under Derived Satake to the obvious map in the spectral side (up to a certain involution).

Let us explain this in more detail (cf. \cite[§12.3]{BZSV} where the global version of the general phenomenon is discussed).
\subsubsection{Local geometric Langlands functoriality}
Let $H,G$ be reductive groups over $\CC$, and consider a morphism $\check{H}\to \check{G}$ between reductive groups.

It defines a functor $$\pi: \mathrm{QCoh}_{\check{G}}^{[]}(\mathfrak{\check{g}}^*)\to \mathrm{QCoh}_{\check{H}}^{[]}(\mathfrak{\check{h}}^*),$$
given by restricting equivariance from $\check{G}$ to $\check{H}$ and then pushing forward along the projection $\mathfrak{\check{g}}^*\to \mathfrak{\check{h}}^*$.
Here the shearing is as in the derived Satake equivalence (see Section \ref{sec: derived Satake}).
We also have a functor in the opposite direction $$\nu: \mathrm{QCoh}_{\check{H}}^{[]}(\mathfrak{\check{h}}^*)\to \mathrm{QCoh}_{\check{G}}^{[]}(\mathfrak{\check{g}}^*),$$
given by pullback along $\check{\mathfrak{g}}^*\to \check{\mathfrak{h}}^*$ followed by inducing the $\check{H}$-action to a representation of $\check{G}$.

We define the ``functorial lift" functors $$\pi^L:\operatorname{D-mod}_{G_\O}(\Gr_G)\to \operatorname{D-mod}_{H_\O}(\Gr_H)$$
and
$$\nu^L:\operatorname{D-mod}_{H_\O}(\Gr_H) \to \operatorname{D-mod}_{G_\O}(\Gr_G)$$
to be the functors corresponding to $\pi$ and $\nu$ respectively under the derived Satake equivalence.
It is a general question of interest to describe $\pi^L$ and $\nu^L$ explicitly, in particular in terms intrinsic to the geometric side, without involving Langlands duality.

Rather than doing that, we will describe $i_H^L\circ \pi^L$ and $\nu^L\circ i_H^L$ for the pair $G=\Sp_{2n}, H=\SO_{2n}$, where $i_H$ is the involution of $\mathrm{QCoh}_{\check{H}}^{[]}(\mathfrak{\check{h}}^*)$ induced by the involution $x\mapsto -x$ of $\check{\mathfrak{h}}^*$, and $i_H^L$ is the involution on $\Dmod_{H_\O}(\Gr_H)$ corresponding to it under derived Satake.
\begin{remark}
    We are currently not sure how to describe the involution $i_H$ in terms intrinsic to the geometric side.
\end{remark}

\subsubsection{Local geometric theta correspondence}
Consider the functor $$\mathrm{act}_{\SO_{2n},E_0}:\operatorname{D-mod}_{\SO_{2n,\O}}(\Gr_{\SO_{2n}})\to \operatorname{Ind}(\operatorname{\mathcal{W}-\mathrm{mod}}^{\SO_{2n,\O}\times \Sp_{2n,\O},\mathrm{lc}})$$ given by acting on the object $E_0$. Similarly, consider $$\mathrm{act}_{\Sp_{2n},E_0}:\operatorname{D-mod}_{\Sp_{2n,\O}}(\Gr_{\Sp_{2n}})\to \operatorname{Ind}(\operatorname{\mathcal{W}-\mathrm{mod}}^{\SO_{2n,\O}\times \Sp_{2n,\O},\mathrm{lc}}).$$
These two functors admit right adjoints, given by inner hom from $E_0$ (see \cite[Remark 8.1.2]{BZSV}). We will denote them by $a_{\SO_{2n}}$ and $a_{\Sp_{2n}}$ respectively.

We define the Theta lift functor $\Theta:\operatorname{D-mod}_{\Sp_{2n,\O}}(\Gr_{\Sp_{2n}}) \to \operatorname{D-mod}_{\SO_{2n,\O}}(\Gr_{\SO_{2n}})$ to be $$\Theta:=a_{\SO_{2n},E_0}\circ \mathrm{act}_{\Sp_{2n}}.$$
We also define a functor in the opposite direction:
$$\Theta':=a_{\Sp_{2n},E_0}\circ \mathrm{act}_{\SO_{2n}}.$$
The following is a consequence of Theorem \ref{thm: main}:
\begin{theorem}
    Let $\pi,\nu$ be as above for the pair $G=\Sp_{2n}, H=\SO_{2n}$ with the natural inclusion $\check{H}=\SO_{2n}\to \SO_{2n+1}=\check{G}$. Then we have isomorphisms of functors $i_{\SO_{2n}}^L\circ \pi^L\cong \Theta$ and $\nu^L\circ i_{\SO_{2n}}^L\cong \Theta'$.
\end{theorem}
\begin{proof}
    Let $$\mu_1:T^*\SO_{2n+1} / (\SO_{2n}\times \SO_{2n+1})\to \mathfrak{so}^*_{2n} / \SO_{2n}$$ and $$\mu_2:T^*\SO_{2n+1} / (\SO_{2n}\times \SO_{2n+1})\to \mathfrak{so}^*_{2n+1} / \SO_{2n+1}$$ be the usual (equivariant) moment maps for the action of $\SO_{2n}\times \SO_{2n+1}$ on $T^*\SO_{2n+1}$.
    Under the isomorphism $T^*\SO_{2n+1} / (\SO_{2n}\times \SO_{2n+1})\cong \mathfrak{so}^*_{2n+1} / \SO_{2n}$, the functor $\mu_1$ comes from the negative of the natural morphism $\mathfrak{so}^*_{2n+1}\to \mathfrak{so}^*_{2n}$ (the dual to the inclusion of Lie algebras), while $\mu_2$ corresponds to the quotient morphism $\mathfrak{so}^*_{2n+1} / \SO_{2n}\to \mathfrak{so}^*_{2n+1} / \SO_{2n+1}$.
    All the above morphisms are compatible with the natural $\GG_m$-actions, hence with the shearings (see the discussion of shearing in \ref{sec: derived Satake}).

    We the following diagram:
    \[
    \begin{tikzcd}
        &\operatorname{Ind}(\operatorname{\mathcal{W}-\mathrm{mod}}^{\SO_{2n,\O}\times \Sp_{2n,\O},\mathrm{lc}})\ar[dd, "\cong"]\ar[dr, "a_{\SO_{2n}}"', shift right=2pt]\ar[dl, "a_{\Sp_{2n}}", shift left=2pt]\\
        \operatorname{D-mod}_{\Sp_{2n,\O}}(\Gr_{\Sp_{2n}})\ar[dd, "\cong"]\ar[ur, "\mathrm{act}_{\Sp_{2n}}", shift left=2pt]&&\operatorname{D-mod}_{\SO_{2n,\O}}(\Gr_{\SO_{2n}})\ar[dd, "\cong"]\ar[ul, "\mathrm{act}_{\SO_{2n}}"', shift right=2pt]\\
        &\mathrm{QCoh}_{\SO_{2n}\times \SO_{2n+1}}^{[]}(T^*\SO_{2n+1})\ar[dr, "\mu_{1,*}"', shift right=2pt]\ar[dl, "\mu_{2,*}", shift left=2pt]\\
        \mathrm{QCoh}_{\SO_{2n+1}}^{[]}(\mathfrak{so}^*_{2n+1})\ar[ur, "\mu_2^*", shift left=2pt] && \mathrm{QCoh}_{\SO_{2n}}^{[]}(\mathfrak{so}^*_{2n})\ar[ul, "\mu_1^*"', shift right=2pt]
    \end{tikzcd}
    \]
    Here the vertical arrows are the equivalence of Theorem \ref{thm: main} and the derived Satake equivalences. Out of each pair of diagonal maps, the bottom one is right adjoint to the top one.
    The compatibilities described in Theorem \ref{thm: main} give commutations of the left and right half of the diagram each, when using the upper diagonal arrows. Indeed, the upper diagonal arrows are given by acting on the basic objects.
    Passing to right adjoints, we also get commutations of each half of the diagram when using the lower diagonal arrows.

    By the above descriptions of $\mu_1$ and $\mu_2$, the composition $\mu_{1,*}\circ \mu_2^*$ is isomorphic to the functor $i_{\SO_{2n}}\circ \pi$, and the composition $\mu_{2,*}\circ \mu_2^*$ is isomorphic to $\nu \circ i_{\SO_{2n}}$, so the theorem follows from the above commutations in the diagram.
\end{proof}
\subsection{Outline of the Proof}
The first step of the proof, Proposition \ref{prop: Hecke identification}, is to prove an identification of the two Hecke actions of $\mathrm{Rep}(\SO_{2n+1})$ and $\mathrm{Rep}(\SO_{2n})$ on $\operatorname{\mathcal{W}-\mathrm{mod}}^{\SO_{2n,\O}\times \Sp_{2n,\O},\mathrm{lc}}$. That is, that the action of the category $\mathrm{Rep}(\SO_{2n+1})\cong \operatorname{D-mod}_{\Sp(V)}(\Gr_{\Sp(V)})^\heartsuit$ factors as the restriction to $\mathrm{Rep}(\SO_{2n})\cong \operatorname{D-mod}_{\SO(V')}(\Gr_{\SO(V')})^\heartsuit$ and then the Hecke action of this category. The analogous result for $\ell$-adic sheaves was proved in \cite{Lys}.
We do this by a direct computation for the generator of the category $U\in \mathrm{Rep}(\SO_{2n+1})$, and show that it is compatible with the self duality of this object. 
The above identification is manifest on the spectral side, as the map $\SO_{2n}\backslash T^*\SO_{2n+1} / \SO_{2n+1}\to \mathrm{pt}/\SO_{2n+1}$ factors as the map $\SO_{2n}\backslash T^*\SO_{2n+1} / \SO_{2n+1}\to \mathrm{pt}/\SO_{2n}$ composed with the natural quotient map.
Also, taking the point of view of Section \ref{sec: functoriality}, this result can be thought upon as the statement that the Theta correspondence gives the expected mapping of Langlands parameters.

Using the description of the spectral category given in Remark \ref{remark: reduced form of the Spectral category}, we see that it is isomorphic to the category of $\SO_{2n}$-equivariant perfect modules for the (dg) deequivariantized (over $\SO_{2n}$) endomorphism algebra of the structure sheaf (see Definition \ref{def: deeq}). Our strategy will be to describe the geometric category $\operatorname{\mathcal{W}-\mathrm{mod}}^{\SO_{2n,\O}\times \Sp_{2n,\O},\mathrm{lc}}$ in the same manner, and to construct an isomorphism of the deequivariantized endomorphism algebras.

Indeed, such a description follows from rather abstract reasonings given that we know that the object $E_0$ (which should correspond to the structure sheaf under Theorem \ref{thm: main}) generates the category $\operatorname{\mathcal{W}-\mathrm{mod}}^{\SO_{2n,\O}\times \Sp_{2n,\O},\mathrm{lc}}$ under the Hecke action, in the sense of Definition \ref{def: generation}. This is done in Section \ref{sec: generation} by classifying the irreducible objects in $(\operatorname{\mathcal{W}-\mathrm{mod}}^{\SO_{2n,\O}\times \Sp_{2n,\O},\mathrm{lc}})^\heartsuit$, which are given by dominant coweights of $\SO_{2n}$.

Then theorem \ref{thm: main} reduces to a computation of the deequivariantized endomorphism algebra of $E_0$, which we denote by $\mathscr{E}$, along with compatibilities with moment maps. The action of $\operatorname{D-mod}_{\Sp_{2n}}(\Gr_{\Sp_{2n}})$ on $E_0$ gives a map
$$\psi: \Sym(\mathfrak{so}_{2n}[-2])\cong \mathrm{End}_{\operatorname{D-mod}_{\Sp_{2n}}(\Gr_{\Sp_{2n}}),\mathrm{deeq}}(\delta_1, \delta_1)\to \mathscr{E}.$$
By Theorem \ref{thm: main}, this map is expected to be an isomorphism, and this is indeed what occupies the rest of this paper.
The compatibility with the second moment map (which is similarly constructed using the action of $\operatorname{D-mod}_{\SO_{2n}}(\Gr_{\SO_{2n}})$) is proved by reducing it to a computation concerning Chern classes (Proposition \ref{prop: psi psi' identification}).

It follows from purity considerations that $\mathscr{E}$ must be formal (see Theorem \ref{thm: purity}), and using the localization Theorem we show that it must be concentrated in even cohomological degrees and commutative (see Claim \ref{claim: loc thm}). This allows us to think about it geometrically, i.e. to consider the spectrum of its cohomology ring (we call this ring $\mathfrak{E}^\bullet$), which is a scheme carrying a $\GG_m\times \SO_{2n}$-action. $\psi$ gives a map from this scheme to $\mathfrak{so}_{2n}^*$, compatible with the action (the action of $\GG_m$ on $\mathfrak{so}_{2n}^*$ is the linear one of weight 2). We also denote $\mathfrak{G}^\bullet:=\CC[\mathfrak{so}_{2n}^*]$.

Using some invariant theory for the action of $\SO_{2n}$ on $\mathfrak{so}_{2n}^*$, we show in Proposition \ref{prop: psi-inv} that the resulting map between the categorical quotients by $\SO_{2n}$ is an isomorphism. This categorical quotient is isomorphic to $\Sigma_{2n}\times \Sigma_{2n+1}:= \mathfrak{so}_{2n}^* \git \SO_{2n} \times \mathfrak{so}_{2n}^* \git \SO_{2n+1}$.
We give a finite cover $\Sigma\epi\Sigma_{2n}\times \Sigma_{2n+1}$, such that the quotient map $\mathfrak{so}_{2n}^*\to \Sigma_{2n}\times \Sigma_{2n+1}$ admits a ``pseudo-slice" $\Sigma\to \mathfrak{so}_{2n}^*$, and such that the image of the above ``pseudo-slice" becomes dense after saturation by the $\SO_{2n}$ action, with the complement being of codimension $\geq 2$ (see Lemma \ref{lemma: Sigma saturation}).
Moreover, we show that after base change along $\Sigma\epi\Sigma_{2n}\times \Sigma_{2n+1}$ and localization along $\Sigma$, $\mathfrak{so}_{2n}^*$ becomes the trivial $\SO_{2n}$ torsor (see Corollary \ref{cor-item: upsilon gen}).
Using the localization theorem, we show in Claim \ref{claim: loc thm} that $\Spec(\mathfrak{E}^\bullet)$ too becomes the trivial $\SO_{2n}$ torsor after base change along $\Sigma\epi\Sigma_{2n}\times \Sigma_{2n+1}$ and localization along $\Sigma$, and deduce that the map $\psi$ itself becomes an isomorphism after this base change and localization (as any equivariant map between trivial torsors is an isomorphism).

Given the above, it is enough to show that the considered pseudo slice factors through $\mathrm{Spec}(\mathfrak{E}^\bullet)$.
This we do in Section \ref{sec: equivariant cohomology}, where we construct a pseudo slice $\Sigma\to \mathrm{Spec}(\mathfrak{E}^\bullet)$ and show that its image under $\psi$ is the pseudo-slice constructed earlier, which finishes the proof.
\subsection{Structure of this paper}
Let us briefly describe what each section of this paper contains.

In Section \ref{sec: notation} we introduce some notation and include preliminaries on D-modules on infinite dimensional ind-schemes, as well as the Weyl algebra and its category of modules.

Section \ref{sec: Weil rep} discusses the Geometric Weil representation. This is really a preliminary for the paper, but we include it both for completeness and to include some explicit formulas for the Weil action, which will be used in subsequent sections.

Section \ref{sec: Hecke actions} is dedicated to the proof of Proposition \ref{prop: Hecke identification}, identifying the two Hecke actions on $\operatorname{\mathcal{W}-\mathrm{mod}}^{\SO_{2n,\O}\times \Sp_{2n,\O},\mathrm{lc}}$.

Section \ref{sec: generation} is dedicated to the proof of Proposition \ref{prop:Hecke generation}, and does so by classifying the irreducible objects of $(\operatorname{\mathcal{W}-\mathrm{mod}}^{\SO_{2n,\O}\times \Sp_{2n,\O},\mathrm{lc}})^\heartsuit$. This section also includes a useful description of $\operatorname{\mathcal{W}-\mathrm{mod}}^{\SO_{2n,\O}\times \Sp_{2n,\O},\mathrm{lc}}$ as twisted equivariant D-modules on $(V\otimes V')_\mathcal{K} / (V\otimes V')_\O$.

Section \ref{sec: deeq} introduces the deequivariantized endomorphism algebra $\mathscr{E}$, constructs a map $\psi: \Sym(\mathfrak{so}_{2n+1}[-2])\to \mathscr{E}$, which will play the role of the moment map with respect to $\SO_{2n+1}$, and shows that it is enough to prove that $\psi$ is an isomorphism. In this section we also prove the compatibility with the second moment map, which is with respect to $\SO_{2n}$ (see Proposition \ref{prop: psi psi' identification}), and give some consequences of purity considerations for $\mathscr{E}$ (see Theorem \ref{thm: purity}).
As a result, we reduce Theorem \ref{thm: main} to proving that $\psi$ is an isomorphism.

Section \ref{sec: invariant theory} is concerned with the invariant theory of the action of $\SO_{2n}$ on $\mathfrak{so}_{2n+1}^*$. In particular, we construct there a ``pseudo-slice" for this action. For some of the results in this section we assume that $n\geq 2$.

In Section \ref{sec: localization} we use the Loclalization Theorem to deduce several consequences on the localization over $\Sigma_{2n}\times \Sigma_{2n+1}$ of $\mathfrak{E}^\bullet$.

In Section \ref{sec: equivariant cohomology}, we show that in the case $n\geq 2$ the pseudo-slice of Section \ref{sec: invariant theory} factorizes through $\Spec(\mathfrak{E}^\bullet)$ using $\psi$. We do this by constructing a pseudo slice $\Sigma\to \Spec(\mathfrak{E}^\bullet)$ by means of equivariant cohomology, and then show that it gives such a factorization.

Section \ref{sec: conclusion} wraps up the proof that $\psi$ is an isomorphism for $n\geq 2$.

Appendix \ref{appendix: n=1} completes the proof with the case $n=1$, which is done by direct computation of $\mathfrak{E}^\bullet$.

Appendix \ref{appendix: line bundles and quadratic forms} discusses multiplicative line bundles on the affine Grassmannian. This appendix contains a result relating certain ext-elements in the Derived Satake category given by invariant bilinear forms on $\check{\mathfrak{g}}$ to ext elements given by equivariant Chern classes of multiplicative line bundles on $\Gr_G$, via derived Satake. Moreover, the second part of this appendix discussion of the determinant line bundle.
The results of this section are used mainly in the proof of Proposition \ref{prop: psi psi' identification}.

\subsection{Acknowledgements}
The author is grateful to his advisor M. Finkelberg for his guidance and support, as well as for many interesting discussions.
We also thank Y. Zhao, S. Lysenko, and R. Travkin for useful discussions.

Many of the ideas in this paper are inspired by the work of \cite{BFKT}.

The author is partially supported by ISF grant 994/24.
\section{Notations and preliminaries}\label{sec: notation}
Throughout this paper we set $\mathcal{K}=\CC((t))$ and $\O=\CC[[t]]$. We will denote the loop group of an algebraic group $G$ over $\CC$ by $G_\mathcal{K}$, and the positive loop group by $G_\O$.
We will denote by $\Gr_G=G_\mathcal{K}/G_\O$ the affine Grassmannian of $G$.
Similarly, for any $\CC$-scheme $X$ we will denote by $X_\mathcal{K}$ and $X_\O$ the corresponding loop scheme and arc scheme.

We let $V$ be a $2n$-dimensional symplectic vector space, and $V'$ a $2n$-dimensional vector space equipped with a non-degenerate symmetric bilinear form. Let also $U,U'$ be $2n+1,2n$-dimensional vector spaces equipped with a non-degenerate symmetric bilinear form, with $U'\subseteq U$. We identify the Langlands dual groups of $\Sp(V), \SO(V')$ with $\SO(U), \SO(U')$ respectively.

Given an algebraic group $G$ we will denote by $\mathrm{Rep}(G)$ the category of algebraic representations of $G$.

\subsection{Conventions on the usage of categorical language}
Usually in this paper, we use the word category to mean a dg-category over $\CC$, and use the term abelian category to mean an abelian 1-category (that is, a ``usual" non-dg category). We use $\mathrm{Vect}$ for the dg-category of $\CC$-vector spaces. In a dg-category, we use $\Hom$ to denote the derived Hom, which gives an object of $\mathrm{Vect}$, and we use $\Ext^i$ to denote its $i$-th cohomology. In particular $\Ext^0$ is the non-derived Hom.

We refer the reader to \cite{Lu} and \cite{GaR} for the theory of dg-categories and higher category theory in general.
Let us record the following useful lemma, applicable also for when $\mathrm{Rep}(G)$ is replaced by an arbitrary rigid monoidal category:
\begin{lemma}\label{lemma: rigidity}
    Let $G$ be an algebraic group, and let $M$ be a module category for $\mathrm{Rep}(G)$. Then for any $X,Y\in M$ and $V\in \mathrm{Rep}(G)$, we have a natural isomorphism
    $$\Hom(V\otimes X, Y)\cong \Hom(X, V^*\otimes Y).$$
\end{lemma}
\begin{proof}
    The result follows from the adjunction of functors in \cite[Lemma 9.3.2]{GaR} (applied to the ind completion $\mathrm{Ind}(\mathrm{Rep}(G))$) composed with the adjunction of the functor $M\to \mathrm{Ind}(\mathrm{Rep}(G))\otimes M$ of tensoring with $V$ and its right adjoint, given by $\Hom_G(V,-)$.
\end{proof}
\subsection{D-modules on infinite dimensional ind-schemes}\label{sec: D-modules}
We will use the theory of (possibly equivariant) D-modules on ind-schemes (not necessarily of finite type) as appears in \cite{Ras}. Namely, all of our considered schemes and ind-schemes are placid, so that we have $\operatorname{D-mod}^*\cong \operatorname{D-mod}^!$.
In particular, $\operatorname{D-mod}^*$ has functoriality by $*$-pushforwards and $(!,ren)$-pullbacks, while $\operatorname{D-mod}^!$ has functoriality by $!$-pushforwards and $(*,ren)$-pullbacks. For any placid ind-scheme $S$ we have the dualizing D-module $\omega_S\in\operatorname{D-mod}^!(S)$, and the renormalized dualizing D-module $\omega_S^{ren}\in \operatorname{D-mod}^*(S)$.

For an ind-scheme of ind-finite type, both of these theories coincide with the usual theory of D-modules (however the above isomorphism for placid ind-schemes $\operatorname{D-mod}^*\cong \operatorname{D-mod}^!$ does not correspond to the identity functor in this case, but rather to a certain cohomological shift). In this case we will not distinguish between them and simply write $\operatorname{D-mod}(S)$.

Given an action of an algebraic group (or more commonly for us, an arc group) $G$ on $S$, we will write $\Dmod^*_G(S)$ (or $\Dmod_G(S)$ for an ind-scheme of ind-finite type) for the {\it renormalized} category of equivariant D-modules. By definition, this is the ind completion of the category of locally compact equivariant D-modules. By locally compact equivariant D-modules we mean those equivariant D-modules that become compact after forgetting the equivariance. The compact objects in $\Dmod^*_G(S)$ are then by construction locally compact equivariant D-modules, in the usual sense.

For in-depth discussion of this renormalization, also called {\it canonical renormalization}, see \cite{Ras2}, (see also \cite[Appendix B]{BZSV}, \cite[§12]{AD}).

\subsection{Derived Satake}\label{sec: derived Satake}
We will use both $\mathrm{QCoh}^{[]}_{\check{G}}(\check{\mathfrak{g}}^*)$ and $\mathrm{QCoh}_{\check{G}}(\check{\mathfrak{g}}^*[2])$ to denote the category $(\operatorname{\check{\mathfrak{g}}[-2]-mod})^{\check{G}}$ of $\check{G}$-equivariant modules of $\check{\mathfrak{g}}[-2]$ (here in the notation $\check{\mathfrak{g}}^*[2]:=\Spec(\check{\mathfrak{g}}[-2])$ is considered as an affine dg-scheme). The category of compact objects will be denoted by $D^{[]}_{\mathrm{perf},\check{G}}(\check{\mathfrak{g}}^*)$. This is the category of perfect equivariant modules.
In the above categories we have the structure sheaf, i.e. the free module of rank 1, which we denote by $\mathcal{O}_{\check{\mathfrak{g}}^*[2]}$ as well as $\mathcal{O}^{[]}_{\check{\mathfrak{g}}^*}$.

In general if $X$ is a $\CC$-scheme with a $\GG_m$-action we use $\mathrm{QCoh}^{[]}(X)$ to mean the category of quasi-coherent sheaves on $X$, sheared by the grading defined by the $\GG_m$-action, and $D_{\mathrm{perf}}^{[]}(X)$ for the compact objects. We also denote the structure sheaf by $\O^{[]}_{X}$. For example, in the above we consider $\check{\mathfrak{g}}^*$ with the contracting $\GG_m$ action of weight 2. See \cite[§6]{BZSV} for a general discussion of Shearing.

Recall the derived Satake equivalence of \cite{BF}, of which we will use both the small and large versions:
\begin{theorem}
    For a reductive group $G$, there are equivalences of monoidal Categories
    $$\Dmod_{G_\O}(\Gr_G)^c\cong D^{[]}_{\mathrm{perf},\check{G}}(\check{\mathfrak{g}}^*),$$
    $$\Dmod_{G_\O}(\Gr_G)\cong \mathrm{QCoh}^{[]}_{\check{G}}(\check{\mathfrak{g}}^*).$$
    The second equivalence is achieved from the first one by ind-completion, and the first one may be recovered from the second by restricting to compact objects.
\end{theorem}
\begin{remark}
    Note that we are using the notation of the previous subsection for the renormalization of the categories of equivariant D-modules on ind-schemes. That is, the categories appearing in the left-hand side are that of locally compact objects (in the usual sense) and its ind-completion.
\end{remark}
\subsection{The category of Weyl-algebra modules}
As already appearing in the formulation of Theorem \ref{thm: main}, we let $\mathcal{W}$ be the completed Weyl algebra of the symplectic vector space $V\otimes V'$, and $\operatorname{\mathcal{W}-\mathrm{mod}}$ the category of renormalized discrete modules over it, along with its natural t-structure (see e.g. \cite[§2.1]{BFKT}).
Given a Lagrangian lattice $\Lambda\subset (V\otimes V')_\mathcal{K}$ (which will typically be $(V\otimes V')_\O$), we have a description of the category $\operatorname{\mathcal{W}-\mathrm{mod}}$ as twisted D-modules on $(V\otimes V')_\mathcal{K}/\Lambda$, given in Remark \ref{remark: twisted D-modules}.
If we fix a (discrete) Lagrangian colattice $\Lambda^*$ transversal to $\Lambda$ (the typical example of which being $\Lambda^* = t^{-1}(V\otimes V')[t^{-1}]$), we get an equivalence of dg-categories between $\operatorname{\mathcal{W}-\mathrm{mod}}$ and $\operatorname{D-mod}((V\otimes V')_\mathcal{K}/\Lambda)$, compatible with the t-structures.
Another option is to fix two transversal Lagrangian subspaces $L,L^*\subset V\otimes V'$, and this will give an equivalence of $\operatorname{\mathcal{W}-\mathrm{mod}}$ with $\operatorname{D-mod}^*(L_\mathcal{K}/L_\O)$, again compatible with the t-structures. We will use both of these descriptions in this paper.

We let $\mathcal{C}:=\operatorname{\mathcal{W}-\mathrm{mod}}^{\SO(V')_\O\times \Sp(V)_\O,\mathrm{lc}}$, where the action of $\SO(V')_\O\times \Sp(V)_\O$ on $\operatorname{\mathcal{W}-\mathrm{mod}}$, preserving compact objects, is induced from Lemma \ref{lemma: Weil rep} below. Here lc stands for locally compact, which by definition is the full subcategory consisting of objects which map to compact objects under the forgetful functor to (non-equivariant) $\operatorname{\mathcal{W}-\mathrm{mod}}$.
The category $\mathcal{C}$ is the geometric category appearing in the statement of Theorem \ref{thm: main}.

We let $\mathrm{oblv}:\mathcal{C}\to \operatorname{\mathcal{W}-\mathrm{mod}}^c$ be the forgetful functor, forgetting the $\Sp(V)_\O\times \SO(V')_\O$-equivariant structure.
The following proposition follows from the description of $\mathcal{C}$ given as twisted D-modules on $(V\otimes V')_\mathcal{K}/(V\otimes V')_\O$:
\begin{proposition}\label{prop:oblv}
    \begin{enumerate}
        \item The forgetful functor $$\mathrm{oblv}:\mathcal{C}\to \operatorname{\mathcal{W}-\mathrm{mod}}^c$$ is conservative.
        \item There is a unique t-structure on $\mathcal{C}$ such that $\mathrm{oblv}:\mathcal{C}\to \operatorname{\mathcal{W}-\mathrm{mod}}^c$ is t-exact.
        \item The functor $\mathrm{oblv}$ restricted to the heart $\mathcal{C}^\heartsuit$ is fully faithful.
        \item An object of $\mathcal{C}$ is in the heart if and only if its image under the forgetful functor is.
    \end{enumerate}
\end{proposition}

Let $E_0\in \operatorname{\mathcal{W}-\mathrm{mod}}$ be the module $\mathcal{W} / (V\otimes V')_\O$.
Identifying $\operatorname{\mathcal{W}-\mathrm{mod}}$ with D-modules on $(V\otimes V')_\mathcal{K}/ (V\otimes V')_\mathcal{O}$ as above, $E_0$ corresponds to the delta D-module at $0$. In particular, $E_0$ is a compact object of $\operatorname{\mathcal{W}-\mathrm{mod}}$.
It follows either from the discussion in Section \ref{sec: Weil rep} or from the description of $\mathcal{C}$ via twisted D-modules that $E_0$ has a natural $\SO(V')_\O\times \Sp(V)_\O$-equivariant structure. Such a structure is unique, by Proposition \ref{prop:oblv}.
Hence, we have the object which we denote by the same name $E_0\in \mathcal{C}$.

\section{The geometric Weil representation}\label{sec: Weil rep}
In this section we recall the geometric (D-module) Weil representation, and the corresponding Hecke action on the invariants with respect to the positive loop group.
We end this section by giving a somewhat explicit formula for the Hecke action, which will be used in the rest of this paper (see Corollary \ref{cor: Hecke calculation recipe} and Remark \ref{remark: equivariant Hecke calculation recipe}).

We do not aim to develop this theory from scratch, but rather to recall the necessary results and constructions. For references, see \cite[§10]{Ras2} and \cite[Lemma 2.4.1]{BDFRT}, and also \cite{Laf}.
The parallel theory for $\ell$-adic sheaves is developed in \cite{LL}.

Let $M$ be a symplectic finite dimensional vector space over $\CC$.
\subsection{The finite case}
We start from the finite case, and let $Weyl$ be the Weyl algebra of $M$.
We shall construct an action of $\operatorname{D-mod}(\Sp(M))$ on $Weyl\mathrm{-mod}$.

Let $\alpha: \mathfrak{sp}(M)\to Weyl$ be the map given by
$$\mathfrak{sp}(M)\cong \mathrm{Sym}^2(M)\xrightarrow{xy\mapsto \frac{1}{2}(xy+yx)} Weyl,$$
It satisfies the following properties:
\begin{enumerate}
    \item $\alpha(a)x - x\alpha(a) = a(x)$ for any $a\in \mathfrak{sp}(M)$ and $x\in M$. That is, the infinitesimal action of $\mathfrak{sp}(M)$ on $Weyl$ is given by commutation with $\alpha(-)$.
    \item $\alpha(a)\alpha(b) - \alpha(b)\alpha(a) = \alpha([a,b])$. That is, $\alpha$ is a map of Lie algebras.
    \item $\alpha$ is $\Sp(M)$-equivariant.
\end{enumerate}

Given $s\in Weyl$, there is a natural element $\tilde{s}\in \Gamma(\Sp(M), \O_{\Sp(M)}\otimes Weyl)$, whose value over $g\in \Sp(M)$ is $g(s)$.
\begin{definition}
    Let $K\in \operatorname{D-mod}(\Sp(M))\otimes Weyl-\mathrm{mod}\otimes Weyl^{\mathrm{op}}-\mathrm{mod}$ be the cohomological shift by $-\dim \Sp(M)$ of the object defined as follows:
    \begin{enumerate}
        \item As a quasi-coherent sheaf on $\Sp(M)$, $K$ is the vector bundle $\mathcal{O}_{\Sp(M)}\otimes Weyl$.
        \item The connection on $K$ is given by the formula $$\nabla_a(f\otimes s) = -\mathrm{Lie}_a(f)\otimes s - f\otimes s\alpha(a).$$
        \item The left $Weyl$ action is given by multiplication from the left on the $Weyl$ factor.
        \item The $Weyl^{\mathrm{op}}$ action by $s\in Weyl$ is given by multiplication from the right by $\tilde{s}$.
    \end{enumerate}
    One easily verifies that the above connection is flat, and commutes with the left and right $Weyl$ actions, so that $K$ is indeed an object of $\operatorname{D-mod}(\Sp(M))\otimes Weyl\mathrm{-mod}\otimes Weyl^{\mathrm{op}}{-mod}$.
\end{definition}
\begin{proposition}
    \begin{enumerate}
        \item $K$ is holonomic. That is, $K\in \operatorname{D-mod}_{\mathrm{hol}}(\Sp(M))\otimes Weyl\mathrm{-mod}\otimes Weyl^{\mathrm{op}}{-mod}$
        \item $K$ is weakly $\Sp(M)$-equivariant.
        \item $K$ as a kernel object defines an action of $\operatorname{D-mod}(\Sp(M))$ on $Weyl\mathrm{-mod}$.
    \end{enumerate}
\end{proposition}
Let us spell out (3) explicitly. Given $A\in \operatorname{D-mod}(\Sp(M))$ and $N\in Weyl\mathrm{-mod}$, we have
$$A * N := \Gamma_{dR}(\Sp(M), A\stackrel{!}{\otimes} K \otimes_{Weyl} N).$$

\subsubsection{Invertibility of the kernel}\label{sec: kernel invertibility}
Let $$K^{-1}\in\operatorname{D-mod}(\Sp(M))\otimes Weyl\mathrm{-mod}\otimes Weyl^{\mathrm{op}}{-mod}$$ be defined as the pullback of $K$ through the inverse automorphism of $\Sp(M)$.

Let $$K^{-1}\star K\in \operatorname{D-mod}(\Sp(M))\otimes Weyl\mathrm{-mod}\otimes Weyl^{\mathrm{op}}{-mod}$$ be the image of $$K^{-1}\stackrel{!}{\otimes} K\in \operatorname{D-mod}(\Sp(M))\otimes Weyl\mathrm{-mod}\otimes Weyl^{\mathrm{op}}{-mod}\otimes Weyl\mathrm{-mod}\otimes Weyl^{\mathrm{op}}{-mod}$$ under the map $1_{\Sp(M)}\otimes 1_{Weyl\mathrm{-mod}}\otimes \epsilon_{Weyl} \otimes 1_{Weyl^{op}\mathrm{-mod}}$, where $$\epsilon_{Weyl}:Weyl^{\mathrm{op}}{-mod}\otimes Weyl\mathrm{-mod}\to \mathrm{Vect}$$ is the co-unit map for the algebra $Weyl$ (that is, $N_1\boxtimes N_2\mapsto N_1\otimes_{Weyl} N_2$).
Then by definition of $K$, one verifies that $K^{-1}\star K$ is the external tensor product of the dualizing D-module on $\Sp(V)$ and the unit $Weyl \in Weyl{-mod}\otimes Weyl^{\mathrm{op}}\mathrm{-mod}$.

\begin{corollary}\label{cor: kernel invertability}
    Assume that $A\in \operatorname{D-mod}(\Sp(M))^\heartsuit$ and $N\in (Weyl\mathrm{-mod})^\heartsuit$ are simple objects. Then $A\stackrel{!}{\otimes}K \otimes_{Weyl} N\in (\operatorname{D-mod}(\Sp(M))\otimes Weyl\mathrm{-mod})^\heartsuit$ is a simple object too.
\end{corollary}
\begin{proof}
    First note that $K[\dim \Sp(M)], K^{-1}[\dim \Sp(M)]$ are in the hearts of the categories in which they are defined. This already implies that $A\stackrel{!}{\otimes}K \otimes_{Weyl} N$ is in the heart, and also that the functor $$(1_{\Sp(M)}\otimes 1_{Weyl\mathrm{-mod}}\otimes \epsilon_{Weyl})(K^{-1}\stackrel{!}{\otimes} -)$$ is t-exact. Applying it to $A\stackrel{!}{\otimes}K \otimes_{Weyl} N$ we get $$A\stackrel{!}{\otimes}(\omega_{\Sp(V)}\boxtimes Weyl)\otimes_{Weyl} N\cong A\boxtimes N\in (\operatorname{D-mod}(\Sp(M))\otimes Weyl\mathrm{-mod})^\heartsuit,$$ which is a simple object. It follows that $A\stackrel{!}{\otimes}K \otimes_{Weyl} N$ is a simple object too.
\end{proof}

\subsubsection{Calculation of the Hecke action}\label{sec: finite case Hecke action theta calculation}
We now give the finite dimensional analogue of the calculation of the Hecke action on the basic object $E_0$.

Let $G$ be a reductive group acting on $M$ by automorphisms preserving the symplectic form (i.e. we have a map $G\to \Sp(M)$).
Let $L_0$ be a Lagrangian subspace of $M$, and let $P\subset G$ and $P_0\subset \Sp(M)$ be the parabolic subgroups stabilizing it.
Let $\mathcal{D}$ be the determinant line bundle on $\Sp(M)/P_0$, defined as the inverse of the determinant of the tautological vector bundle on $\Sp(M)/P_0$ whose fiber over a point $g$ is $gL_0$. It is an ample line bundle, and its pullback to $\Sp(M)$ is trivial. We will denote its restriction to $G/P$ by the same letter
We let $\operatorname{D-mod}_\kappa$ be the category of twisted D-modules on an algebraic variety equipped with a morphism to $\Sp(M)/P_0$, where the twisting is by $\mathcal{D}^{\otimes\kappa}$ ($\kappa\in \RR$).
Since the pullback of $\mathcal{D}$ to $\Sp(M)$ is trivial, we may view the previous action of $\operatorname{D-mod}(\Sp(M))$ as an action of $\operatorname{D-mod}_{-1/2}(\Sp(M))$. In particular, we view $K$ as an object of $$\operatorname{D-mod}_{1/2}(\Sp(M))\otimes Weyl-\mathrm{mod}\otimes Weyl^{\mathrm{op}}-\mathrm{mod}.$$

Restricting the action constructed above to $G$ and passing to $P$-invariants, we get a functor $$\operatorname{D-mod}_{-1/2}(G/P)\otimes (Weyl\mathrm{-mod})^P\to Weyl\mathrm{-mod},$$ where the superscript $P$ denotes the category of strongly $P$-equivariant objects (with respect to the twisted by $-1/2$ action). This functor also upgrades to give an action of the Hecke algebra $\operatorname{D-mod}_{-1/2,P}(G/P)$ on the category of strongly $P$-equivariant objects.

We consider $E:=Weyl / L_0\in Weyl\mathrm{-mod}$.
This object is strongly $P$-equivariant (this would fail if we considered the action without the twist). This is equivalent to saying that $K\otimes _{Weyl} Weyl / L_0$ descends (with respect to the t-exact pullback) to a twisted D-module on $G/P$. This is an object of $\operatorname{D-mod}_{1/2}(G/P)\otimes Weyl\mathrm{-mod}$ which we denote by $E'$. Its image under the forgetful functor to $\operatorname{D-mod}_{1/2}(G/P)$ is described in \cite[§1.3]{Laf}, up to a cohomological shift. In order to recover $E'$ itself from this description, one should take into the account the $Weyl$-action given by left multiplication, which commutes with all other structures given. Again, the twist here is essential.

As a result, given $A\in \operatorname{D-mod}_{-1/2}(G/P)$, whose !-pullback to $G$ is $\tilde{A}$, we have $$A*E:=\tilde{A}*E \cong \Gamma_{dR}(G/P, A\stackrel{!}{\otimes} E').$$

Now let us fix a Lagrangian $L_0^*\subset M$ transversal to $L_0$, which gives an equivalence $Weyl\mathrm{-mod}\cong \operatorname{D-mod}(M/L_0)$. Under this equivalence, $E$ corresponds to the delta D-module at $0\in M/L_0$.
This way, $E'$ corresponds to an element of $$\operatorname{D-mod}_{1/2}(G/P)\otimes \operatorname{D-mod}(M/L_0)\cong \operatorname{D-mod}_{1/2}(G/P\times M/L_0),$$ which we also denote by $E'$ (in the last expression the twisting is only on the first factor).

Given $g\in G / P$, it defines a Lagrangian subspace $L_g$.
The !-restriction of $E'$ to the preimage of $g$ is isomorphic to $Weyl / L_g$, which under the above equivalence corresponds to a D-module supported on the image of $L_g$ in $M/L_0$. Let us now describe what it is as a D-module on $(L_g + L_0) / L_0$.
The restriction of the symplectic form to $(L_g + L_0) / L_g\cap L_0$ is non-degenerate, and inside the resulting symplectic vector space, we have the following three Lagrangian subspaces: $L_0 / L_g\cap L_0$, $L_g / L_g\cap L_0$, and $(L_0^*\cap (L_g + L_0) + L_g\cap L_0) / L_g\cap L_0$. The second and the third are transversal to the first, and so we can consider the third as the graph of a linear map from $L_g/ L_g\cap L_0$ to $L_0/ L_g\cap L_0$.
This gives a symmetric bilinear form on $L_g/ L_g\cap L_0\cong (L_g + L_0) / L_0$, and the D-module above is the exponential of the associated quadratic form.
The reason for this is simple - this exponential D-module is defined by the differential equations given by $\partial_\ell - B(\ell,-)$ for $\ell\in L_g$, where $B$ is the bilinear form described above. Here $\partial_\ell$ is a linear differential and $B(\ell,-)$ is a linear function.

This isomorphism is consistent along smooth families in which the dimension of $L_g\cap L_0$ is constant, and the twisting is integral. Namely, given such a family $U$, and $\sqrt{\mathcal{D}}$ a square root for the determinant line bundle on $U$, we have a polynomial function $f$ which over a point $g$ gives the above quadratic form. Then the !-restriction of $E'$ to the preimage of $U$ is isomorphic to the cohomological shift by $-\dim U$ of $\sqrt{\mathcal{D}}\otimes \exp(f)$. Here $\exp(f)$ is a D-module on the total space of the vector bundle on $U$ whose fibers are $(L_g + L_0) / L_0$. It is defined by the differential equations $\partial_\ell - B_g(\ell, -)$ as above, where $\ell\in L_g$ and $B_g$ is dependent on $g$ in a polynomial manner.

Given an element $A\in \operatorname{D-mod}_{-1/2}(G/P)$, the object $A*E$ is given by the pushforward of $A\stackrel{!}{\otimes} E'$ along the projection $G/P\times M/L_0\to M/L_0$.

In particular, let $Z$ be a closed subscheme of $G/P$, and $U$ a smooth open and dense subscheme of $Z$ such that the Lagrangian determinant line bundle has a square root $\sqrt{\mathcal{D}}$ over $U$, and the dimension of $L_g\cap L_0$ is constant for $g\in U$. Let $A$ be the twisted delta (IC) D-module on $Z$, that is, the intermediate extension (in the category of twisted D-modules) from $U$ of $\sqrt{\mathcal{D}}\otimes \omega_U[\dim U]$.
Then $A\stackrel{!}{\otimes} E'$ is supported on $Z\times M/L_0$. It is holonomic (as both $A$ and $E'$ are). Its restriction to $U\times M/L_0$ is (the pushforward of) $\sqrt{\mathcal{D}}\otimes \exp(f)$, and the !-pullback to the preimage of any $g\in Z$ is the exponential D-module on $(L_g + L_0) / L_0$ described above. By corollary \ref{cor: kernel invertability}, the pullback of $A\stackrel{!}{\otimes}E'$ to $G\times M/L_0$ is a simple object of the heart (here by pullback we mean the t-exact pullback). It follows that $A\stackrel{!}{\otimes}E'$ is a simple object of the heart too.

\subsection{The Loop group case}
We now consider the loop Weyl algebra and its category of modules, as recalled in the introduction to this paper. As there is no risk of confusion, we will denote it by $\mathcal{W}$.
We do not aim to treat the technicalities of the situation, only recall the main results and bottom lines.

Let $\mathcal{D}$ be the determinant line bundle on $\Gr(\Sp(M))$, (see Appendix \ref{appendix: det line bundle}). Among other properties, it is an ample $\Sp(M)_\O$-equivariant line bundle.
We write our twistings with respect to (pullbacks) of this line bundle.

The following lemma is given in \cite[Lemma 2.4.1]{BDFRT}.
\begin{lemma}\label{lemma: Weil rep}
    There is a categorical action of $\operatorname{D-mod}_{-1/2}^*(\Sp(M)_\mathcal{K})$ on $\operatorname{\mathcal{W}-\mathrm{mod}}$.
\end{lemma}
The underlying idea of the construction is the same as in the finite case, but there are technical difficulties arising from infinite dimensionality. They are treated in \cite{Ras2} (see in particular Section 10 there, where the notion of a Harish-Chandra data is introduced).
One important feature of the passage to the loop group case is that the analogue of the map $\alpha$ defined in the finite case exists only as a map from the affine Lie algebra (see \cite{FF}), and does not factor through the loop algebra. This makes the appearance of the twist by $-1/2$ essential (in particular $\mathcal{D}$ is not trivialized on $\Sp(M)_\mathcal{K}$).

The contents of Sections \ref{sec: kernel invertibility} and \ref{sec: finite case Hecke action theta calculation} translate almost verbatim to the loop group case, with $L_0$ replaced by $M_\O$, $G$ replaced by $G_\mathcal{K}$, and $P$ replaced by $G_\O$. Let us call the analogue of the object $E$ that we had by the name $E_0$, matching the notation of Section \ref{sec: notation}.
We fix $t^{-1}M[t^{-1}]$ as a Lagrangian subspace of $M_\mathcal{K}$ transversal to $M_\O$, using which we identify $\operatorname{\mathcal{W}-\mathrm{mod}}$ with $\operatorname{D-mod}(M_\mathcal{K} / M_\O)$.
Note that for any $g\in \mathrm{Gr}_G$ the space $L_g + M_\O / L_g\cap M_\O$ is finite dimensional, so that the definition of the exponential D-module given there still makes sense.

In conclusion, we have the following corollary, summarizing the results needed for the rest of the paper.
\begin{corollary}\label{cor: Hecke calculation recipe}
    Let $Z$ be a closed finite dimensional subscheme of $\mathrm{Gr}_G$, and $U$ a smooth open and dense subscheme of $Z$ such that the Lagrangian determinant line bundle has a square root $\sqrt{\mathcal{D}}$ over $U$, and the dimension of $L_g\cap L_0$ is constant for $g\in U$. Let $A$ be the twisted delta (IC) D-module on $Z$, that is, the intermediate extension (in the category of twisted D-modules) from $U$ of the cohomologically shifted dualizing D-module (so that it will be in the heart) twisted by $\sqrt{\mathcal{D}}$.

    Then, there exists $\F\in \operatorname{D-mod}_{\mathrm{hol}}(\mathrm{Gr}_G\times M_\mathcal{K} / M_\O)^\heartsuit$ such that
    \begin{enumerate}
        \item $\F$ is supported on $Z\times M_\mathcal{K} / M_\O$.
        \item The restriction of $\F$ to $U\times M_\mathcal{K} / M_\O$ is isomorphic to (the pushforward along closed embedding of) $\sqrt{\mathcal{D}}\otimes \exp(f)$. Here $\exp(f)$, described in Section \ref{sec: finite case Hecke action theta calculation}, is a D-module on the total space of the vector bundle on $U$ whose fiber over $g\in U$ is $(L_g + M_\O) / M_\O$.
        \item The !-pullback of $\F$ to the preimage of any $g\in Z$ is the exponential D-module on $L_g + M_\O / M_\O$ described in Section \ref{sec: finite case Hecke action theta calculation}, tensored with the costalk of $A$ at $g$.
        \item $\F$ is a simple object in $\operatorname{D-mod}_{\mathrm{hol}}(\mathrm{Gr}_G\times M_\mathcal{K} / M_\O)^\heartsuit$.
        \item $A*E_0\in \operatorname{D-mod}(M_\mathcal{K} / M_\O)^c$ is given by the pushforward of $\F$ along the projection $\mathrm{Gr}_G\times M_\mathcal{K} / M_\O\to M_\mathcal{K} / M_\O$.
    \end{enumerate}
\end{corollary}
\begin{remark}\label{remark: equivariant Hecke calculation recipe}
    The above can be upgraded to the $G_\O$-equivariant setting in the following sense:
    If in the above we assume $U$ to be a $G_\O$-orbit, and the square root of the determinant line bundle on $U$ to be $G_\O$-equivariant, then there is a $G_\O$ equivariant structure on $A$ which we denote $\tilde{A}$.
    Then there exists $\tilde{\F}\in \operatorname{D-mod}_{G_\O}^{\mathrm{hol}, \chi}(\mathrm{Gr}_G\times M_\mathcal{K} / M_\O)^\heartsuit$ (in the notation of Remark \ref{remark: twisted D-modules} below) such that forgetting the $G_\O$-equivariant structure we get $\F$ (note that the choice of Lagrangian $t^{-1}M[t^{-1}]$ that we made above gives an isomorphism $\operatorname{D-mod}_{\chi}(\mathrm{Gr}_G\times M_\mathcal{K} / M_\O)\cong \operatorname{D-mod}(\mathrm{Gr}_G\times M_\mathcal{K} / M_\O)$).
    $\tilde{A}* E_0\in \operatorname{D-mod}_{G_\O}^{\mathrm{hol}, \chi}(M_\mathcal{K} / M_\O)^\mathrm{lc}\cong \operatorname{\mathcal{W}-\mathrm{mod}}^{G_\O,\mathrm{lc}}$ is then given by the pushforward of $\tilde{\F}$ along the projection $\mathrm{Gr}_G\times M_\mathcal{K} / M_\O\to M_\mathcal{K} / M_\O$.
\end{remark}

\section{Identification of the two Hecke actions}\label{sec: Hecke actions}
In this section we show an identification of the two Hecke actions on $\mathcal{C}$ (Proposition \ref{prop: Hecke identification}). The parallel result for $\ell$-adic sheaves is proved in \cite{Lys}.

\subsection{Matching the basic representation}\label{sec: matching basic reps}
Consider the standard representation $U\in\mathrm{Rep}(SO(U))$. Our first goal is to prove the following
\begin{proposition}\label{prop: U Hecke Identification}
    There is an isomorphism:
    $$\phi_{\Sp(V)}(U)*E_0\cong \phi_{\SO(V')}(U|_{\SO(U')})*E_0.$$
\end{proposition}
We have $U|_{\SO(U')} = U'\oplus \mathbf{1}$.
Thus, we need to prove that $\phi_{\Sp(V)}(U)*E_0\cong \phi_{\SO(V')}(U')*E_0\oplus E_0$.

It follows from Proposition \ref{prop:oblv} that in order to prove Proposition \ref{prop: U Hecke Identification} it is enough to prove that an isomorphism exists after applying the functor $\mathrm{oblv}$, forgetting the $\Sp(V)_\O\times \SO(V')_\O$-equivariance, to both sides, and that the resulting (isomorphic) objects are in the heart of $\operatorname{\mathcal{W}-\mathrm{mod}}^c$.
\begin{definition}
    Let $\mathrm{oblv}_{\SO(V')_\O}$ and $\mathrm{oblv}_{\Sp(V)_\O}$ be the forgetful functors forgetting the $\SO(V')_\O$ and $\Sp(V)_\O$ equivariance respectively.
    Let $\mathring{\phi}_{\Sp(V)}(-)$ be the forgetful functor forgetting the left $\Sp(V)_\O$ equivariance applied to $\phi_{\Sp(V)}(-)$, so that for example $\mathring{\phi}_{\Sp(V)}(-)*E_0 = \mathrm{oblv}(\phi_{\Sp(V)}(-)*E_0)$. Similarly, define $\mathring{\phi}_{\SO(V')}(-)$.
\end{definition}

As in \cite[§3.1]{BFKT}, we identify $\Gr_{\Sp(V)}^\mathrm{min}$ with the moduli space of pairs $(h,\ell)$, where $\ell$ is a line in $V$ and $h\in \Hom(\ell, V / \ell^\perp)$, and define $$X:=\Gr_{\Sp(V)}^\mathrm{min}\times_{\PP(V)}(\mathcal{O}_{\PP(V)}(-1)\otimes V')\subset \Gr_{\Sp(V)}\times t^{-1}(V\otimes V')_\O / (V\otimes V')_\O.$$
Let $f$ be the function defined on $X$ by $f(h,\ell, v\otimes v') = (v',v')\langle v, hv\rangle$. Note that here $v\in \ell$.
By Corollary \ref{cor: Hecke calculation recipe}, There is a D-module $\mathcal{F}_\mathrm{min}$ on $\overline{\Gr_{\Sp(V)}^\mathrm{min}}\times t^{-1}(V\otimes V')_\O / (V\otimes V')_\O$, which is supported on the closure of $X$, whose restriction to $X$ is $\exp(f)$, and whose !-restriction to the preimage of $\{1\}\in \Gr_{\Sp(V)}$ is isomorphic to $E_0[-2n]$, where we consider here $E_0$ as the delta D-module at the origin $0\in t^{-1}(V\otimes V')_\O / (V\otimes V')_\O$ (here we used the fact that $\overline{\Gr}_{\Sp(V)}^\mathrm{min}$ is rationally smooth, so that the costalk at $1$ of the IC sheaf is indeed $E_0[-2n]$). Note that the preimage of $\Gr_{\Sp(V)}^\mathrm{min}$ in the closure of $X$ is $X$ itself.

Still by Corollary \ref{cor: Hecke calculation recipe}, $\mathring{\phi}_{\Sp(V)}(U)*E_0$ considered as a D-module on $t^{-1}(V\otimes V')_\O / (V\otimes V')_\O$ is isomorphic to the pushforward of $\mathcal{F}_\mathrm{min}$ along the projection $\overline{\Gr_{\Sp(V)}^\mathrm{min}}\times t^{-1}(V\otimes V')_\O / (V\otimes V')_\O \to t^{-1}(V\otimes V')_\O / (V\otimes V')_\O$.

Similarly, let $Q$ be the minuscule orbit of $\Gr_{\SO(V')}$, isomorphic to the quadratic $V'_0 / \CC^\times$, where $V'_0$ is the space of isotropic vectors in $V'$. Let $X':=\O_Q(-1)\otimes V\subset \Gr_{\SO(V')}\times t^{-1}(V\otimes V')_\O / (V\otimes V')_\O$ as in \cite[§3.1]{BFKT}, and let $\F_Q:=\omega_{X'}[-4n+2]\cong \IC_{X'}$ (the constant D-module on $X'$ lying in the heart of the t-structure).

Let us perform a short computation, showing that the quadratic form inside the exponential D-module of Corollary \ref{cor: Hecke calculation recipe} vanishes:
At a point $\ell'\in Q$, where $\ell'$ is considered an isotropic line in $V'$, the corresponding Lagrangian in $(V\otimes V')_\mathcal{K}$ is $L':=V\otimes (tV'_\O\oplus \ell'^\perp \oplus t^{-1}\ell')$. Its intersection with $V\otimes V'_\O$ is $V\otimes (tV'_\O\oplus \ell'^\perp)$, and the sum of the two Lagrangians as vector spaces is $V\otimes (V'_\O\oplus t^{-1}\ell')$. The quotient is then $V\otimes (t^{-1}\ell'\oplus V') / \ell'^\perp$. In this quotient, the image of $L'$ is $V\otimes t^{-1}\ell'$, the image of $(V\otimes V')_\O$ is $V\otimes (V' / \ell'^\perp)$, and the image of $t^{-1}V'[t^{-1}]\cap V\otimes (V'_\O\oplus t^{-1}\ell')$ is also $V\otimes t^{-1}\ell'$. Since the first and the third Lagrangians coincide, the quadratic form vanishes.

Thus, by Corollary \ref{cor: Hecke calculation recipe}, $\mathring{\phi}_{\SO(V')}(U')*E_0$ is isomorphic to the pushforward of $\F_Q$ along the projection $$X'\to t^{-1}(V\otimes V')_\O / (V\otimes V')_\O.$$

As in \cite[§3.1]{BFKT}, let $C\subset V'\otimes V\cong t^{-1}(V'\otimes V)_\O/(V'\otimes V)_\O$ be the cone formed by all pure tensors $v'\otimes v$, with $v\in V$ and $v'\in V'_0$.
Consider also $$\tilde{C}_1:=\left(V \times (V'_0\setminus \{0\})\right)/\CC^\times_{\mathrm{hyperb}}\cong X',$$
$$\tilde{C}_2:=\left((V\setminus \{0\}) \times V'_0\right)/\CC^\times_{\mathrm{hyperb}}.$$
Note that $\tilde{C}_1$ is smooth while $\tilde{C}_2$ is not.
They are equipped with natural (proper) maps to $C$, denoted $\pi_1,\pi_2$ respectively.

\begin{claim}\label{claim: Hecke computation}
    Both sides of the equation of Proposition \ref{prop: U Hecke Identification} are supported on $C$. We have
    $$\mathring{\phi}_{\SO(V')}(U')*E_0\cong \pi_{1*}\omega_{\tilde{C}_1}[-4n+2],$$
    and the push forward from $X$ to $t^{-1}(V\otimes V')_\O / (V\otimes V')_\O$ of $\exp(f)$ is isomorphic to $\pi_{2*}\omega_{\tilde{C}_2}[-4n+2].$
\end{claim}
\begin{proof}
    The claims about $\mathring{\phi}_{\SO(V')}(U')*E_0$ were proved in the last few paragraphs.

Let $$p: \Gr_{\Sp(V)}^\mathrm{min}\times (t^{-1}(V\otimes V')_\O / (V\otimes V')_\O)\to \PP(V)\times (t^{-1}(V\otimes V')_\O / (V\otimes V')_\O)$$ be the projection in the first factor, sending $(h,\ell)$ to $\ell$. We also consider $\tilde{C}_2$ as a closed subvariety of $\PP(V)\times (t^{-1}(V\otimes V')_\O / (V\otimes V')_\O)$ via the embedding $i$ sending $(v, v')$ to $(\CC v, t^{-1}v\otimes v')$.
Then $p_* \exp(f)$ is supported on $\tilde{C}_2$. Indeed, for any point $(\ell, v\otimes v')\notin \tilde{C}_2$, that is $v'\notin V'_0$, we compute the costalk of $p_* \exp(f)$ at $(\ell, v\otimes v')$ by base change to be the push-forward from $\AA^1$ of the exponential D-module of a non-constant linear function, and this push-forward vanishes.
Next we compute $i^! p_* \exp(f)$, again using base change. This is the push-forward from a line bundle of $\omega[-4n]$ (as $f$ vanishes on the preimage of $\tilde{C}_2$), that is, $\omega_{\tilde{C}_2}[-4n+2]$. We get that the $p_* \exp(f)\cong i_* \omega_{\tilde{C}_2}[-4n+2]$. Pushing this forward to $t^{-1}(V\otimes V')_\O / (V\otimes V')_\O$ gives the desired result.
\end{proof}

\begin{corollary}\label{cor: Hecke exact triangle}
    There is an exact triangle
    $$\mathrm{oblv}(E_0)[-2n]\to \mathring{\phi}_{\Sp(V)}(U)*E_0\to \pi_{2*}\omega_{\tilde{C}_2}[-4n+2]\to.$$
    In particular we see that $\mathring{\phi}_{\Sp(V)}(U)*E_0$ is regular, and so by the Riemann-Hilbert correspondence, corresponds to a constructible sheaf.
\end{corollary}
\begin{proof}
    Consider the open embedding $j$ of $X$ into its closure, and the closed embedding $i$ of its complement. Applying the push-forward to $t^{-1}(V\otimes V')_\O / (V\otimes V')_\O$ to the exact triangle
    $$i_* E_0[-2n]\cong i_* i^!\F_\mathrm{min}\to \F_\mathrm{min}\to j_*j^*\F_\mathrm{min}\cong j_* \exp(f)\to$$
    we get the desired result.
\end{proof}
As in the situation of \cite[§3.1]{BFKT}, $\pi_1$ is small. However, in our situation, the same is not true for $\pi_2$. Indeed, the fiber over $\{0\}$ is
$\left(V'_0\setminus \{0\}\right)/\CC^\times$.
It follows that $$\mathring{\phi}_{\SO(V')}(U')*E_0\cong IC_C.$$
Let $\iota_0:\{0\}\hookrightarrow C$. By base change, we have
\begin{equation}\label{eq: costalk of IC at 0}
\iota_0^!IC_C = \iota_0^! \pi_{1*} \omega_{\tilde{C}_1}[-4n+2]=H^*((V'_0\setminus\{0\}) / \CC^\times, \omega[-2]) = \bigoplus_{i=1}^{2n-1} \CC[-2i]\oplus \CC[-2n].
\end{equation}
Similarly, we have
$$\iota_0^! \pi_{2*}\omega_{\tilde{C}_2}[-4n+2]=H^*(\PP(V), \CC)[-4n+2] = \bigoplus_{i=0}^{2n-1} \CC[-2i].$$
It follows by Corollary \ref{cor: Hecke exact triangle} (and the fact the the costalks at $0$ of the left and right sides of the exact triangle are supported in even cohomological degrees) that
\begin{equation}\label{eq: costalk of push at 0}
\iota_0^! (\mathring{\phi}_{\Sp(V)}(U)*E_0) = \bigoplus_{i=0}^{2n-1} \CC[-2i] \oplus \CC[-2n].
\end{equation}
By \cite[Theorem 19.4.2]{Moch1} (Kashiwara's conjecture), since $\F_\mathrm{min}$ is simple (by Corollary \ref{cor: Hecke calculation recipe}), we know that $\mathring{\phi}_{\Sp(V)}(U)*E_0$ is semisimple, i.e. a direct sum of copies of $IC$ sheaves. Since the restriction of $\mathring{\phi}_{\Sp(V)}(U)*E_0$ to $C\setminus \{0\}$ is $\omega[-4n+2]$, we deduce that $\mathring{\phi}_{\Sp(V)}(U)*E_0$ is a direct sum of a single copy of $IC_C$ and a sheaf supported at $0$. By \eqref{eq: costalk of IC at 0} and \eqref{eq: costalk of push at 0}, we deduce that $$\mathring{\phi}_{\Sp(V)}(U)*E_0\cong IC_C\oplus E_0,$$ as desired. Using Proposition \ref{prop:oblv}, this proves Proposition \ref{prop: U Hecke Identification}.

\subsection{Matching the basic relation}
Let $\nu:\phi_{\Sp(V)}(U)*E_0\to \phi_{\SO(V')}(U'\oplus \CC)*E_0$ be an isomorphism (which exists, by Proposition \ref{prop: U Hecke Identification}).
Let $\sigma=\sigma_\nu$ be the following composition of isomorphisms:
\begin{align*}
    &\phi_{\Sp(V)}(U)*\phi_{\Sp(V)}(U)*E_0 \xrightarrow{\phi_{\Sp(V)}(U)*\nu}
    \phi_{\Sp(V)}(U)*\phi_{\SO(V')}(U'\oplus \CC)*E_0\xrightarrow{\tau}\\\to&
    \phi_{\SO(V')}(U'\oplus \CC)*\phi_{\Sp(V)}(U)*E_0\xrightarrow{\phi_{\SO(V')}(U'\oplus \CC)*\nu}\\\to&
    \phi_{\SO(V')}(U'\oplus \CC)*\phi_{\SO(V')}(U'\oplus \CC)*E_0
\end{align*}
Here $\tau$ is the natural commutation relation of the two Hecke actions.
\begin{proposition}\label{prop: Hecke relation identification}
    There is a choice of an isomorphism $\nu$ making the following diagram commute:
    \[
    \begin{tikzcd}
    E_0\ar[r, "\mathrm{id}"]\ar[d] & E_0\ar[d]\\
    \phi_{\Sp(V)}(U)*\phi_{\Sp(V)}(U)*E_0 \ar[r, "\sigma"] & \phi_{\SO(V)}(U'\oplus \CC)*\phi_{\SO(V)}(U'\oplus \CC)*E_0
    \end{tikzcd}
    \]
    Where the vertical maps are those coming from the embedding $\CC\to U\otimes U^*\cong U\otimes U$ and its restriction to $\mathrm{Rep}(\SO(U'))$. Here we have used the quadratic form on $U$ to identify it with its dual.
\end{proposition}

The rest of this section will be occupied by the proof of Proposition \ref{prop: Hecke relation identification}, which will go as follows:
\begin{enumerate}
    \item We know that
    \begin{equation}
        \begin{aligned}\label{eq: decomp of hom space}
            \Ext^0(E_0, \phi_{\Sp(V)}(U)*\phi_{\Sp(V)}(U)*E_0) &= \Ext^0(\phi_{\Sp(V)}(U)*E_0, \phi_{\Sp(V)}(U)*E_0) \\&\stackrel{\ref{prop: U Hecke Identification}, \ref{prop:oblv}}{=} \Ext^0(IC_C\oplus E_0, IC_C\oplus E_0)\cong \CC\oplus \CC. 
        \end{aligned}
    \end{equation}
    Here by $IC_C$ we mean the unique element of $\mathcal{C}^\heartsuit$ which maps to $IC_C$ after forgetting equivariance (recall that by Proposition \ref{prop:oblv}, this forgetful functor is fully faithful when restricted to $\mathcal{C}^\heartsuit$).
    \item We show that the decomposition of the above two-dimensional space into the direct sum of two one-dimensional spaces (corresponding to endomorphisms of $IC_C$ and of $E_0$ respectively) is compatible with the isomorphism $\sigma$ (note that indeed one can write the parallel of \eqref{eq: decomp of hom space} for the $\SO(V')$ action, yielding again a decomposition into one dimensional summands. Compatibility under $\sigma$ is not immediate).
    \item We then show that both vertical maps in the diagram of Proposition \ref{prop: Hecke relation identification} have non-zero contribution in each of the direct summands.
    \item We utilize the automorphism group of $IC_C\oplus E_0$ to rescale each summand for the diagram to commute.
\end{enumerate}
Let $$\alpha:E_0\to (\phi_{\SO(V')}(U')\oplus \delta_1)*E_0$$ be the embedding of the second summand, and $$\beta:(\phi_{\SO(V')}(U')\oplus \delta_1)*E_0\to E_0$$ the projection on it.
We choose an isomorphism $\nu:\phi_{\Sp(V)}(U)*E_0\xrightarrow{\sim} \phi_{\SO(V')}(U'\oplus \CC)*E_0$, and set $$\alpha' = \nu^{-1}\circ \alpha:\ E_0\to \phi_{\Sp(V)}(U)*E_0$$ and $$\beta' = \beta\circ \nu: \phi_{\Sp(V)}(U)*E_0\to E_0.$$
By construction, we have $\beta\circ\alpha=\beta'\circ\alpha'=\mathrm{id}_{E_0}$
Consider the compositions
$$a:E_0\xrightarrow{\alpha} (\phi_{\SO(V')}(U')\oplus \delta_1)*E_0\xrightarrow{(\phi_{\SO(V')}(U')\oplus \delta_1)*\alpha} (\phi_{\SO(V')}(U')\oplus \delta_1)*(\phi_{\SO(V')}(U')\oplus \delta_1)*E_0$$
and
$$b: (\phi_{\SO(V')}(U')\oplus \delta_1)*(\phi_{\SO(V')}(U')\oplus \delta_1)*E_0\xrightarrow{(\phi_{\SO(V')}(U')\oplus \delta_1)*\beta} (\phi_{\SO(V')}(U')\oplus \delta_1)*E_0\xrightarrow{\beta} E_0.$$
We have $b\circ a = \mathrm{id}_{E_0}$.
We define $a', b'$ similarly using $\alpha', \beta'$ instead of $\alpha, \beta$, and again we have $b'\circ a' = \mathrm{id}_{E_0}$.
We then have the following diagram:
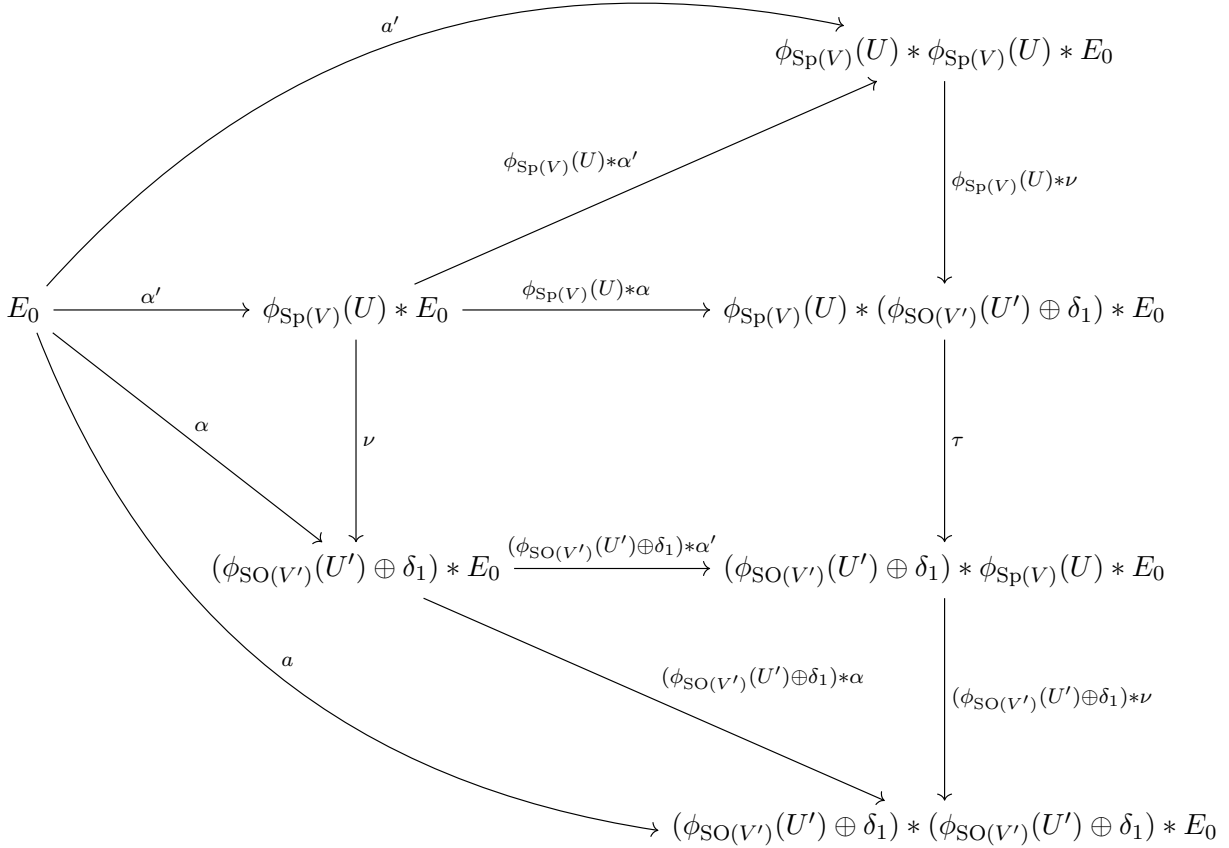
\[
\begin{tikzcd}[row sep=7em, column sep=5em]
&&\phi_{\Sp(V)}(U)*\phi_{\Sp(V)}(U)*E_0 \ar[d, "\phi_{\Sp(V)}(U)*\nu"]\\
E_0\ar[r, "\alpha'"]\ar[dr, "\alpha"]\ar[urr, "a'", bend left]\ar[ddrr, "a", bend right]&\phi_{\Sp(V)}(U)*E_0\ar[ur, "\phi_{\Sp(V)}(U)*\alpha'"]\ar[r, "\phi_{\Sp(V)}(U)*\alpha"]\ar[d, "\nu"]& \phi_{\Sp(V)}(U)*(\phi_{\SO(V')}(U')\oplus\delta_1)*E_0\ar[d, "\tau"]\\
&(\phi_{\SO(V')}(U')\oplus \delta_1)*E_0\ar[r, "(\phi_{\SO(V')}(U')\oplus \delta_1)*\alpha'"]\ar[dr, "(\phi_{\SO(V')}(U')\oplus \delta_1)*\alpha"]& (\phi_{\SO(V')}(U')\oplus \delta_1)*\phi_{\Sp(V)}(U)*E_0\ar[d, "(\phi_{\SO(V')}(U')\oplus \delta_1)*\nu"]\\
&&(\phi_{\SO(V')}(U')\oplus \delta_1)*(\phi_{\SO(V')}(U')\oplus \delta_1)*E_0
\end{tikzcd}
\]
Each of the five triangles in the above diagram commutes by construction, and the vertical arrows in the right column are isomorphisms. We claim that the square commutes up to a non-zero scalar when precomposed with $\alpha'$ (that is, the upper route and the lower one along the diagram give the same result up to a non-zero scalar).
Indeed, the space of maps from $E_0$ to $(\phi_{\SO(V')}(U')\oplus \delta_1)*\phi_{\Sp(V)}(U)*E_0$ that factor through the $\delta_1*\phi_{\Sp(V)}(U)*E_0$ direct summand in the target is one dimensional. Both routes factor through this summand, and both are non-zero.

Letting as before $$\sigma=((\phi_{\SO(V')}(U')\oplus \delta_1)*\nu)\circ (\tau\circ\phi_{\Sp(V)}(U)*\nu),$$ we get that $\sigma\circ a = s a'$ for some $s\in \CC^\times$.
The same type of argument shows that $b\circ \sigma^{-1}$ is equal to a scalar multiple of $b'$, and by compatibility with $a,a'$ we must have $b\circ \sigma^{-1} = s^{-1} b'$.
This shows that the decomposition $$\Ext^0(E_0, (\phi_{\SO(V')}(U')\oplus \delta_1)*(\phi_{\SO(V')}(U')\oplus \delta_1)*E_0)\cong \CC a \oplus \Ext^0(E_0, \ker(b))$$ is compatible under $\sigma$ with the respective decomposition $$\Ext^0(E_0, \phi_{\Sp(V)}(U)*\phi_{\Sp(V)}(U)*E_0)\cong \CC a' \oplus \Ext^0(E_0, \ker(b')).$$

Let us proceed to step 2 of the above plan. Write
$$\Ext^0(E_0, (\phi_{\SO(V')}(U')\oplus \delta_1)*(\phi_{\SO(V')}(U')\oplus \delta_1)*E_0)= \Ext^0((\phi_{\SO(V')}(U')\oplus \delta_1)*E_0, (\phi_{\SO(V')}(U')\oplus \delta_1)*E_0),$$
$$\Ext^0(E_0, \phi_{\Sp(V)}(U)*\phi_{\Sp(V)}(U)*E_0) = \Ext^0(\phi_{\Sp(V)}(U)*E_0, \phi_{\Sp(V)}(U)*E_0).$$
\begin{claim}
The identity map in the right-hand side of either of the above equations corresponds to an element of the left-hand side whose image is contained neither in the image of $a$ (resp. $a'$) nor in the kernel of $b$ (resp. $b'$).
\end{claim}
\begin{proof}
    For the first line above the claim is clear. Let us prove it for the second line. If the image of the aforementioned map were contained in the image of $a$, then it would mean that it would be contained in the image of the first row in the second row of the following diagram:
    \[
    \begin{tikzcd}
        \Ext^0(E_0, \phi_{\Sp(V)}(U)*E_0) \ar[r, equal]\ar[d, "(\phi_{\Sp(V)}(U)*\alpha')\circ-"] &\Ext^0(\phi_{\Sp(V)}(U)*E_0, E_0)\ar[d, "\alpha'\circ-"]\\
        \Ext^0(E_0, \phi_{\Sp(V)}(U)*\phi_{\Sp(V)}(U)*E_0) \ar[r, equal] &\Ext^0(\phi_{\Sp(V)}(U)*E_0, \phi_{\Sp(V)}(U)*E_0)
    \end{tikzcd}
    \]
    which is clearly not the case.
    Similarly, if the image were contained in the kernel of $b$, then it would mean that our elements vanishes under the map
    \[
    \begin{tikzcd}
        \Ext^0(E_0, \phi_{\Sp(V)}(U)*\phi_{\Sp(V)}(U)*E_0) \ar[r, equal]\ar[d, "(\phi_{\Sp(V)}(U)*\beta')\circ-"] &\Ext^0(\phi_{\Sp(V)}(U)*E_0, \phi_{\Sp(V)}(U)*E_0)\ar[d, "\beta'\circ-"]\\
        \Ext^0(E_0, \phi_{\Sp(V)}(U)*E_0) \ar[r, equal] &\Ext^0(\phi_{\Sp(V)}(U)*E_0, E_0)
    \end{tikzcd}
    \]
    but again this is not the case.
\end{proof}
\begin{proof}[Proof of Proposition \ref{prop: Hecke relation identification}]
    Consider the automorphism of $\phi_{\SO(V')}(U')\oplus \delta_1$ which rescales the two summands by $s,s'$ respectively, and let $j$ be the resulting automorphism of $(\phi_{\SO(V')}(U')\oplus \delta_1)*E_0$.
    The resulting automorphism of $\Ext^0(E_0, (\phi_{\SO(V')}(U')\oplus \delta_1)*(\phi_{\SO(V')}(U')\oplus \delta_1)*E_0)$ preserves the constructed decomposition into direct summands, and rescales them by $s^2,s'^2$ respectively.
    This automorphism can also be written as $((\phi_{\SO(V')}(U')\oplus \delta_1)*j)\circ \tau \circ ((\phi_{\SO(V')}(U')\oplus \delta_1)*j)$.
    Moreover, $\sigma_{j\circ \nu}$ is equal to the composition of $\sigma_\nu$ with the above automorphism.

    Hence, if we choose suitable scalars $s,s'$ as above and replace $\nu$ by $j\circ \nu$, we can make the diagram of Proposition \ref{prop: Hecke relation identification} commute.

\end{proof}
\begin{lemma}\label{lemma: Rep(SO) universal property}
    Let $D$ be an abelian $\CC$-linear symmetric monoidal (with unit) category, such that the endomorphism algebra of the unit is $\CC$.
    Let $\Phi_1, \Phi_2$ be two $\CC$-linear symmetric monoidal functors from $\mathrm{Rep}(\SO(U))$ to $D$, and let $\nu_U:\Phi_1(U)\xrightarrow{\cong} \Phi_2(U)$ be an isomorphism compatible with the self duality of $U$, then there is a unique isomorphism of symmetric monoidal functors $\nu:\Phi_1\xrightarrow{\cong} \Phi_2$ whose value at $U$ is $\nu_U$.
\end{lemma}
\begin{proof}
    Since $U$ is odd-dimensional, we have $\operatorname{O}(U)\cong \SO(U)\times \{\pm 1\}$, and so we have the restriction functors $\mathrm{Rep}(\SO(U))\to \mathrm{Rep}(\operatorname{O}(U))\to \mathrm{Rep}(\SO(U))$ whose composition is the identity. It follows that it suffices to show the same statement with $\operatorname{O}(U)$ instead of $\SO(U)$. This follows from \cite[Theorems 9.4, 9.6]{Del}. Indeed, by Theorem 9.4 of loc. cit., the isomorphism of self-dual objects $\nu_U$ extends uniquely to an isomorphism of the compositions of $\Phi_1$ and $\Phi_2$ with the natural functor $$\mathrm{Rep}(\operatorname{O}(2n+1),\CC)\to \mathrm{Rep}(\operatorname{O}(U)),$$ where $\mathrm{Rep}(\operatorname{O}(2n+1),\CC)$ is the category defined in section 9.3 of loc. cit.
    By Theorem 9.6 of loc. cit., a functor from $\mathrm{Rep}(\operatorname{O}(2n+1),\CC)$ factors uniquely through $\mathrm{Rep}(\operatorname{O}(U))$ (if it does factor at all). Thus, we get the desired isomorphism of functors $\nu$.
\end{proof}

Recall that $\mathcal{C}$ is equipped with the fusion monoidal structure, defined by taking nearby cycles along the embedding of the affine Grassmannian into the Beilinson-Drinfeld Grassmannian.
The unit for this monoidal structure is $E_0$, and it is a standard fact that the functors $\phi_{\Sp(V)}(-)*E_0$ and $\phi_{\SO(V')}(-)*E_0$ from $\mathrm{Rep}(\SO(U))$ and $\mathrm{Rep}(\SO(U'))$ respectively are symmetric monoidal with respect to fusion, which follows from descriptions using fusion of the Hecke action and of the convolution monoidal structure on $\operatorname{D-mod}_{G_\mathcal{O}}(\Gr_G)$ (where $G$ is either $\Sp(V)$ or $\SO(V')$).

We are now ready to prove the main result of this section, whose parallel result for $\ell$-adic sheaves is proved in \cite{Lys}.
\begin{proposition}\label{prop: Hecke identification}
    There is an isomorphism of symmetric monoidal functors from $\mathrm{Rep}(\SO(U))$ to $\mathcal{C}$:
    $$\nu: \phi_{\Sp(V)}(-)*E_0\cong \phi_{\SO(V')}(-|_{\SO(U')})*E_0.$$
\end{proposition}
\begin{proof}
    In order to apply Lemma \ref{lemma: Rep(SO) universal property}, we need to show that the symmetric monoidal structure preserves $\mathcal{C}^\heartsuit$ and that the functors $\phi_{\Sp(V)}(-)*E_0$ and $\phi_{\SO(V')}(-)*E_0$ take values in $\mathcal{C}^\heartsuit$.
    
    The first property is satisfied because nearby cycles are t-exact.
    The fact that the images of these functors are contained in $\mathcal{C}^\heartsuit$ follows from the fact that $\phi_{\SO(V')}(U')*E_0, \phi_{\Sp(V)}*E_0\in \mathcal{C}^\heartsuit$ (by the results of Section \ref{sec: matching basic reps}), the fact that any object in $\mathrm{Rep}(\SO(U))$ (resp. $\mathrm{Rep}(\SO(U'))$) is a direct summand of a direct sum of tensor powers of $U$ (resp. $U'$), and the compatibility of the functors with the symmetric monoidal structure.

    We can now apply Lemma \ref{lemma: Rep(SO) universal property} to get the desired isomorphism of functors $\nu$ from Proposition \ref{prop: U Hecke Identification} and Proposition \ref{prop: Hecke relation identification}.
\end{proof}
\section{Classification of irreducible objects and a generation result}\label{sec: generation}
In this section we will give a sketch of the proof of the Proposition \ref{prop:Hecke generation} below, which is proved in the same way as \cite[Corollary 3.2.3]{BFKT}, as well as a description of the irreducible objects in $\mathcal{C}^\heartsuit$. The material in this section is not new, and is included for completeness.
The classification of relevant orbits given below is also proved in a different way in \cite[§3]{Fu}. Note that the current version of loc. cit. contains a typo, stating that relevant orbits are classified by coweights of $\SO(V')$ whereas they are classified by dominant coweights (and this is indeed what is proved there).
\begin{definition}\label{def: generation}
    Given a stable category $\mathcal{D}$, and a set $T$ of objects in it, we let $T^\perp$ be the full subcategory of $\mathcal{D}$ consisting of objects $M$ such that $\Hom_{\mathcal{D}}(x, M) = 0$ for all $x\in T$.
    We define the subcategory generated by $T$ to be the full stable subcategory of $\mathcal{D}$ consisting of objects $y$ such that $\Hom_{\mathcal{D}}(y, M) = 0$ for all $M$ in $T^\perp$.
\end{definition}
Note that it contains $T$, and that it is closed under colimits and retracts.
Moreover, $\mathcal{D}$ being generated by $T$ is equivalent to the functor $\mathcal{D}\to \mathrm{Spc}^T$ sending $x$ to $\mathrm{Hom}(t, x)$ for each $t\in T$ is conservative. 
\begin{proposition}\label{prop:Hecke generation}
    The category $\mathcal{C}$ is generated by $E_0$ under the Hecke action of $\mathrm{Rep}(\SO(U'))$.
    That is, the full subcategory of $\mathcal{C}$ generated by the essential image of $\mathrm{Rep}(\SO(U'))$ under the functor of action on $E_0$ is equal to $\mathcal{C}$.
\end{proposition}
\subsection{Description of the category $\mathcal{C}$ via twisted D-modules}\label{subsec: twisted D-modules}
For this subsection we will work in the setup of $M$ a finite dimensional symplectic vector space (in our case, $V\otimes V'$) and $G\subset \Sp(M)$ an algebraic subgroup (in our case $\Sp(V)\times \SO(V')$).

Let $\mathrm{Heis}$ be the group scheme whose underlying scheme is $M_\mathcal{K}\times \GG_a$, and whose multiplication is given by $(v_1, a_1)\cdot (v_2, a_2):=(v_1+v_2, a_1+a_2+\langle v_1, v_2 \rangle)$. It is a central extension of $M_\mathcal{K}$ by $\GG_a$. Note that given a Lagrangian subspace $L$ of $M_\mathcal{K}$, such as $M_\O$, the subscheme $L\times \GG_a\subset \mathrm{Heis}$ is a commutative subgroup.
Let also $\chi$ be the exponential D-module on $\GG_a$.
The following theorem is a version of \cite[Lemma 2.2.1]{BFT2}.
\begin{theorem}
    There is an equivalence of categories
    $$\operatorname{D-mod}_{M_\O \times \GG_a, \chi}(\mathrm{Heis})\xrightarrow{\sim}\operatorname{\mathcal{W}-\mathrm{mod}},$$
    Compatible with the $\Sp(M)_\O$ action on both sides, as well as with the t-structure.
    In particular, we have
    $$\operatorname{D-mod}_{G_\mathcal{O}\ltimes M_\O \times \GG_a, \chi}(\mathrm{Heis})\xrightarrow{\sim}\operatorname{\mathcal{W}-\mathrm{mod}}^{G_\O}.$$
\end{theorem}

\begin{remark}\label{remark: twisted D-modules}
The above categories are equivalent to $\operatorname{D-mod}_{\GG_a, \chi}(\mathrm{Heis} / M_\O)$ and to $\operatorname{D-mod}_{G_\mathcal{O} \times \GG_a, \chi}(\mathrm{Heis} / M_\O)$. As $\mathrm{Heis} / M_\O$ is a line bundle over $M_\mathcal{K} / M_\O$, this category may be viewed as twisted ($G_\mathcal{O}$-equivariant) D-modules on $M_\mathcal{K} / M_\O$, and we denote it by $\operatorname{D-mod}_\chi(M_\mathcal{K}/M_\O)$ (resp. $\operatorname{D-mod}_{G_\O,\chi}(M_\mathcal{K}/M_\O)$). Note that by choosing a transversal Lagrangian to $M_\O$ in $M_\mathcal{K}$ we get an equivalence $$\operatorname{D-mod}_\chi(M_\mathcal{K}/M_\O)\cong \operatorname{D-mod}(M_\mathcal{K}/M_\O),$$ but this equivalence is not canonical and does not intertwine the $G_\O$-actions on both sides.
\end{remark}
In particular, any irreducible object in $\operatorname{D-mod}_{G_\mathcal{O}\ltimes M_\O \times \GG_a, \chi}(\mathrm{Heis})^\heartsuit$ is supported on the preimage of a single $G_\O$-orbit in $M_\mathcal{K} / M_\O$.
\begin{definition}\label{def:inner hom}
    There is a natural action of the category $\operatorname{D-mod}_{G_\O}(M_\mathcal{K}/M_\O)$ on $\operatorname{D-mod}_{G_\mathcal{O},\chi}(M_\mathcal{K}/M_\O)$, given by tensor product.
    This action gives rise to an inner hom functor $$\mathcal{H}om: \left(\operatorname{D-mod}_{\chi}^{G_\O}(M_\mathcal{K} / M_\O)\right)^\mathrm{op} \otimes \operatorname{D-mod}_{\chi}^{G_\O}(M_\mathcal{K} / M_\O)\to \operatorname{D-mod}_{G_\O}(M_\mathcal{K} / M_\O).$$
    Taking equivariant cohomologies of this inner hom recovers the usual ext groups in $\operatorname{D-mod}_{\chi}^{G_\O}(M_\mathcal{K} / M_\O)$.
\end{definition}

\begin{definition}
    We call a $G_\O$-orbit $\mathbb{O}$ in $M_\mathcal{K} / M_\O$ {\it relevant} if for a point $q\in\mathbb{O}$ and $\tilde{q}$ a lift of it to $M_\mathcal{K}$, for any $g\in {C}^\circ_{G_\O}(x)$ (the connected component of the identity in the centralizer of $q$ inside $G_\O$) we have
    $$\langle \tilde{q}, g\tilde{q}-\tilde{q} \rangle = 0.$$
\end{definition}
\begin{remark}
    The pairing $\langle \tilde{q}, g\tilde{q}-\tilde{q} \rangle$ is independent of the choice of a lift $\tilde{q}$ of $q$. Also, this condition is equivalent for all points of a single $G_\O$-orbit.
\end{remark}

\begin{proposition}
    The preimage in Heis of a $G_\O$-orbit in $M_\mathcal{K}/M_\O$ carries an irreducible $(G_\mathcal{O}\ltimes M_\O \times \GG_a, \chi)$-equivariant D-module if and only if it is relevant.
    Moreover, for a relevant orbit $\mathbb{O}$, the natural action of the category $\operatorname{D-mod}_{G_\O}(\mathbb{O})$ on $\operatorname{D-mod}_{G_\mathcal{O} \ltimes M_\O \times \GG_a, \chi}(\tilde{\mathbb{O}})$ gives an equivalence of categories between the two. 
\end{proposition}
\begin{proof}
    Let $\mathbb{O}\subset M_\mathcal{K}/M_\O$ be a $G_\O$-orbit, and let $\tilde{\mathbb{O}}$ be its preimage in $\mathrm{Heis}$.
    It is a $G_\mathcal{O}\ltimes M_\O \times \GG_a$-orbit, and this orbit carries a $(G_\mathcal{O}\ltimes M_\O \times \GG_a, \chi)$-equivariant D module if and only if $\chi$ is trivial on the connected component of the identity in the centralizer inside $G_\mathcal{O}\ltimes M_\O \times \GG_a$ of a point $(\tilde{q}, \alpha)\in \tilde{\mathbb{O}}\subset \mathrm{Heis}$.
    The projection to the first coordinate maps this centralizer isomorphically to the centralizer of the image $q\in M_\mathcal{K}/M_\O$ of $\tilde{q}$ inside $G_\O$, with the inverse map being $g\mapsto (g, \tilde{q}-g\tilde{q}, \langle \tilde{q}, g\tilde{q}-\tilde{q}\rangle)$. Thus, we see that the above condition is equivalent to $\mathbb{O}$ being relevant.

    The second part is clear, as both are in correspondence with representations of the connected component groups of the respective centralizers, and these two centralizers map isomorphically to each other.
\end{proof}

Consider the moment map $M_\mathcal{K}\xrightarrow{\mu} \mathfrak{g}_\O^*$, given by $\mu(m)(\xi) = \langle m, \xi m\rangle$ for $\xi\in \mathfrak{g}_\O$.
\begin{lemma}\label{lemma: mu 0}
    The 0 level $\mu^{-1}(0)$ is contained in the preimage in $M_\mathcal{K}$ of the union of relevant orbits in $M_\mathcal{K}/M_\mathcal{O}$.
    For each relevant orbit $\mathbb{O}$, the intersection of $\mu^{-1}(0)$ with its preimage is a torsor for $T^*_\mathbb{O}(M_\mathcal{K}/M_\O)$ (i.e. it is a ``twisted conormal bundle"). More precisely, the fiber over each point in $\mathbb{O}$ is an affine subspace in $M_\mathcal{O}$ which is a translation of the conormal space to $\mathbb{O}$ inside $M_\mathcal{K}$ at the point, via the identification $M_\mathcal{O}\cong (M_\mathcal{K} / M_\O)^*$.
\end{lemma}
\begin{proof}
    The condition for $\tilde{q}\in M_\mathcal{K}$ to belong to $\mu^{-1}(0)$ is $\langle \tilde{q}, \xi \tilde{q}\rangle = 0$ for every $\xi\in \mathfrak{g}_\O$. Denoting by $q$ the image of $\tilde{q}$ in $M_\mathcal{K} / M_\O$, the above condition for $\xi$ in the Lie algebra of the stabilizer of $q$ is exactly for $\mathbb{O}$ to be relevant, proving the first part of the Lemma.
    Now let $q\in\mathbb{O}$ be a point in a relevant orbit. Choose a linear subspace $S$ of $\mathfrak{g}_\O$ mapping isomorphically to $T_q\mathbb{O}$. By virtue of $\mathbb{O}$ being relevant, the condition for a point $\tilde{q}$ over $q$ to be in $\mu^{-1}(0)$ is that for any $\xi\in S$ we have $\langle \tilde{q}, \xi \tilde{q}\rangle = 0$. Choosing a lift $\tilde{q}_0\in M_\mathcal{K}$ of $q$ (not necessarily in $\mu^{-1}(0)$), the above reads
    $$\begin{aligned}
        0 &= \langle \tilde{q}, \xi \tilde{q}\rangle = \langle \tilde{q} - \tilde{q}_0, \xi \tilde{q}\rangle + \langle \tilde{q}_0, \xi (\tilde{q} - \tilde{q}_0)\rangle + \langle \tilde{q}_0, \xi \tilde{q}_0\rangle\\
        &= 2\langle \tilde{q} - \tilde{q}_0, \xi \tilde{q}_0\rangle + \langle \tilde{q}_0, \xi \tilde{q}_0\rangle.
    \end{aligned}$$
    This is indeed the equation of a translation of $(T^*_\mathbb{O}(M_\mathcal{K}/M_\O))_q = T_q\mathbb{O}^\perp\subseteq M_\O$ (note that the first term in the last equation only depends on the image of $\xi$ in $T_q\mathbb{O}$ and not on $\xi$ itself).
\end{proof}

\subsection{Classification of irreducible objects}
Let $M \leq N$ positive integers.
The following Lemma is proved in \cite[Lemmas 3.2.1, 3.2.2]{BFKT}:
\begin{lemma} \label{lemma: GL orbits}
    \begin{enumerate}
        \item The $\GL(M)_\O\times\GL(N)_\O$-orbits in $\mathrm{Mat}(M\times N, \mathcal{K})$ invariant under translations by a sublattice are indexed by the set of length $M$ signatures. An orbit $\mathbb{O}_{\boldsymbol{\lambda}}$ has a representative $(t^{-{\boldsymbol{\lambda}}}, 0)$ (the last $N-M$ columns are all zero).
        \item Given a signature ${\boldsymbol{\lambda}}$, we set $\IC_{\boldsymbol{\lambda}}:=\IC_{\mathbb{O}_{\boldsymbol{\lambda}}}$. Moreover, let $\boldsymbol{\lambda}^*$ denote the signature given by $(-\lambda_M\geq \dots\geq -\lambda_1)$, where $\boldsymbol{\lambda}=(\lambda_1\geq\dots\geq\lambda_M)$. Then under the Fourier transform equivalence, $$\mathrm{FT}:\operatorname{D-mod}(\mathrm{Mat}(M\times N, \mathcal{K}))^{\GL(M)_\O\times \GL(N)_\O}\xrightarrow{\sim} \operatorname{D-mod}(\mathrm{Mat}(M\times N, \mathcal{K}))^{\GL(M)_\O\times \GL(N)_\O},$$ we have $$\mathrm{FT}(\IC_{\boldsymbol{\lambda}}) = \IC_{\boldsymbol{\lambda}^*}.$$
    \end{enumerate}
\end{lemma}

From now on, our argument will follow that of \cite[§2.6]{BFT2}
We shall make use of the above Lemma for $M=N=2n$. Let $\sigma$ be the involution of $\GL(V)\times \GL(V')$ whose fixed point set is $\Sp(V)\times \SO(V')$. We will denote by the same letter $\sigma$ the involution of $T^*(V\otimes V')\cong V\otimes V' \oplus (V\otimes V')^*$, permuting the two summands using their isomorphism coming from the symplectic form. This involution preserves the natural symplectic form on $T^*(V\otimes V')$ (the one in which $V\otimes V'$ and $(V\otimes V')^*$ are Lagrangians).
These two involutions are compatible with the action of $\GL(V)\times \GL(V')$ on $V\otimes V'$, as well as with the moment map $$\mu_{\GL,0}: T^*(V\otimes V')\to \mathfrak{gl}(V)^* \times \mathfrak{gl}(V')^*.$$
Now take the (positive) loop of the above, considering the action of $\GL(V)_\O\times \GL(V')_\O$ on $(T^*(V\otimes V'))_{\mathcal{K}}$, preserving the symplectic form given by residue of the symplectic form defined above, and with the corresponding moment map $\mu_{\GL}$. We denote by $\sigma$ the induced involution.

Taking $\sigma$-fixed points yields the natural action of $\Sp(V)_\O\times \SO(V')_\O$ on $(V\otimes V')_{\mathcal{K}}$, with moment map $\mu$ being the one given in \ref{subsec: twisted D-modules}.

By Lemma \ref{lemma: GL orbits}, the set of $\GL(V)_\O\times \GL(V')_\O$ orbits on $(V\otimes V')_\mathcal{K}$ is discrete. From this, it follows that the irreducible components of $\mu_{\GL}^{-1}(0)$ are the closures of conormal bundles to $\GL(V)_\O\times \GL(V')_\O$-orbits on $(V\otimes V')_\mathcal{K}$, which are indexed by signatures of length $2n$. We will denote the component corresponding to $\boldsymbol{\lambda}$ by $\Lambda_{\GL}^{\boldsymbol{\lambda}}$, with $\mathring{\Lambda}_{\GL}^{\boldsymbol{\lambda}}$ denoting the open part of it which is the conormal orbit itself.
From the second part of the lemma it follows that $\sigma$ takes $\Lambda_{\GL}^{\boldsymbol{\lambda}}$ to $\Lambda_{\GL}^{\boldsymbol{\lambda^*}}$.
Since $\Lambda_{\GL}^{\boldsymbol{\lambda}}$ is a linear subspace of $T^*(V\otimes V')_\mathcal{K}$, so is its intersection with the set of $\sigma$-fixed points, and in particular this intersection is irreducible. We denote it by $\Lambda^{\boldsymbol{\lambda}}$.

\begin{lemma} \label{lem: mu 0 lagrangian}
    The zero level $\mu^{-1}(0)\subset (V\otimes V')_\mathcal{K}$ is Lagrangian.
\end{lemma}
\begin{proof}
    By the description above, it follows that $\mu_{\GL}^{-1}(0)$ is Lagrangian. As the inclusion of $V\otimes V'$ into $T^*(V\otimes V')$ as the $\sigma$-fixed points is compatible with the symplectic forms, we deduce that $\mu^{-1}(0)$ is isotropic, as it is contained in $\mu_{\GL}^{-1}(0)$. However, from the description in Lemma \ref{lemma: mu 0} of $\mu^{-1}(0)$ as the union of twisted conormal bundles to relevant orbits, we get that it must be coisotropic.
\end{proof}
\begin{corollary}
    The set of relevant orbits in $(V\otimes V')_\mathcal{K} / (V\otimes V')_\O$ is discrete.
\end{corollary}
We make use of the following elementary linear algebra lemma:
\begin{lemma}\label{lemma: Lagrangians and involutions}
    Let $\Lambda\subset M$ be a Lagrangian subspace of a symplectic vector space $M$. Let $\sigma$ be an involution of $M$ preserving the symplectic form.
    Then $\sigma(\Lambda)=\Lambda$ if and only if $\Lambda^\sigma\subset M^\sigma$ is Lagrangian.
\end{lemma}

\begin{proposition}
The set of irreducible components of $\mu^{-1}(0)$ is equal to the set of $\Lambda^{\boldsymbol{\lambda}}$ for self-dual signatures $\boldsymbol{\lambda}$.
\end{proposition}
\begin{proof}
    By Lemma \ref{lemma: Lagrangians and involutions}, if a signature $\boldsymbol{\lambda}$ is self-dual then $\Lambda^{\boldsymbol{\lambda}}$ is Lagrangian, and particular it is an irreducible component of $\mu^{-1}(0)$ (note that as discussed above, it is indeed irreducible). Moreover, each irreducible component of $\mu^{-1}(0)$ is Lagrangian, and must be contained in $\Lambda^{\boldsymbol{\lambda}}$ for some $\boldsymbol{\lambda}$ (not necessarily self dual), due to the description of the irreducible components of $\mu_{\GL}^{-1}(0)$. It follows that $\Lambda^{\boldsymbol{\lambda}}$ is an isotropic linear subspace of $(V\otimes V')_\mathcal{K}$ containing a Lagrangian variety, hence itself must be Lagrangian. By Lemma \ref{lemma: Lagrangians and involutions}, this implies that $\boldsymbol{\lambda}$ is self-dual.
\end{proof}
\begin{remark}
    It follows that the set of relevant orbits in $(V\otimes V')_\mathcal{K}/(V\otimes V')_\O$ is also indexed by self-dual signatures, and the closure of the orbit corresponding to a self-dual signature $\boldsymbol{\lambda}$ is the image of $\Lambda^{\boldsymbol{\lambda}}$ in $(V\otimes V')_\mathcal{K}/(V\otimes V')_\O$.
\end{remark}
\begin{lemma}
    Let $\boldsymbol{\lambda} = (\boldsymbol{\theta}, \boldsymbol{\theta}^*)$ be a self dual signature (so that $\boldsymbol{\theta}$ is a non-negative signature of length $n$). Then the stabilizer in $\Sp(V)_\O\times \SO(V')_\O$ of a point the image of $\Lambda^{\boldsymbol{\lambda}}$ in $(V\otimes V')_\mathcal{K} / (V\otimes V')_\mathcal{O}$ is connected.
\end{lemma}
\begin{proof}
    The proof of \cite[Corollary 3.2.3]{BFKT} applies verbatim.
\end{proof}
\begin{corollary}
    The irreducible objects in $\mathcal{C}$ are classified by their support in $(V\otimes V')_\mathcal{K} / (V\otimes V')_\O$, which is the closure of a $\Sp(V)_\O\times \SO(V')_\O$-orbit corresponding to a self-dual signature as above.
\end{corollary}
\begin{proof}[Proof of Proposition \ref{prop:Hecke generation}]
    Self dual signatures of length $2n$ also classify irreducible representations of $\SO(U')$ (call these representations $\tau_{\boldsymbol{\lambda}}$), and if $\boldsymbol{\lambda}$ is such a self dual signature, then the support of $\phi_{\SO(V')}(\tau_{\boldsymbol{\lambda}})$ is equal to the closure of the orbit corresponding to $\boldsymbol{\lambda}$.
    
    Let $\Delta_{\boldsymbol{\lambda}}$ be the standard object in $\mathcal{C}$ for the orbit corresponding to $\boldsymbol{\lambda}$, i.e. the $!$-extension of the unique irreducible equivariant twisted D-module on this orbit. We prove by induction on the closure order of orbits that $\Delta_{\boldsymbol{\lambda}}$ belongs to the full stable subcategory generated by $E_0$ under the Hecke action of $\mathrm{Rep}(\SO(U'))$. Call this subcategory $\mathcal{C}_0$.

    Letting $j$ be the inclusion of the orbit associated to $\boldsymbol{\lambda}$ into its closure, and $i$ the closed embedding of the complement, we have an exact triangle
    $$j_! j^! \phi_{\SO(V')}(\tau_{\boldsymbol{\lambda}}) \to \phi_{\SO(V')}(\tau_{\boldsymbol{\lambda}}) \to i_* i^* \phi_{\SO(V')}(\tau_{\boldsymbol{\lambda}})\to$$
    Since the right term is supported on smaller orbits, it belongs to $\mathcal{C}_0$, hence so does the left term.
    Moreover, $j^! \phi_{\SO(V')}(\tau_{\boldsymbol{\lambda}})$ is a non-zero direct sum of cohomologically shifted copies of the unique equivariant twisted D-module on the orbit, thus the left term is isomorphic to a non-zero direct sum of cohomologically shifted copies of $\Delta_{\boldsymbol{\lambda}}$. We deduce that $\Delta_{\boldsymbol{\lambda}}$ belongs to $\mathcal{C}_0$, as desired.

    The proposition now follows by a standard devissage argument.
\end{proof}

\section{A Deequivariantized endomorphism algebra}\label{sec: deeq}
\begin{definition}\label{def: deeq}
    For $A,B\in \mathcal{C}$, set
    $$\mathrm{Hom}_{\mathrm{deeq}}(A,B):=\mathrm{Hom}_{\mathrm{Ind}(\mathcal{C})}(A, \phi_{\SO(V')}(\CC[SO(U')])*B).$$
Here for the convolution action, $\CC[\SO(U')]$ is considered as an (ind-)representation of $\SO(U')$ via the left action. This Hom space acquires a natural action of $\SO(U')$.
\end{definition}
We set $\mathscr{E}:=\mathrm{End}_{\mathrm{deeq}}(E_0)$ and $\mathscr{G}:=\Sym(\mathfrak{so}(U)[-2])$.
These are dg-algebras carrying an $\SO(U')$-action (resp. $\SO(U)$-action).
We also let $\mathfrak{E}^\bullet:=H^\bullet(\mathscr{E})$ and $\mathfrak{G}^\bullet:=H^\bullet(\mathscr{G})$, considered as classical graded superalgebras.

Let us construct an $\SO(U')$-equivariant map $\psi:\mathscr{G}\to \mathscr{E}$.
Using the derived Satake equivalence, we get an $\SO(U)$-equivariant isomorphism
$$\Hom_{\operatorname{D-mod}_{\Sp(V)_\O}(\Gr_{\Sp(V)})}(\delta_1, \phi_{\Sp(V)}(\CC[\SO(U)]))\cong \Sym^\bullet(\mathfrak{so}(U)[-2])=\mathscr{G}.$$
This gives a natural $\SO(U)$-equivariant map
$$\mathscr{G}\to \mathrm{Hom}_{\mathrm{Ind}(\mathcal{C})}(E_0, \phi_{\Sp(V)}(\CC[\SO(U)])* E_0)\cong \mathrm{Hom}_{\mathrm{Ind}(\mathcal{C})}(E_0, \phi_{\SO(V')}(\CC[\SO(U)]|_{\SO(U')})* E_0).$$
Composing this with $\phi(-)*E_0$ applied to the restriction of polynomials map $$\CC[SO(U)]\to \CC[\SO(U')],$$ which is $\SO(U')\times \SO(U')$-equivariant, we get the desired map $\psi:\mathscr{G}\to \mathscr{E}$. We denote by the same notation the resulting map on cohomologies $\psi:\mathfrak{G}^\bullet\to \mathfrak{E}^\bullet$.

We set also $\psi':\Sym(\mathfrak{so}(U')[-2])\cong \Hom_{\operatorname{D-mod}_{\SO(V')_\O}(\Gr_{\SO(V')})}(\delta_1, \phi_{\SO(V')}(\CC[\SO(U')]))\xrightarrow{*E_0} \mathscr{E}$ to be the map given by convolution with $E_0$, compatible with the action of $\SO(U')$.
Let $\varpi_0:\mathfrak{so}(U')\to \mathfrak{so}(U)$ be the negative of the natural inclusion (i.e. $x\mapsto -x$). We will denote in the same way the induced map $\Sym(\mathfrak{so}(U')[-2])\to \Sym(\mathfrak{so}(U)[-2])$
\begin{proposition}\label{prop: psi psi' identification}
    $\psi'\cong\psi\circ \varpi_0$.
\end{proposition}
\begin{proof}
    Consider the invariant symmetric bilinear form induced on $\mathfrak{so}(U')$ from the embedding in $\mathrm{End}(U')$, which carries the natural trace form (compatible under the isomorphism $\mathrm{End}(U')\cong U'\otimes (U')^*$ with the obvious self duality of the second). This bilinear form is non-degenerate, and we denote its inverse by $B_{U'}$. It is a nondegenerate invariant bilinear form on $\mathfrak{so}(U')^*$.
    Similarly, one has the invariant symmetric bilinear form $B_U$ on $\mathfrak{so}(U)^*$ constructed in the same way.
    
    These two forms are considered in Appendix \ref{appendix: det line bundle} where certain line bundles $\mathcal{D}_{\Sp(V)}$ and $\mathcal{D}_{\SO(V')}$ on $\Gr_{\Sp(V)}$ and $\Gr_{\SO(V')}$ are constructed, and it is shown that for any representation $\rho$ of $\SO(U')$ (resp. $\SO(U)$), the endomorphism of $\rho\otimes \mathcal{O}_{\mathfrak{so}(U')^*[2]}$ (resp. $\rho\otimes \mathcal{O}_{\mathfrak{so}(U)^*[2]}$) given by $4B_{U'}$ under the map $(\mathfrak{so}(U')\otimes \mathfrak{so}(U'))^{\SO(U')}\to \mathrm{End}(\rho)\otimes \Sym(\mathfrak{so}(U')[-2]) [2]$ (resp. $4B_U$ under a similar map) corresponds under derived Satake to cup product with the equivariant first Chern class of $\mathcal{D}_{\SO(V')}$ (resp. $\mathcal{D}_{\Sp(V)}$).
    
    In order to prove Proposition \ref{prop: psi psi' identification}, it is enough to prove the isomorphism in question for the linear generators of $H^2(\Sym(\mathfrak{so}(U')[-2])) = \mathfrak{so}(U')$, which is what we will do.
    The image $\psi'(\mathfrak{so}(U'))\subseteq \Ext^2_{\operatorname{D-mod}_{\SO(V')_\O}(\Gr_{\SO(V')})}(E_0, \phi_{\SO(V')}(\CC[\SO(U')])*E_0)$ is contained in the direct summand
    $$\Ext^2_{\operatorname{D-mod}_{\SO(V')_\O}(\Gr_{\SO(V')})}(E_0, \phi_{\SO(V')}(\mathfrak{so}(U')^*)*E_0)\otimes \mathfrak{so}(U'),$$
    and is given by tensoring with the action on $E_0$ of the canonical element in $$\Ext^2_{\operatorname{D-mod}_{\SO(V')_\O}(\Gr_{\SO(V')})}(\delta_1, \phi_{\SO(V')}(\mathfrak{so}(U')^*)),$$ given under derived Satake by the image of 1 under the composition $$\CC\to (\mathfrak{so}(U')\otimes \mathfrak{so}(U')^*)^{\SO(U')}\to \Ext^2_{\mathrm{QCoh}^{[]}_{\SO(U')}(\mathfrak{so}(U')^*)}(\O_{\mathfrak{so}(U')^*[2]}, \mathfrak{so}(U')^*\otimes \O_{\mathfrak{so}(U')^*[2]}).$$
    Using the bilinear form $4B_{U'}$ and derived Satake, the above $\Ext^2$ group is viewed a direct summand of $$\begin{aligned}
        &\Ext^2_{\operatorname{D-mod}_{\SO(V')_\O}(\Gr_{\SO(V')})}(\delta_1, \phi_{\SO(V')}(U'\otimes (U')^*))\\\cong& \Ext^2_{\operatorname{D-mod}_{\SO(V')_\O}(\Gr_{\SO(V')})}(\phi_{\SO(V')}(U'), \phi_{\SO(V')}(U'))
    \end{aligned}.$$
    We have the following commutative diagram:
    $$\begin{tikzcd}
    \CC\ar[r]\ar[rd, "4B_{U'}"']& (\mathfrak{so}(U')\otimes \mathfrak{so}(U')^*)^{\SO(U')}\ar[r]\ar[d, "4B_{U'}"]& \Hom(\O_{\mathfrak{so}(U')^*}, \mathfrak{so}(U')^*\otimes \O_{\mathfrak{so}(U')^*})^{\SO(U')}\ar[d, "4B_{U'}"]\\
    &(\mathfrak{so}(U')\otimes \mathfrak{so}(U'))^{\SO(U')}\ar[r]& \Hom(\O_{\mathfrak{so}(U')^*}, \mathfrak{so}(U')\otimes \O_{\mathfrak{so}(U')^*})^{\SO(U')}\ar[d, hook, "\oplus"]\\&&\mathrm{End}(U'\otimes \O_{\mathfrak{so}(U')^*})^{\SO(U')}\ar[d, "\cong"]\\&&\Ext^2_{\operatorname{D-mod}_{\SO(V')_\O}(\Gr_{\SO(V')})}(\phi_{\SO(V')}(U'), \phi_{\SO(V')}(U'))
    \end{tikzcd}$$
    From Corollary \ref{cor: derived Satake of canonical endomorphism} we get that the resulting element in $\Ext^2_{\operatorname{D-mod}_{\SO(V')_\O}(\Gr_{\SO(V')})}(\phi_{\SO(V')}(U'), \phi_{\SO(V')}(U'))$ is given by cap product by the equivariant first Chern class of $\mathcal{D}_{\SO(V')}$.
    Similarly, the restriction of $\psi$ to $\mathfrak{so}(U')$ is given by tensoring with a similarly constructed element in the bottom left corner of the following commutative diagram:
    $$\begin{tikzcd}
        \Ext^2_{\operatorname{D-mod}_{\Sp(V)_\O}(\Gr_{\Sp(V)})}(\delta_1, \phi_{\Sp(V)}(\mathfrak{so}(U)^*))\ar[r, hook, "4B_U"]\ar[d, "*E_0"]&\Ext^2_{\operatorname{D-mod}_{\Sp(V)_\O}(\Gr_{\Sp(V)})}(\phi_{\Sp(V)}(U), \phi_{\Sp(V)}(U))\ar[d, "*E_0"]\\
        \Ext^2_\mathcal{C}(E_0, \phi_{\Sp(V)}(\mathfrak{so}(U)^*)* E_0)\ar[r, hook, "4B_U"]\ar[d, two heads]&\Ext^2_\mathcal{C}(\phi_{\Sp(V)}(U)*E_0, \phi_{\Sp(V)}(U)*E_0) \ar[d, two heads]\\
        \Ext^2_{\mathcal{C}}(E_0, \phi_{\SO(V')}(\mathfrak{so}(U')^*) *E_0)\ar[r, hook, "4B_{U'}"]&\Ext^2_{\mathcal{C}}(\phi_{\SO(V')}(U')*E_0, \phi_{\SO(V')}(U')*E_0)
    \end{tikzcd}$$
    Here the horizontal maps are inclusion of direct summands and the bottom vertical maps are projections onto direct summands. Namely, the specified element is the image of an element in the top left corner of the above diagram, whose image in the top right corner is given by cap product by the equivariant first Chern class of $\mathcal{D}_{\Sp(V)}$.

    Note that it is in the commutativity of this diagram that our choices of $B_U$ and $B_{U'}$ matter.
    
    Recall from section \ref{sec: Hecke actions} the varieties $C, \tilde{C}_1, Q, X$, the map $\pi_1$, the D-modules $\F_\mathrm{min}, \F_Q$, and the results of Claim \ref{claim: Hecke computation} and Corollary \ref{cor: Hecke exact triangle}.
    By Remark \ref{remark: equivariant Hecke calculation recipe}, there are equivariant lifts
    $$\tilde{\F}_\mathrm{min}\in \operatorname{D-mod}_{\chi}^{\SO(V')_\O\times \Sp(V)_\O}(\overline{\Gr_{\Sp(V)}^\mathrm{min}}\times (V\otimes V')_{\mathcal{K}} / (V\otimes V')_\O),$$
    $$\tilde{\F}_Q\in \operatorname{D-mod}_{\chi}^{\SO(V')_\O\times \Sp(V)_\O}(\overline{\Gr_{\SO(V')}^\mathrm{min}}\times (V\otimes V')_{\mathcal{K}} / (V\otimes V')_\O),$$
    of $\F_\mathrm{min}, \F_Q$ respectively, which are supported on $\overline{X}, \tilde{C}_1$ respectively (and so we will consider them as twisted D-modules on $\overline{X}, \tilde{C}_1$), and such that their push forwards along the projections to $(V\otimes V')_{\mathcal{K}} / (V\otimes V')_\O$ are isomorphic to $\phi_{\Sp(V)}(U)*E_0, \phi_{\SO(V')}(U')*E_0$ respectively.
    
    We have $$\phi_{\SO(V')}(U')\cong \mathrm{IC}_Q\cong \underline{\CC}_Q[2n-2],$$ $$\phi_{\Sp(V)}(U)\cong \mathrm{IC}_{\overline{\Gr_{\Sp(V)}^\mathrm{min}}}\cong \underline{\CC}_{\overline{\Gr_{\Sp(V)}^\mathrm{min}}}[2n].$$
    The map $$\begin{aligned}
        *E_0\/:\ &\Ext^2_{\operatorname{D-mod}_{\SO(V')_\O}(\Gr_{\SO(V')})}(\phi_{\SO(V')}(U'), \phi_{\SO(V')}(U'))\to \Ext^2_\mathcal{C}(\phi_{\SO(V')}(U')*E_0, \phi_{\SO(V')}(U')* E_0)\\&\cong \Ext^2_\mathcal{C}(\pi_{1*}\tilde{\F_Q}, \pi_{1*}\tilde{\F_Q})
    \end{aligned}$$ is then given by considering the cap product with the pull back along the map $\tilde{C}_1\to Q\subseteq \Gr_{\SO(V')}$ of the appropriate equivariant cohomology class, and then pushing forward along $\pi_1$.
    The element corresponding to the map $\psi'$ is thus given by the push forward along $\pi_1$ of the multiplication by the pull back of the equivariant first Chern class $c_1(\mathcal{D}_{\SO(V')}|_Q)= c_1(\O_Q(2))$ (see Appendix \ref{appendix: det line bundle}) to $\tilde{C}_1$.

    Similarly, recall that $\Gr_{\Sp(V)}^\mathrm{min}\cong \O_{\PP(V)}(2)$. Let $\pi$ be the projection from $\overline{X}$ to $t^{-1}(V\otimes V')_\mathcal{O} / (V\otimes V')_\O$.
    The map $$\begin{aligned}
        *E_0\/:\ &\Ext^2_{\operatorname{D-mod}_{\Sp(V)_\O}(\Gr_{\Sp(V)})}(\phi_{\Sp(V)}(U), \phi_{\Sp(V)}(U))\to \Ext^2_\mathcal{C}(\phi_{\Sp(V)}(U)*E_0, \phi_{\Sp(V)}(U)* E_0)\\&\cong \Ext^2_\mathcal{C}(\pi_* \tilde{\F}_\mathrm{min}, \pi_* \tilde{\F}_\mathrm{min})
    \end{aligned}$$
    is given by considering the cap product with the pull back along the map $\overline{X}\to\overline{\Gr_{\Sp(V)}^\mathrm{min}}$ of the appropriate equivariant cohomology class, and then pushing forward to $(V\otimes V')_\mathcal{K} / (V\otimes V')_\O$.
    The element corresponding to $\psi$ is given by applying this map to the element given by cap product with $c_1\left(\mathcal{D}_{\Sp(V)}|_{\overline{\Gr_{\Sp(V)}^\mathrm{min}}}\right)$.

    Let $\mathcal{H}om$ be the inner hom functor of Definition \ref{def:inner hom}.
    We let $\widetilde{\mathrm{IC}}_C\in \mathcal{C}$ be the unique equivariant lift of $\mathrm{IC}_C$ (see the results of Section \ref{sec: generation}). We have $\phi_{\SO(V')}(U')*E_0\cong \widetilde{\mathrm{IC}}_C$.
    Letting $\iota_0:\{0\}\to C$ be the closed embedding of the origin, and $j$ the open embedding of its complement, we have an exact triangle
    $$\iota_0^* \iota_0^! \mathcal{H}om(\widetilde{\mathrm{IC}}_C, \widetilde{\mathrm{IC}}_C)\to \mathcal{H}om(\widetilde{\mathrm{IC}}_C, \widetilde{\mathrm{IC}}_C) \to j_* j^* \mathcal{H}om(\widetilde{\mathrm{IC}}_C, \widetilde{\mathrm{IC}}_C)\to$$
    Taking equivariant cohomology, we get a long exact sequence
    $$\cdots \to H^2_{\Sp(V)_\O\times \SO(V')_\O}(\iota_0^! \mathcal{H}om(\widetilde{\mathrm{IC}}_C, \widetilde{\mathrm{IC}}_C))\to \Ext^2(\widetilde{\mathrm{IC}}_C, \widetilde{\mathrm{IC}}_C)\to \Ext^2(j^*\widetilde{\mathrm{IC}}_C, j^*\widetilde{\mathrm{IC}}_C)\to \cdots$$
    Now $\iota_0^! \mathcal{H}om(\widetilde{\mathrm{IC}}_C, \widetilde{\mathrm{IC}}_C)$ is an equivariant D-module on a point, whose underlying vector space is $$\iota_0^!\mathcal{H}om(\mathrm{IC}_C, \mathrm{IC}_C) \cong \iota_0^!\mathrm{IC}_C \otimes \iota_0^!\mathrm{IC}_C\stackrel{\eqref{eq: costalk of IC at 0}}{\cong} \left(\bigoplus_{i=1}^{2n-1}\CC[-2i]\oplus \CC[-2n]\right)\otimes \left(\bigoplus_{i=1}^{2n-1}\CC[-2i]\oplus \CC[-2n]\right).$$

    Consider the spectral sequence $$H^p_{\Sp(V)_\O\times \SO(V')_\O}(H^q(\iota_0^! \mathcal{H}om(\mathrm{IC}_C, \mathrm{IC}_C)))\Rightarrow H^{p+q}_{\Sp(V)_\O\times \SO(V')_\O}(\iota_0^! \mathcal{H}om(\widetilde{\mathrm{IC}}_C, \widetilde{\mathrm{IC}}_C)).$$
    All terms in its second page for which $p<0$ or $q<4$ vanish, and so in particular $$H^2_{\Sp(V)_\O\times \SO(V)_\O}(\iota_0^! \mathcal{H}om(\widetilde{\mathrm{IC}}_C, \widetilde{\mathrm{IC}}_C))=0.$$
    As a result, we get that the map $\Ext^2(\widetilde{\mathrm{IC}}_C, \widetilde{\mathrm{IC}}_C)\to \Ext^2(j^*\widetilde{\mathrm{IC}}_C, j^*\widetilde{\mathrm{IC}}_C)$ is injective.

    Thus, it is enough to compare the restrictions to $C\setminus \{0\}$ of the two ext-elements constructed above.
    By base change, the restriction to $C \setminus \{0\}$ of the element corresponding to $\psi$ can be computed by first restricting from $\overline{X}$ to the preimage of $(V\otimes V')_\mathcal{K} / (V\otimes V')_\O \setminus \{0\}$ in $X$, and then pushing forward to $(V\otimes V')_\mathcal{K} / (V\otimes V')_\O \setminus \{0\}$ (we know the result to be supported on $C\setminus \{0\}$).
    Recall that by the computations performed in Appendix \ref{appendix: det line bundle}, the restriction of $\mathcal{D}_{\Sp(V)}$ to $\Gr_{\Sp(V)}^\mathrm{min}$ is isomorphic to the pullback of $\O_{\PP(V)}(2)$.
    By the projection formula, Corollary \ref{cor: Hecke exact triangle} and Claim \ref{claim: Hecke computation}, the above ext element is given by cap product with $c_1(\O_{\PP(V)}(2))$.
    Over $C\setminus \{0\}$, $\pi_1$ is an isomorphism, and so we must compare the first Chern classes of the line bundles $\O_Q(2)$ and $\O_{\PP(V)}(2)$ pulled back to $C\setminus \{0\}$. We may describe $C\setminus \{0\}$ as the complement to the zero section in the total space of the line bundle $\O_Q(-1)\otimes \O_{\PP(V)}(-1)$ over $Q\times \PP(V)$. The sum of these two equivariant Chern classes is $$c_1(\O_Q(2)) + c_1(\O_{\PP(V)}(2)) = -2c_1(\O_Q(-1)\otimes \O_{\PP(V)}(-1)) = 0,$$
    which concludes the proof of the proposition.
    The last equality holds because on the total space of $\O_Q(-1)\otimes \O_{\PP(V)}(-1)$ on $Q\times \PP(V)$, the equivariant first Chern class of the line bundle $\O_Q(-1)\otimes \O_{\PP(V)}(-1)$ is given by the equivariant characteristic class of the zero section. Restricting to the complement of the zero section, we get 0.
\end{proof}
\begin{proposition}\label{prop: Barr-Beck}
    The functor $\Hom_{\mathrm{deeq}}(E_0, -)$ induces an equivalence of categories $$\mathcal{C}\xrightarrow{\sim} D_\mathrm{perf}^{\SO(U')}(\mathscr{E}).$$
\end{proposition}
\begin{proof}
    We consider the adjunction (after passing to the ind-completions)
    $$\begin{tikzcd}[column sep=3.5cm]\mathrm{Ind}(\mathrm{Rep}(\SO(U')))\ar[r, "-*E_0" , shift left=1.2ex] & \mathrm{Ind}(\mathcal{C})\ar[l,"{\Hom_\mathrm{deeq}(E_0, -)}", shift left=1.2ex]\end{tikzcd}.$$
    Now by Proposition \ref{prop:Hecke generation}, the functor $\Hom_\mathrm{deeq}(E_0, -)$ is conservative.
    It also preserves geometric realizations (in fact, all colimits, as $E_0$ is compact in the ind-completion, and the Hecke action commutes with arbitrary colimits).
    Thus, by the Barr-Beck-Lurie theorem (\cite[Theorem 4.7.3.5]{Lu}), we get that $\Hom_\mathrm{deeq}(E_0, -)$ induces an equivalence of categories between $\mathrm{Ind}(\mathcal{C})$ and the category of modules over the monad $\Hom_\mathrm{deeq}(E_0, -*E_0)$ in $\mathrm{Ind}(\mathrm{Rep}(\SO(U')))$. Notice that for a representation $V$ of $\SO(U')$, we have
    $$\Hom_\mathrm{deeq}(E_0, V*E_0)\cong V\otimes \Hom_\mathrm{deeq}(E_0, E_0),$$ and so this monad is just tensoring with the algebra $\mathscr{E}$.
    Thus, we get an equivalence
    $$\mathrm{Ind}(\mathcal{C})\xrightarrow{\sim} (\mathscr{E}\mathrm{-mod})^{\SO(U')}.$$ Passing to compact objects, we get the desired equivalence of categories.
\end{proof}

Fix a pair of transversal Lagrangians $L,L^*\subset V$. The subgroup of $\Sp(V)$ preserving $L,L^*$ is a Levi subgroup isomorphic to $\GL(L)$.
\begin{definition}
    For a Lie group $G$ with a Lie algebra $\mathfrak{g}$ we set $\Sigma_{\mathfrak{g}}:=\mathfrak{g}^*\git G$.
    We also set
    \begin{enumerate}
        \item $\Sigma_{2n+1}:=\Sigma_{\mathfrak{so}(U)}$
        \item $\Sigma_{2n}:=\Sigma_{\mathfrak{so}(U')}$
        \item $\Sigma_{\mathfrak{gl}}:=\Sigma_{\mathfrak{gl}_n}$, considered with the natural map to $\Sigma_{2n+1}$.
        \item $\Sigma:=\Sigma_{2n}\times \Sigma_{\mathfrak{gl}}$, considered with the natural map to $\Sigma_{2n}\times \Sigma_{2n+1}$.
    \end{enumerate}
\end{definition}

The following theorem follows from purity considerations, as in \cite[Lemmas 3.3.1, 3.3.2, 3.3.3]{BFKT}, \cite[§2.7]{BFT2}:
\begin{theorem}\label{thm: purity}
    \begin{enumerate}
        \item The dg-algebra $\mathscr{E}$ is formal, i.e. it is quasi-isomorphic to $\mathfrak{E}^\bullet$ considered as a dg-algebra with trivial differential.
        \item $\mathfrak{E}^\bullet$ is a free module over $H^\bullet_{\SO(V')_\O\times \Sp(V)_\O}(\mathrm{pt})\cong \CC[\Sigma_{2n}\times \Sigma_{2n+1}]$.
        \item $$\tilde{\mathfrak{E}}^\bullet:=\Ext^\bullet_{\operatorname{\mathcal{W}-\mathrm{mod}}^{\SO(V')_\O\times \GL(L)_\O, \mathrm{lc}}_{\mathrm{deeq}}}(E_0, E_0)\cong \CC[\Sigma]\otimes_{\CC[\Sigma_{2n}\times \Sigma_{2n+1}]}\mathfrak{E}^\bullet.$$ In particular, it is free as a module over $\CC[\Sigma]$. Here the deequivariantized Ext (i.e. cohomologies of the deequivariantized hom) is defined in the same manner as before, again with respect to the Hecke action coming from $\SO(V')$.
    \end{enumerate}
\end{theorem}
\begin{proof}[Proof (Sketch)]
    For part 1, we first lift all encountered objects to the mixed setting of polarizable twistor D-modules of \cite{Moch1},\cite{Moch2} (this can be done as by \cite[Theorem 1.4.4]{Moch1}, all semisimple holonomic D-modules can be lifted to pure objects in this mixed setting).
    We then aim to show that the mixed lift of $\Hom_{\operatorname{\mathcal{W}-\mathrm{mod}}}(E_0, \phi_{\SO(V')}(\CC[\SO(U')])*E_0)$ is pure as a mixed sheaf on a point.
    This can be identified with $\iota_0^!(\phi_{\SO(V')}(\CC[\SO(U')])*E_0)$, where $\iota_0:\{0\}\hookrightarrow (V\otimes V')_\mathcal{K} / (V\otimes V')_\O$.
    We have used the equivalence of categories $\operatorname{\mathcal{W}-\mathrm{mod}}\cong \operatorname{D-mod}((V\otimes V')_\mathcal{K} / (V\otimes V')_\O)$ obtained by fixing $t^{-1}(V\otimes V')[t^{-1}]$ as a Lagrangian co-lattice transversal to $(V\otimes V')_\O$.
    The sheaf $\phi_{\SO(V')}(\CC[\SO(U')])*E_0$ is pure, using Corollary \ref{cor: Hecke calculation recipe}, Remark \ref{remark: equivariant Hecke calculation recipe}, and \cite[Theorem 18.1.1]{Moch1}.
    Moreover, the functor $\iota_0^!$ preserves purity, as it is a hyperbolic restriction functor using the contracting $\GG_m$-action on the vector space $(V'\otimes V)_{\mathcal{K}} / (V'\otimes V)_\O$.
    Any pure sheaf on a point is formal, in particular so is $\Hom_{\operatorname{\mathcal{W}-\mathrm{mod}}}(E_0, \phi_{\SO(V')}(\CC[\SO(U')])*E_0)$.
    Now consider the spectral sequence converging to equivariant Exts, whose second page consists of the equivariant cohomologies of the point with coefficients in non-equivariant Exts. Since all objects appearing in the second page are pure, this spectral sequence degenerates, and we get that $\mathfrak{E}^\bullet$ is also pure (when taking the grading into consideration), which implies that $\mathscr{E}$ is pure, hence formal. From the degeneration of this spectral sequence we also get Part 2 immediately. For Part 3 consider the above spectral sequence, as well with the analogous spectral sequence for $\SO(V')_\O\times \GL(L)_\O$ instead of $\SO(V')_\O\times \Sp(V)_\O$ (which also degenerates at the second page for the same reason), along with the natural map between them. Indeed, at the second page, in which both degenerate, the second spectral sequence is obtained from the first by base change along the map $\CC[\Sigma_{2n}\times \Sigma_{2n+1}]\to \CC[\Sigma]$.
\end{proof}

The rest of this paper is dedicated to the proof of the following Proposition
\begin{proposition}\label{prop: psi isom}
    The map $\psi:\mathfrak{G}^\bullet\to \mathfrak{E}^\bullet$ is an isomorphism.
\end{proposition}
\begin{proof}[Proof that Proposition \ref{prop: psi isom} implies Theorem \ref{thm: main}]
    Assuming Proposition \ref{prop: psi isom}, it follows that $\psi:\mathscr{G}\to \mathscr{E}$ is an isomorphism too. We get from Proposition \ref{prop: Barr-Beck} the desired equivalence of categories, and the required compatibility follows from \ref{prop: Hecke identification} along with \ref{prop: psi psi' identification}.
\end{proof}

\section{Invariant theory}\label{sec: invariant theory}
\subsection{The categorical quotient}\label{sec: Categorical Quotient}
We choose and fix a vector $u_0\in(U')^\perp\subset U$ of square length 1.
We also fix the Killing forms on $\mathfrak{so}(U),\mathfrak{so}(U'), \mathfrak{gl}(L)$ associated to the representations $U,U', L$ respectively. We will use these throughout this section and subsequent ones to identify the above Lie algebras with their linear duals. We will try to formulate our results in a manner that is independent of these choices, but use these identifications freely in the proofs.

There is an $\SO(U')$-equivariant isomorphism $\mathfrak{so}(U')\times U' \cong \mathfrak{so}(U)$, given by
$$(x, v) \mapsto -x-u_0v^t+vu_0^t.$$
Where $w^t:=\langle w, -\rangle$ for any $w\in U$, and $w_1w_2^t:=\langle w_2, -\rangle w_1$.
We will denote by $$\varpi: \mathfrak{so}(U')^*\times U'\to \mathfrak{so}(U)^*$$ the resulting isomorphism, where we use the Killing forms above to identify the Lie algebras with their linear duals.
By \cite[Table 3a, row 2]{Sch}, the ring of invariants $\CC[\mathfrak{so}(U')^*\times U']^{\SO(U')}$ is a polynomial ring generated by the following invariants:
\begin{itemize}
    \item The generators of $\CC[\mathfrak{so}(U')^*]^{\SO(U')}=\CC[\Sigma_{2n}]$.
    \item $(x,v)\mapsto\langle v, x^{2k} v\rangle$, for $k=0,1,\ldots, n$.
\end{itemize}
Given a $k\times k$ matrix $x$ we denote its characteristic polynomial by $$z^k + c_1(x) z^{k-1} + \ldots + c_k(x).$$

We will make use of the following formula, (see for example \cite[Theorem A.2]{M}):
\begin{proposition}\label{prop: char poly formula}
    Let $x\in \mathfrak{gl}(W)$ and let $w\in W, w^*\in W^*$. Then
    $$c_i(x+ww^*) = c_i(x) - \sum_{j=0}^{i-1} c_{i-j-1}(x) \langle w^*, x^jw \rangle.$$
\end{proposition}
\begin{claim}\label{claim: categorical quotient}
    The map $$(\mathfrak{so}(U')^*\times U')\git \SO(U') \to \Sigma_{2n}\times \Sigma_{2n+1}$$
    whose first coordinate is given projecting on the $\mathfrak{so}(U')^*$-coordinate and applying the quotient map for $\mathfrak{so}(U')^*$, and whose second coordinate is given by composing $\varpi$ with the quotient map for $\mathfrak{so}(U)^*$, is an isomorphism.
\end{claim}
\begin{proof}
    The ring $\CC[\Sigma_{2n+1}]$ is generated by the even coefficients of the characteristic polynomial (the odd ones vanish). We have for $i$ even:
    \begin{equation}\label{eq: invariants}
        \begin{aligned}
            c_i(-x-u_0v^t+vu_0^t) =&c_i(x+u_0v^t-vu_0^t) \\=& c_i(x+u_0v^t) + \sum_{j=0}^{i-1} c_{i-j-1}(x+u_0v^t) \langle u_0, (x+u_0v^t)^j v \rangle\\
            =&c_i(x+u_0v^t) + \sum_{j=0}^{i-1} c_{i-j-1}(x+u_0v^t) \langle u_0, u_0v^tx^{j-1} v \rangle\\
            =& c_i(x) - \sum_{j=0}^{i-1} c_{i-j-1}(x) \langle v, x^j u_0 \rangle \\&+ \sum_{j=0}^{i-1} \left(c_{i-j-1}(x) - \sum_{k=0}^{i-j-2} c_{i-j-2-k}(x)\langle v, x^k u_0 \rangle\right) \langle v, x^{j-1} v \rangle\\
            =& c_i(x) + \sum_{j=0}^{i-1} \left(c_{i-j-1}(x)\right) \langle v, x^{j-1} v \rangle\\
        \end{aligned}
    \end{equation}
    We see that indeed over the ring $\CC[\Sigma_{2n}]$, the term $c_i(-x-u_0v^t+vu_0^t)$ is equal to $\langle v, x^{i-2} v \rangle$ plus lower order terms.
\end{proof}

\begin{proposition}\label{prop: psi-inv}
    Under the isomorphisms $$(\mathfrak{E}^\bullet)^{\SO(U')} = \Ext^\bullet_{\mathcal{C}}(E_0, E_0)\cong H^\bullet_{\SO(V')_\O\times \Sp(V)_\O}(\mathrm{pt})=\CC[\Sigma_{2n}\times \Sigma_{2n+1}],$$
    $$(\mathfrak{G}^\bullet)^{\SO(U')}\stackrel{\ref{claim: categorical quotient}}{\cong} \CC[\Sigma_{2n}\times \Sigma_{2n+1}],$$
    the map $(\mathfrak{G}^\bullet)^{\SO(U')}\to (\mathfrak{E}^\bullet)^{\SO(U')}$ induced from $\psi$ is the identity map.
\end{proposition}
\begin{proof}
    First recall that for any reductive group $G$ that the map given by the composition
    $$\CC[\Sigma_{\check{G}}] \cong \CC[\mathfrak{\check{g}}^*]^{\check{G}}\cong \Ext^\bullet_{\operatorname{D-mod}_{G_\O}(\Gr_{G})}(\delta_1, \delta_1)\cong \CC[\Sigma_{\check{G}}],$$
    where the middle map is given by the derived Satake isomorphism, is the identity map.

    Next, the composition
    $$\begin{aligned}
        \CC[\Sigma_{2n}\times \Sigma_{2n+1}]&\cong \Ext^\bullet_{\operatorname{D-mod}_{\SO(V')_\O\times \Sp(V)_\O}(\Gr_{\SO(V')}\times \Gr_{\Sp(V)})}(\delta_1, \delta_1) \\&\xrightarrow{*E_0} \Ext^\bullet_{\mathcal{C}}(E_0, E_0)\cong \CC[\Sigma_{2n}\times \Sigma_{2n+1}]
    \end{aligned}$$ is again the identity map.

    By the above, the restriction of the map $*E_0$ above to $\CC[\Sigma_{2n}]$ is the identity map on this component, and is equal to the map induced by $\psi'$. By Proposition \ref{prop: psi psi' identification}, this is the same as the composition $\CC[\Sigma_{2n}]\to (\mathfrak{G}^\bullet)^{\SO(U')}\xrightarrow{\psi} (\mathfrak{E}^\bullet)^{\SO(U')}\cong \CC[\Sigma_{2n}\times \Sigma_{2n+1}]$ (note that the composition of $\varpi^{-1}$ with the projection on $\mathfrak{so}(U')^*$ is the natural restriction map along the embedding $\mathfrak{so}(U')\to \mathfrak{so}(U)$).
    Last, since the composition $\CC\to \CC[\SO(U)] \xrightarrow{\mathrm{restr}} \CC[\SO(U')]$ is the natural map (inclusion of scalars), we get that the composition $$\begin{aligned}
        &\Ext^\bullet_{\mathcal{C}}(E_0, E_0)\cong \Ext^\bullet_{\mathrm{Ind}(\mathcal{C})}(E_0, \phi_{\SO(V)}(\CC[\SO(U)])*E_0)^{\SO(U)} \to \Ext^\bullet_{\mathrm{Ind}(\mathcal{C})}(E_0, \phi_{\SO(V)}(\CC[\SO(U)])*E_0)^{\SO(U')}\\
        & \cong \Ext^\bullet_{\mathrm{Ind}(\mathcal{C})}(E_0, \phi_{\SO(V')}(\CC[\SO(U)])*E_0)^{\SO(U')}
        \xrightarrow{\mathrm{restr}} \Ext^\bullet_{\mathrm{Ind}(\mathcal{C})}(E_0, \phi_{\SO(V')}(\CC[\SO(U')])*E_0)^{\SO(U')}
        \\&\cong \Ext^\bullet_{\mathcal{C}}(E_0, E_0)
    \end{aligned}$$
    is the identity map.
    Thus, we get that the composition $\CC[\Sigma_{2n+1}]\to (\mathfrak{G}^\bullet)^{\SO(U')}\xrightarrow{\psi} (\mathfrak{E}^\bullet)^{\SO(U')}$ is the same as the restriction of the map $*E_0$ above to $\CC[\Sigma_{2n+1}]$, hence the identity on this component too.
\end{proof}

\subsection{A pseudo-slice}\label{sec: pseudo-slice}
In this section we will assume that $n\geq 2$.
A full treatment of the case $n=1$ will be given in Appendix \ref{appendix: n=1}.
We shall construct a ``pseudo-slice" for the quotient map $\mathfrak{so}(U')^*\times U' \to \Sigma_{2n}\times \Sigma_{2n+1}$. Let $(e,f,h)$ be a principal $\mathfrak{sl}_2$ triple in $\mathfrak{so}(U')$, on the choice of which the slice will depend. As an $\mathfrak{sl}_2$-module, $U'$ is the direct sum of two irreducible representations - one, which we denote by $J\subset U'$, of dimension $1$ and one, which we denote by $I\subset U'$, of dimension $2n-1$. It is an easy check that $I,J$ are perpendicular.
Set $N:=\ker(f^{n-1})\subset I$ (the span of negative weight spaces for $h$ in $I$), an isotropic subspace of $U'$ of dimension $n-1$. Similarly, we set $N^*:=\ker(e^{n-1})\cap I$. It is another isotropic subspace of the same dimension, which the bilinear form identifies with the linear dual of $N$.
We choose a vector $v_0^*\in N^*$ of weight 2 with respect to $h$, such that $\langle fv_0^*, fv_0^*\rangle=-(n(n-1))^2$
(our construction will depend on such a choice, unique up to a sign). Consider $$M:=N + \CC(fv_0^* + n(n-1) u_0).$$ It is an $n$-dimensional isotropic subspace of $U$.
Finally, we set $$\Sigma_{(e, f, h, v_0^*)}:= \Sigma_{2n}\times (v_0^* + N + \CC(fv_0^*))\subseteq \mathfrak{so}(U')^*\times U'.$$
Where the embedding above is via the Kostant slice of $\mathfrak{so}(U')^*$. That is, $\Sigma_{2n}\cong e+\mathfrak{z}_f\subset \mathfrak{so}(U')\cong \mathfrak{so}(U')^*$.
\begin{claim}
    For all $\sigma\in \Sigma_{(e, f, h, v_0^*)}$, the element $\varpi(\sigma)\in \mathfrak{so}(U)^*\cong \mathfrak{so}(U')$ preserves $M$.
\end{claim}
\begin{proof}
    It is enough to show that $\varpi(e,v_0^*), \varpi(\mathfrak{z}_f, 0),\varpi(0, N)$ and $\varpi(0,fv_0^*)$ each preserve $M$.
    First, it is clear that $\varpi(e, v_0^*)$ sends the subspace of $N$ perpendicular to $v_0^*$ into $N$. We then compute
    $$\varpi(e, v_0^*)(f^2v_0^*) = -ef^2v_0^* - \langle v_0^*, f^2v_0^*\rangle u_0 = -n(n-1)fv_0^* - (n(n-1))^2u_0\in M$$
    and
    $$\varpi(e, v_0^*)(fv_0^* + n(n-1) u_0) = -efv_0^* + n(n-1)v_0^* = 0.$$
    This proves the claim for $(e, v_0^*)$.
    For $(\mathfrak{z}_f, 0)$, we first note the image of $M$ is contained in $N+J=\ker(f^{n-1})\subset U'$. Indeed, the image is contained in $\mathfrak{z}_f \ker(f^n)\subseteq \ker(f^{n-1})$ (the last inclusion is because $\mathfrak{z}_f$, preserving $\ker(f^n)$ and being nilpotent, acts trivially on the 1-dimensional $\ker(f^n)/\ker(f^{n-1})$). Second, the image is contained in $\mathfrak{z}_f \mathrm{Im}(f)\subset \mathrm{Im}(f)$. Thus, the image is contained in $(N+J)\cap \mathrm{Im}(f) = N\subseteq M$.

    For $(0,N)$ the claim is clear.
    Finally, for $(0,fv_0^*)$ it is only needed to check that
    $$\varpi(0,fv_0^*)(fv_0^*+n(n-1)u_0) = -\langle fv_0^*, fv_0^* \rangle u_0 + n(n-1) fv_0^* = n(n-1)(fv_0^* + n(n-1)u_0)\in M.$$
\end{proof}
\begin{claim}\label{claim: pseudo slice}
    There is an isomorphism $\Sigma_{(e, f, h, v_0^*)}\cong \Sigma=\Sigma_{2n}\times \Sigma_{\mathfrak{gl}}$, given by $(x,v)\mapsto (x, [\varpi(x, v)|_M])$, which intertwines the quotient map of Claim \ref{claim: categorical quotient} into $\Sigma_{2n}\times \Sigma_{2n+1}$ with the natural map $\Sigma_{2n}\times \Sigma_{\mathfrak{gl}}\to \Sigma_{2n}\times \Sigma_{2n+1}$.
\end{claim}
\begin{proof}
Let $(e+y, v_0^*+v)\in \Sigma_{(e, f, h, v_0^*)}$, with $y\in \mathfrak{z}_f, v\in N+\mathrm{Span}(fv_0^*)$.
Recall that the map from the Kostant slice of $\mathfrak{so}(U')$ to $\Sigma_{2n}$ is an isomorphism, hence the polynomials on $\Sigma_{2n}$ determine $y$ uniquely.
We choose the basis for $M$ given by $w_i=f^i (fv_0^* + n(n-1)u_0)$, $i=0,\dots,n-1$.
In this basis, $x_0:=\varpi(e,v_0^*)|_M$ takes the form of a nilpotent upper triangular matrix with 
$w_0\in \ker(e, v_0^*)$ and for $i>0$, $w_i=f^{i+1}v_0^*$ goes to a non-zero multiple of $w_{i-1}$. Also, $y$ takes the form of a nilpotent lower triangular matrix, that is, $yw_i\in \mathrm{Span}(w_{i+1},\dots,w_{n-1})$.
Write $v = \sum_{k=0}^{n-1} a_k f^{k+1}v_0^*$.
We have $$v^t|_M = a_0 (fv_0^*)^t|_M = -n(n-1)a_0u_0^t|_M,$$ and we calculate using \ref{prop: char poly formula} for $i=1,\dots n$:
$$
\begin{aligned}
c_i(\varpi(e+y, v_0^*+v)|_M) =& (-1)^ic_i(x_0+(y+u_0v^t-vu_0^t)|_M)
\\=& (-1)^ic_i(x_0+(y-(n(n-1)a_0u_0+v)u_0^t)|_M)
\\=& (-1)^ic_i(x_0+y|_M) +\\&+ (-1)^i\sum_{j=0}^{i-1}c_{i-j-1}(x_0+y|_M)\langle u_0, (x_0+y)^j (n(n-1)a_0u_0+v)\rangle
\\=& (-1)^ic_i(x_0+y|_M) +\\&+ (-1)^i\sum_{j=0}^{i-1}c_{i-j-1}(x_0+y|_M)\left\langle u_0, (x_0+y)^j \left(\sum_{k=0}^{n-1}a_kw_k\right)\right\rangle
\\=& (-1)^i\alpha_i a_{i-1} + (-1)^iP_i(y, a_0,\dots, a_{i-2})
\end{aligned}
$$
Where $\alpha_i=\langle u_0, x_0^{i-1} w_{i-1}\rangle$ is a non-zero constant and $P_i(y, a_0,\dots, a_{i-1})$ is an expression depending on $y$ and $a_j$ for $j<i$.
This indeed shows that the map in the claim is an isomorphism. The claim about compatibility of maps into $\Sigma_{2n}\times \Sigma_{2n+1}$ is obvious.
\end{proof}

\begin{lemma}\label{lemma: Sigma saturation}
    The image of the saturation map $\SO(U')\times \Sigma_{(e,f,h, v_0^*)}\to \mathfrak{so}(U')^*\times U'$ is Zariski dense with complement of codimension $\geq 2$. This image also contains the preimage along the map $\mathfrak{so}(U')^*\times U'\to \Sigma_{2n}\times \Sigma_{2n+1}$ of a Zariski open dense subset. Moreover, for any point $(x,v)$ in this preimage, its stabilizer in $SO(U')$ is trivial.
\end{lemma}
\begin{proof}
    We set $X\subset \mathfrak{so}(U')\times U'\times \mathrm{Gr}_{n-1}(U')\times \Gr_n(U)$ to be the subset of quadruples $(x, v, N', M')$ (here $N', M'$ are $n-1$ and $n$ dimensional linear subspaces of $U'$ and $U$ respectively) satisfying:
    \begin{itemize}
        \item $x$ is regular
        \item $M'$ is isotropic (that is, $(M')^\perp \supset M'$)
        \item $N' = M'\cap U'$
        \item $M'$ is preserved by $\varpi(x, v)$
        \item $\varpi(x, v)\left((N')^{\perp, U}\right)\subset \CC v + (N')^{\perp, U}$, and $v$ is cyclic for $\varpi(x,v)$ modulo $(N')^{\perp, U}$ (that is, $\CC v + (N')^{\perp, U}$ generates $U$ as a $\varpi(x,v)$-module)
    \end{itemize}
    Clearly, by setting $(x, v)=g\sigma, N'= gN, M' = gM$ for $(g, \sigma)\in SO(U')\times \Sigma_{(e,f,h, v_0^*)}$, the saturation map factors through $X$.
    We show that
    \begin{enumerate}
        \item The map $\SO(U')\times \Sigma_{(e,f,h, v_0^*)} \to X$ is onto.
        \item The image of the map $X\to \mathfrak{so}(U')\times U'$ (given by projection on the first two coordinates) is dense with complement of codimension $\geq 2$.
        \item The above image contains the preimage along the map $\mathfrak{so}(U')\times U'\to \Sigma_{2n}\times \Sigma_{2n+1}$ of a Zariski open dense subset.
        \item Over this Zariski open dense subset, the stabilizers are trivial (as in the formulation of the lemma).
    \end{enumerate}
    To show (1), we begin by identifying $U' / (N')^{\perp, U'}\cong U / (N')^{\perp, U} \cong (N')^*$. We obtain an ascending filtration of this space by setting $$F^i(N')^*:=\mathrm{Span}(\{\varpi(x, v)^j v\mid j\leq i\}).$$
    We also consider the dual filtration, defined on $N'$ by taking the perpendicular subspaces.
    Let $\mathrm{B}'\subset \SO(U')$ be the subgroup formed by transformation preserving $N'$ as well as the above filtrations on $N', U' / (N')^\perp$. This is a Borel subgroup of $\SO(U')$, and we transform $(x, v, M', N')$ by an element of $\SO(U')$ transforming this subgroup to the standard negative Borel with respect to the principal $\mathfrak{sl}_2$ triple $(e,f,h)$.
    In particular, we may assume that $N' = N$, and that the above filtrations defined on $N^*, N$ are the standard ones, preserved by the standard negative Borel subgroup.

    Next, we note that by definition, for $0\leq i \leq n-2$, $\varpi(x, v)$ takes $F^iN^*$ into $F^{i+1}N^*$ with the resulting quotient map $$\varpi(x,v): F^iN^*/F^{i-1}N^*\to F^{i+1}N^*/F^iN^*$$ being an isomorphism of one dimensional vector spaces.
    
    One easily observes that it follows that $x$ satisfies these same properties, and that $xN^{\perp, U'}\subset \CC v \oplus N^{\perp, U'}$ and $x N \subset N^{\perp, U'}$.
    Moreover, taking $s\in N$ such that $\langle s, v\rangle = 1$, we see that
    $$\begin{aligned}
        x^2s-v &= x(xs+u_0) - \langle -xs-u_0, v \rangle u_0 - v - \langle xs, v \rangle u_0 =\\& -x(\varpi(x,v)s) - \langle \varpi(x,v)s, v \rangle u_0 + \langle \varpi(x,v)s, u_0 \rangle v - \langle xs, v \rangle u_0 =\\& \varpi(x, v)^2s - \langle xs, v \rangle u_0\in U'\cap(M'+\CC u_0)\subset N^{\perp, U'}.
    \end{aligned}$$
    This implies that $xN\nsubseteq N$ and $x^2N\nsubseteq xN$, as well as $\langle xs, xs\rangle = -\langle v, s \rangle = -1$.
    From all the above, it follows that one may transform by an element of the standard torus (the centralizer of $h$) so that $x\in e + \mathfrak{b}^-$, where $\mathfrak{b}^-$ is the Lie algebra of the standard negative Borel $\mathrm{B}^-$.
    We may further transform by an element of $\mathrm{U}^-$, the unipotent radical of $\mathrm{B}^-$, so that $x\in \Sigma_{2n}$.
    Since $xs\in xN+N = eN+N$ and $x^2s-v\in N^{\perp, U'}$, we deduce that $$xs \in \frac{fv_0^*}{n(n-1)} + N,$$
    which implies that $M'=M$.
    Now, write
    $$M\ni \varpi(x, v) (fv_0^* + n(n-1)u_0) = -xfv_0^* -\langle fv_0^*, v\rangle u_0 + n(n-1)v = n(n-1)(v-v_0^*) +(e-x)fv_0^* -\langle fv_0^*, v\rangle u_0.$$
    This implies that
    $$v-v_0^*\in (M + \CC u_0)\cap U' = \CC fv_0^* + N.$$
    This finishes part 1 of the proof.

    For part 2, look at the spectral decomposition $U = \bigoplus_{\lambda\in S} U_\lambda$ with respect to the linear transformation $\varpi(x, v)$ (here $S$ is the set of eigenvalues). We write also $v = \sum_{\lambda\in S} v_\lambda$, $v_\lambda\in U_\lambda$. We show that the following subset $Y\subset\mathfrak{so}(U')\times U'$, the complement of which is of codimension at least 2, is contained in the image of $X$:
    \begin{enumerate}
        \item $x$ and $\varpi(x, v)$ are regular.
        \item For each $\lambda\in S\setminus \{0\}$, either $v_\lambda$ is a cyclic vector for $\varpi(x,v)$ in $U_\lambda$ or $v_{-\lambda}$ is a cyclic vector for $\varpi(x, v)$ in $U_{-\lambda}$.
        \item Either $x$ is invertible, or $\dim U_0=1$.
    \end{enumerate}

    We start by showing that the complement of $Y$ is indeed of codimension at least 2.
    The first requirement indeed cuts out only a codimension $\geq 2$ subset.
    Now, assuming that the first requirement indeed holds, for any $\lambda\in S$, the $\varpi(x, v)$-module $U_\lambda$ is cyclic.
    The subset of such $(x,v)$ for which there exists a $\lambda\in S\setminus\{0\}$ so that $v_\lambda$ and $v_{-\lambda}$ are both not cyclic in $U_\lambda, U_{-\lambda}$ respectively is also of codimension 2. Last, the third requirement also forbids a space of codimension 2, as the subsets in which $x$ is not invertible and in which $\dim U_0 > 1$ (note that $\dim U_0\geq 1$ always) are both closed subvarieties, and don't share any irreducible components.
    
    Now let us show that $Y$ is contained in the image of $X$.
    Let $(x,v)\in Y$, and let $S^+\subset S\setminus\{0\}$ be a subset which contains exactly one of $\lambda, -\lambda$ for every $\lambda\in S\setminus \{0\}$, and so that for any $\lambda\in S^+$, $v_\lambda$ is a cyclic vector in $U_\lambda$ for $\varpi(x, v)$. Let $S^- := S\setminus (S^+\cup \{0\})$.
    We also set $d := \dim U_0$ (note that $d$ must be odd), and $$\tilde{U}_0 := \ker(\varpi(x,v)^{\frac{d-1}{2}}) = \varpi(x, v)^{\frac{d+1}{2}}U_0.$$
    We define $M' := \bigoplus_{\lambda\in S^-} U_\lambda \oplus \tilde{U}_0$. It is an isotropic subspace of dimension $n$, preserved by $\varpi(x, v)$. 
    
    \begin{claim}\label{claim: v_0 subcyclic}
        If $u_0$ is in the image of $\varpi(x, v)$, then $\dim U_0 = 1$. Moreover, whether this holds or not, $v_0$ is cyclic for $\varpi(x, v)$ in $\varpi(x, v)U_0$.
    \end{claim}
    \begin{proof}
        Assume the $u_0$ is in the image of $\varpi(x, v)$, that is, $\varpi(x,v)w = u_0$ for some $w\in U$. This implies that $w\in U'$, and $xw=0$, hence $\dim U_0 = 1$ by condition (3) above. Note that $v\in \varpi(x,u)U$, and the desired cyclicity condition holds trivially.
        If $u_0$ is not in the image of $\varpi(x, v)$, the $U_0$-part of $u_0$ is cyclic for $\varpi(x, v)$. Since $\varpi(x, v)u_0 = v$, this would imply that $v_0$ is cyclic for $\varpi(x, v)$ in $\varpi(x, v)U_0$.
    \end{proof}

    It is impossible that $M'\subset U'$, as if this were the case, then it would follow from $$U'\supset M'\supset\varpi(x,v)M'=(-x-u_0v^t)M'$$ that $M'\subseteq v^{\perp, U'}$. By definition of $S^+$ along with Claim \ref{claim: v_0 subcyclic}, this implies that $S = \{0\}$ and $\dim U_0 \leq 3$. However, we assumed that $n>1$.

    So the subspace $N':= M'\cap U'$ is of dimension $n-1$, and
    $$(N')^{\perp, U} = (M')^\perp + \CC u_0.$$
    Since $\varpi(x,v)\in \mathfrak{so}(U)$ preserves $M'$, it must also preserve $(M')^\perp$. Also, $\varpi(x, v)u_0 = v$, and so we see that indeed $$\varpi(x, v)\left((N')^{\perp, U}\right)\subset \CC v + (N')^{\perp, U}.$$

    It follows from Claim \ref{claim: v_0 subcyclic} above that
    $$(M')^\perp + \CC u_0\nsubseteq \varpi(x,v)U$$
    and that $v$ is cyclic for $\varpi(x, v)$ modulo $(N')^{\perp, U} = (M')^\perp + \CC u_0$ (This is proved by looking at each eigenvalue separately).
    Thus, $(x, v, N', M')\in X$, concluding the proof that $Y$ is contained in the image of $X$.
    
    To show part 3, namely that the image of $X$ contains the preimage along the map $\mathfrak{so}(U')\times U'\to \Sigma_{2n}\times \Sigma_{2n+1}$ of a Zariski open dense subset, we show that $Y$ contains such a preimage.
    More specifically, we show that $Y$ contains the subset of $\mathfrak{so}(U')\times U'$ consisting of pairs $(x, v)$ such that both $x$ and $\varpi(x, v)$ are regular semisimple and such that $x$ and $\varpi(x, v)$ share no common eigenvalues.
    Condition 1  in the definition of $Y$ holds trivially on this subset, and also condition 3, as $\varpi(x, v)$ must have $0$ as an eigenvalue, hence $x$ must be invertible.
    To show that condition 2 holds, note that each $U_\lambda$ is 1-dimensional, and so it is enough to show that for each $\lambda\neq 0$, the component of $v$ in $U_\lambda$ is non-zero. Assume by contradiction that it is, and take $w\in U_{-\lambda}$ to be a non-zero vector of eigenvalue $-\lambda$. We must have $\langle w, v\rangle = 0$.
    Thus,
    $$-\lambda w = \varpi(x, v) w = -xw + \langle u_0, w \rangle v\in U',$$
    so that necessarily $w\in U'$, and the above equality reads $xw = \lambda w$, contradicting the assumption that $x$ and $\varpi(x, v)$ share no eigenvalues.

    Now let us show that for any $(x,v)\in \mathfrak{so}(U')\times U'$ for which $x,\varpi(x,v)$ are both regular semisimple and share no common eigenvalues, as above, the stabilizer of this pair in $\SO(U')$ is trivial. For such a pair $(x,v)$, we have shown that $v$ has non-zero component in $U_\lambda$ for each $\lambda\neq 0$. Thus, if $g\in \SO(U')$ is an element centralizing $(x, v)$, $g$ must preserve all the subspaces $U_\lambda$ of $U$, and moreover to restrict to the identity operator on each $U_\lambda$ for $\lambda\neq 0$. Since $g$ also fixes $u_0$, and we know that $U_0\nsubseteq U'$ (otherwise we get that $x$ must also annihilate $U_0$, and so $0$ would be a common eigenvalue of $x$ and $\varpi(x, v)$), we get that $g$ must also be the identity on $U_0$, hence $g=1$.
\end{proof}

Let $$\mathfrak{M}:=(\mathfrak{so}(U')^*\times U')\times_{\Sigma_{2n}\times \Sigma_{2n+1}}\Sigma,$$
and let $\upsilon: \SO(U')\times \Sigma_{(e, f, h, v_0^*)}\to \mathfrak{M}$ be the map whose first coordinate is given by the saturation map, and whose second coordinate is the projection map onto $\Sigma_{(e, f, h, v_0^*)}\cong \Sigma$.
\begin{corollary}
    \begin{enumerate}[label=(\roman*), ref=\thecorollary(\roman*)]
        \item We have $\mathfrak{M}\git \SO(U')\cong \Sigma$, and $\Sigma_{e,f,h, v_0^*}\to \mathfrak{M}$ is a Weierstraß section for the $\SO(U')$ action.
        \item \label{cor-item: regularity can be checked on Sigma} If the pullback to $\SO(U')^*\times \Sigma_{e, f, h, v_0^*}$ of a rational function on $\mathfrak{so}(U')^*\times U'$ is regular, then the original function must be regular.
        \item \label{cor-item: upsilon gen} Letting $\Sigma_{\mathrm{gen}}\cong \Sigma_{(e, f, h, v_0^*),\mathrm{gen}}$ and $(\Sigma_{2n}\times \Sigma_{2n+1})_\mathrm{gen}$ be the generic points of $\Sigma\cong \Sigma_{(e, f, h, v_0^*)}$ and $\Sigma_{2n}\times \Sigma_{2n+1}$ respectively, the map $$\upsilon_{\mathrm{gen}}:=\upsilon\times_\Sigma \Sigma_\mathrm{gen}:\SO(U')\times \Sigma_{(e, f, h, v_0^*),\mathrm{gen}} \xrightarrow{\sim}(\mathfrak{so}(U')^*\times U')\times_{\Sigma_{2n}\times \Sigma_{2n+1}}\Sigma_{\mathrm{gen}}$$ is an isomorphism.
    \end{enumerate}
\end{corollary}
\begin{proof}
    Parts 1 and 2 follow trivially from Lemma \ref{lemma: Sigma saturation}.
    For part 3, let $(\Sigma_{2n}\times\Sigma_{2n+1})_\mathrm{gen}$ be the generic point of $\Sigma_{2n}\times\Sigma_{2n+1}$. The composition of $\upsilon_\mathrm{gen}$ with the finite map
    $$(\mathfrak{so}(U')^*\times U')\times_{\Sigma_{2n}\times \Sigma_{2n+1}}\Sigma_{\mathrm{gen}}\to (\mathfrak{so}(U')^*\times U')\times_{\Sigma_{2n}\times \Sigma_{2n+1}}(\Sigma_{2n}\times \Sigma_{2n+1})_\mathrm{gen}$$
    is onto by the second part of the Lemma.
    Since $(\Sigma)_\mathrm{gen}\to (\Sigma_{2n}\times \Sigma_{2n+1})_\mathrm{gen}$ is a connected covering of connected reduced 0-dimensional schemes (spectra of fields), it follows that $\upsilon_\mathrm{gen}$ itself must be onto. By the last part of Lemma \ref{lemma: Sigma saturation}, we get that $\upsilon_\mathrm{gen}$ is in fact an isomorphism.
\end{proof}

\section{Localization}\label{sec: localization}
\begin{definition}
    We set
    $$\tilde{\mathfrak{G}}^\bullet :=\mathfrak{G}^\bullet\otimes_{\CC[\Sigma_{2n}\times \Sigma_{2n+1}]} \CC[\Sigma].$$
    Moreover, for any $\CC[\Sigma]$-module $M$, we set $M_\mathrm{loc}:=M\otimes_{\CC[\Sigma]} \CC(\Sigma)$.
\end{definition}

As a setup for the next lemma, let $G$ be a reductive group, and let $N\in \mathrm{Rep}(G)$. For $W\in \mathrm{Rep}(\check{G})$, we set $W=\oplus_{d\in \ZZ} W_d$ to be the grading by weights of $2\rho$ on $W$. We define $W^{[]}:=\oplus_{d\in \ZZ} W_d[-d]$ (note that all shifts are by even integers). By \cite{MV}, we have $R\Gamma(\Gr_{G}, \phi_{G}(W))\cong W^{[]}$, and this isomorphism is compatible with the tensor structure.
\begin{lemma}\label{lemma: Hecke action on costalk}
    For $G,N$ as above, the functor $i^{!,ren}:\operatorname{D-mod}_{G_\O}^*(N_\mathcal{K})\to \operatorname{D-mod}_{G_\O}(\mathrm{pt})$ intertwines the Hecke action of $\mathrm{Rep}(\check{G})$ on $\operatorname{D-mod}_{G_\O}^*(N_\mathcal{K})$ with the action on $\operatorname{D-mod}_{G_\O}(\mathrm{pt})$ given for $W\in \mathrm{Rep}(\check{G})$ by tensoring with $W^{[]}$, using the obvious tensor structure.
\end{lemma}
\begin{proof}
Consider the Cartesian square
$$\begin{tikzcd}
\Gr_{G}\times \{0\}\ar[r, "i_\mathrm{Gr}"]\ar[d]& G_\mathcal{K}\stackrel{G_\O}{\times} N_\mathcal{K} \ar[d]\\
\{0\}\ar[r, "i"] & N_\mathcal{K}
\end{tikzcd}$$
For any $\F\in \operatorname{D-mod}^*_{G_\O}(N_\mathcal{K})$ and $W\in \mathrm{Rep}(\check{G})$, we have by base change $$i^{!,ren} (\phi_{G}* \F)\cong R\Gamma(\Gr_{G}, i_\mathrm{Gr}^{!,ren}(\phi_{G}(W)\tilde{\boxtimes} \F)) \cong R\Gamma(\Gr_{G}, \phi_{G}(W)\otimes i^{!,ren} \F)\cong W^{[]}\otimes i^{!,ren}\F.$$
To show compatibility with the tensor structure one uses the coherence of base change morphisms along the following diagram:
$$\begin{tikzcd}[row sep=3em, column sep=1em]
    G_\mathcal{K}\times^{G_\O} \Gr_{G}\times \{0\}\ar[rr]\ar[dr, "\mathrm{mul}"]\ar[dd, "\mathrm{pr}_{1,3}"] && G_\mathcal{K}\times^{G_\O} G_\mathcal{K}\times^{G_\O} N_\mathcal{K}\ar[dr, "{(\mathrm{mul},\mathrm{id}_{N_\mathcal{K}})}"]\ar[dd, "{(\mathrm{id}_{G_\mathcal{K}}, \mathrm{act})}"{pos=0.75}]\\
    &\Gr_{G}\times \{0\}\ar[rr]\ar[dd]&& G_\mathcal{K}\times^{G_\O}N_\mathcal{K}\ar[dd]\\
    \Gr_{G}\times \{0\}\ar[rr]\ar[dr]&& G_\mathcal{K}\times^{G_\O}N_\mathcal{K}\ar[dr]\\
    &\{0\}\ar[rr]&& N_\mathcal{K}\\
\end{tikzcd}$$
\end{proof}

\begin{claim}\label{claim: loc thm}\begin{enumerate}
    \item Disregarding the gradings, there exists an $\SO(U')$-equivariant isomorphism of $\CC(\Sigma)$-algebras $$\tilde{\mathfrak{E}}^\bullet_\mathrm{loc} \cong \CC[\SO(U')]\otimes \CC(\Sigma).$$
    \item The algebra $\mathfrak{E}^\bullet$ is supported in even degrees, and is commutative.
    \item Assuming $n\geq 2$, the map $$\psi_\mathrm{loc}:\mathfrak{G}^\bullet_\mathrm{loc}\to \mathfrak{E}^\bullet_\mathrm{loc}$$ induced from $\psi$ is an isomorphism (the case of $n=1$ will be proved in Appendix \ref{appendix: n=1}).
\end{enumerate}
\end{claim}
\begin{proof}
    Let $T\subset \SO(V')\times \GL(L)$ be a maximal torus, and let $W$ be its Weyl group. From part 3 of Theorem \ref{thm: purity} we deduce that
    $$\tilde{\mathfrak{E}}^\bullet\cong \left(\Ext^\bullet_{\mathrm{Ind}(\operatorname{\mathcal{W}-\mathrm{mod}}^{T_\O,\mathrm{lc}})}(E_0, \phi_{\SO(V')}(\CC[\SO(U')])*E_0)\right)^W.$$
    Set $$i:\{0\}\to (V'\otimes L)_\mathcal{K}$$
    and $$i_\O: (V'\otimes L)_\mathcal{O}\to (V'\otimes L)_\mathcal{K}.$$
    By fixing the pair of transversal Lagrangians $L, L^*$, we get an isomorphism $$\operatorname{\mathcal{W}-\mathrm{mod}}\cong \operatorname{D-mod}^*((V'\otimes L)_\mathcal{K})$$
    compatible with the natural $\SO(V')_\mathcal{K}\times \GL(L)_\mathcal{K}$ actions on both sides.

    Under the above equivalence, $E_0$ corresponds to the renormalized dualizing ($\SO(V')_\O\times \GL(L)_\O$-equivariant) D-module supported on $(V'\otimes L)_\O$, i.e. $i_{\O,*}\omega^{ren}_{(V'\otimes L)_\mathcal{\O}}$.
    From the localization theorem (see \cite[(6.2)]{GKM}, cf. \cite[sections 3.1, 3.2]{BFN}), we deduce that for any $\F, \G\in \operatorname{D-mod}_{T_\O}((V'\otimes L)_\mathcal{K})$ we have:
    \begin{equation}\label{eq:loc thm}
        \Ext^\bullet_{\operatorname{D-mod}_{T_\O}((V'\otimes L)_\mathcal{K})}(\F,\G)\otimes_{\CC[\mathfrak{t}[2]]} \CC(\mathfrak{t}[2])\cong\Ext^\bullet_{\operatorname{D-mod}_{T_\O}(\mathrm{pt})}(i^{!,ren}\F,i^{!,ren}\G) \otimes_{\CC[\mathfrak{t}[2]]}\CC(\mathfrak{t}[2]).
    \end{equation}

    From Lemma \ref{lemma: Hecke action on costalk} together with \eqref{eq:loc thm}, we deduce the isomorphism of algebras $$\begin{aligned}
        &\Ext^\bullet_{\mathrm{Ind}(\operatorname{\mathcal{W}-\mathrm{mod}}^{T_\O,\mathrm{lc}})}(E_0, \phi_{\SO(V')}(\CC[\SO(U')])*E_0)\otimes_{\CC[\mathfrak{t}[2]]} \CC(\mathfrak{t}[2])\\\cong&
        \Ext^\bullet_{\operatorname{D-mod}_{T_\O}(\mathrm{pt})}(\CC, \CC[\SO(U')]^{[]})\otimes_{\CC[\mathfrak{t}[2]]}\CC(\mathfrak{t}[2])\\\cong &
        \CC[\SO(U')]\otimes \Ext^\bullet_{\operatorname{D-mod}_{T_\O}(\mathrm{pt})}(\CC, \CC)\otimes_{\CC[\mathfrak{t}[2]]}\CC(\mathfrak{t}[2]) \cong \CC[\SO(U')]\otimes \CC(\mathfrak{t}[2]).
    \end{aligned}$$
    Note that for the above chain of isomorphisms is compatible with the gradings when we equip $\CC[\SO(U)]$ with the grading given by $2\rho$. In particular, the above algebra is supported in even degrees.
    Taking W-invariants, we get an isomorphism $\tilde{\mathfrak{E}}^\bullet_\mathrm{loc} \cong \CC[\SO(U')]\otimes \CC(\Sigma)$, proving part 1.

    Part 2 follows, as due to Theorem \ref{thm: purity} the map $\mathfrak{E}^\bullet\to \tilde{\mathfrak{E}}^\bullet_\mathrm{loc}$ is injective, and $\tilde{\mathfrak{E}}^\bullet_\mathrm{loc}$ is commutative and supported in even degrees.

    By Corollary \ref{cor-item: upsilon gen}, there is an isomorphism $\tilde{\mathfrak{G}}^\bullet_{\mathrm{loc}} \cong \CC[\SO(U')]\otimes \CC(\Sigma)$. Since the map $\tilde{\psi}_\mathrm{loc}$ is $\CC(\Sigma)$-linear (by Proposition \ref{prop: psi-inv}) and $\SO(U')$-equivariant, it must be an isomorphism. It is a base change of $\psi_\mathrm{loc}$, hence $\psi_\mathrm{loc}$ is also an isomorphism, proving part 3.
\end{proof}

\section{Equivariant Cohomology}\label{sec: equivariant cohomology}
The goal of this section is to prove Proposition \ref{prop: pseudo inverse}.
In this section we will once again assume that $n\geq 2$.

As in \cite[§3.5]{BFKT}, we consider the equivariant cohomology functor
$$\chi:\mathrm{Rep}(\SO(U'))\to \mathrm{Vect}^{\GG_m}(\Sigma)$$
$$M\mapsto H^\bullet_{\SO(V')_\O\times \GL(L)_\O}((V'\otimes L)_\mathcal{K}, \phi_{\SO(V')}(M)*E_0).$$
The fact that the resulting coherent sheaves on $\Sigma$ are indeed locally free is proved using the same purity argument used in the proof of Theorem \ref{thm: purity}.
Note that here we take the direct sum of cohomologies in all degrees, so we really get a $\ZZ$-graded (non-derived) vector bundle on $\Sigma$, and the $\GG_m$-equivariant structure is given by the grading.

The following Lemma is proved in \cite[Lemma 3.5.1]{BFKT}:
\begin{lemma}
$\chi$ is a tensor functor.
\end{lemma}
This gives rise to the construction of a $\GG_m$-equivariant $\SO(U')$-torsor $\mathcal{T}$ over $\Sigma$, given by $$\mathcal{T}:=\Spec_{\Sigma}(\chi(\CC[\SO(U')]))\cong \Spec(\Gamma(\Sigma, \chi(\CC[\SO(U')]))).$$
As in \cite{BFKT}, we will define a $\SO(U')\times \GG_m$-equivariant map $\mathcal{T}\xrightarrow{\eta} \Spec(\mathfrak{E}^\bullet)$.
It is given by the natural action of $\mathscr{E}$ on $\chi(\CC[\SO(U')])$, obtained by applying the equivariant cohomologies functor to the composition
$$\begin{aligned}
    \mathscr{E}\otimes(\phi_{\SO(V')}(\CC[\SO(U')])*E_0)&\to \phi_{\SO(V')}(\CC[\SO(U')])*(\phi_{\SO(V')}(\CC[\SO(U')])*E_0) \\&= \phi_{\SO(V')}(\CC[\SO(U')]\otimes \CC[\SO(U')])*E_0\xrightarrow{\mathrm{mult}}\phi_{\SO(V')}(\CC[\SO(U')])*E_0.
\end{aligned}$$
This action respects the structure of $\chi(\CC[SO(U')])$ as a module over itself, and hence gives rise to the desired map $\mathfrak{E}^\bullet\to \CC[\mathcal{T}]$.

We make the following
\begin{proposition}
    The composition
    $$\Sigma\xrightarrow{\eta/\SO(U')} \Spec(\mathfrak{E}^\bullet)/\SO(U')\xrightarrow{\psi/\SO(U')} \Spec(\mathfrak{G}^\bullet)/\SO(U')\to \Spec(\mathfrak{G}^\bullet)\git\SO(U')\stackrel{\ref{claim: categorical quotient}}{\cong}\Sigma_{2n}\times \Sigma_{2n+1}$$ is the natural map $\Sigma\to \Sigma_{2n}\times \Sigma_{2n+1}$.
\end{proposition}
\begin{proof}
    By Corollary \ref{prop: psi-inv}, it is enough to show that the composition
    $$\Sigma\to \Spec(\mathfrak{E}^\bullet)/\SO(U')\to \Spec(\mathfrak{E}^\bullet)\git\SO(U')$$
    is the natural map $\Sigma\to \Sigma_{2n}\times \Sigma_{2n+1}$.
    Look at the map $(\mathfrak{E}^\bullet)^{\SO(U')}\to \CC[\mathcal{T}]$. It is given by
    $$\begin{aligned}
        H^\bullet_{\SO(V')_\O\times \Sp(V)_\O}(\mathrm{pt})&\to H^\bullet_{\SO(V')_\O\times \GL(L)_\O}(\mathrm{pt})\cong \Ext^\bullet_{\operatorname{\mathcal{W}-\mathrm{mod}}^{\SO(V')_\O\times \GL(L)_\O}}(E_0, E_0)\\
        &\to H^\bullet_{\SO(V')_\O\times \GL(L)_\O}((V'\otimes L)_\mathcal{K}, \phi_{\SO(V')}(\CC[\SO(U')])*E_0).
    \end{aligned}$$
    But this coincides with the composition of the obvious map $\CC[\Sigma_{2n}\times \Sigma_{2n+1}]\to \CC[\Sigma]$ and the tautological map $\CC[\Sigma]\to \CC[\mathcal{T}]$, realizing equivariant cohomology as a module over the equivariant cohomology of a point.
\end{proof}
The following proposition is a consequence of \cite[Theorem 1]{Wed}:
\begin{proposition}\label{prop: torsor triviality}
    The torsor $\mathcal{T}$ is trivial (forgetting the $\GG_m$-equivariant structure), and there is a $\GG_m$-equivariant isomorphism $\mathcal{T}\cong \Sigma\times (\mathcal{T}\times_{\Sigma} \{0\})$.
\end{proposition}
From now on, we fix an isomorphism as above and a point $p\in\mathcal{T}\times_{\Sigma} \{0\}$. This yields an isomorphism $\SO(U')\xrightarrow[g\mapsto gp]{\sim}\mathcal{T}\times_{\Sigma} \{0\}$.
Composing the section $\Sigma\xrightarrow{(id,p)} \mathcal{T}$, which we call $s$, with the above map $\eta:\mathcal{T}\to \Spec(\mathfrak{G}^\bullet)\cong \mathfrak{so}(U')^*\times U'$ we obtain a ``pseudo-slice" $\xi: \Sigma\to \mathfrak{so}(U')^*\times U'$ for the quotient map $\mathfrak{so}(U')^*\times U'\to \Sigma_{2n}\times \Sigma_{2n+1}$. We will show that this pseudo-slice is the same as the pseudo-slice constructed in Section \ref{sec: pseudo-slice}, for an appropriate choice of $(e,f,h,v_0^*)$.
More precisely, we will prove the following Proposition, to which the rest of this section is dedicated.

\begin{proposition}\label{prop: pseudo inverse}
    Let $\iota$ be the composition of the (inverse of the) isomorphism in Claim \ref{claim: pseudo slice} with the natural inclusion $\Sigma_{(e, f, h, v_0^*)}$ into $\mathfrak{so}(U')^*\times U'$.
    There are choices of a quadruple $(e,f,h,v_0^*)$ as in Section \ref{sec: pseudo-slice}, a $\GG_m$-equivariant automorphism $\tau$ of $\Sigma$, and a morphism $\gamma:\Sigma\to \SO(U')$ so that $\gamma(\sigma)\xi(\sigma) = \iota \circ \tau(\sigma)$.
    Equivalently, setting $s':=s\circ \tau^{-1}$ and $\gamma' := \gamma \circ \tau^{-1}$, the following diagram commutes:
    
    \[\begin{tikzcd}[column sep=6em]
    \SO(U')\times \Sigma \ar[r, "\cong", "{(g, \sigma)\mapsto g\gamma'(\sigma)s'(\sigma)}"'] \ar[rrr, bend left, "{(g, \sigma)\mapsto g\iota(\sigma)}"] & \mathcal{T} \ar[r,"\eta"]
    &\mathrm{Spec}(\mathfrak{E}^\bullet)\ar[r, "\psi"]&\mathrm{Spec}(\mathfrak{G}^\bullet)=\mathfrak{so}(U')^*\times U'.
    \end{tikzcd}\]
\end{proposition}
\begin{remark}
    The above twist by $\gamma$ can be reformulated as replacing the section $s$ by its twist $\gamma s$.
\end{remark}

We begin with the following lemma:
\begin{lemma}
    The image of $0\in \Sigma$ under $\xi$ is of the form $(e,v^*)$, where $e$ is a regular nilpotent element, and $v^*\in U'$.
\end{lemma}
\begin{proof}
    Since the map $\Sigma\to \Sigma_{2n}$ is smooth, the differential of the map $\mathfrak{so}(U')^*\times U'\to \Sigma_{2n}$ must be surjective for all points in the image of $\xi$, hence their $\mathfrak{so}(U')^*$-component must be a regular element.
\end{proof}

\begin{proposition}\label{prop: pincipal sl2}
    There is a choice of $f,h\in\mathfrak{so}(U')$ which together with $e\in \mathfrak{so}(U')^*\cong \mathfrak{so}(U')$ above form a principal $\mathfrak{sl}_2$-triple $(e,f,h)$, such that if $T:=Z_h$ is the stabilizer of $h$ in $\SO(U')$ and $2\rho^\vee:\GG_m\to T\to \SO(U')$ is the composition of the cocharacter $2\rho^\vee\in X_*(T)$ (given with respect to the choice of positive roots defined by $e$) with the inclusion map into $\SO(U')$, then the $\GG_m$-action on $\mathcal{T}\times_{\Sigma}\{0\}\cong \SO(U')$ is given by right multiplication by $2\rho^\vee$.
\end{proposition}
\begin{proof}
    Since the $\GG_m$-action on $\mathcal{T}\times_{\Sigma}\{0\}\cong \SO(U')$ commutes with the action by multiplication from the left of $\SO(U')$ on itself, it is given by right multiplication by some cocharacter $\lambda:\GG_m\to \SO(U')$.
    For $c\in \CC^\times$, we consider $(\lambda(c), c^{-1})\in \SO(U')\times \GG^m\circlearrowright \mathcal{T}$. We have:
    $$s(0) = (\lambda(c), c^{-1}).s(0)\stackrel{\psi\circ\eta}{\longmapsto} (\lambda(c), c^{-1}).(e, v^*) = (c^{-2} \mathrm{Ad}_{\lambda(c)}e, c^{-2} \lambda(c)v^*).$$
    Thus, $\mathrm{Ad}_{\lambda(c)}e = c^2 e$.
    Differentiating, we get that $h:=d\lambda\in \mathfrak{so}(U')$ is semisimple, and satisfies $[h,e]=2e$.
    
    By the last part of \cite[Theorem 3.6]{Kos}, we get that there exists $f\in \mathfrak{so}(U')$ such that $(e,f,h)$ form a principal $\mathfrak{sl}_2$-triple, and in particular $h$ must be regular semisimple, hence $T:=Z_h$ is a maximal torus.
    Since $e$ is regular, it follows from $[h,e]=2e$ that $\lambda:\GG_m\to T$ must be equal to the cocharacter $2\rho^\vee$.
\end{proof}
From now on we fix such $f,h$ and $T$.
\begin{proposition}\label{prop: Gm equivariance of section}
    The image of the section $\Sigma\xrightarrow{s} \mathcal{T}$ is invariant under the $\GG_m$-action obtained by restricting the $\GG_m\times \SO(U')$-action along $\GG_m\xrightarrow{(1, -2\rho^\vee)}\GG_m\times \SO(U')$. Moreover, the resulting $\GG_m$ action on $\Sigma$ is the natural one.
\end{proposition}
\begin{proof}
    By Propositions \ref{prop: torsor triviality} and \ref{prop: pincipal sl2}, we know that
    $$(c, -2\rho^\vee(c)).(\sigma, 1) = (c.\sigma, (-2\rho^\vee)(c)\cdot (2\rho^\vee(c))) = (c.\sigma, 1).$$
    (Note that here $c.\sigma$ is the natural action, i.e. dilations by $c^2$).
\end{proof}

\begin{corollary}
    The point $\xi(0)=(e,v^*)\in \mathfrak{so}(U')^*\times U'$ is a fixed point for the twisted $\GG_m$ action given by $$c.(x,v)=(c^2\mathrm{Ad}_{(-2\rho^\vee)(c)}(x), c^2 (-2\rho^\vee)(c) v).$$
    Moreover, the image of $\xi$ is contained in the attractor $\mathcal{A}$ to $(e, v^*)$ with respect to this twisted $\GG_m$ action.
\end{corollary}
\begin{proof}
    The described twisted action is nothing but the restriction of the $\GG_m\times \SO(U')$-action along $\GG_m\xrightarrow{(1, -2\rho^\vee)}\GG_m\times \SO(U')$. Since $\xi$ intertwines the natural $\GG_m$-action on $\Sigma$ with this action, and the natural action on $\Sigma$ attracts all of $\Sigma$ to the fixed point $0\in\Sigma$, the result follows.
\end{proof}
One readily checks that $\mathcal{A} = (e + \mathfrak{b}^-)\times (v^* + \ker(f^n))$, $\mathfrak{b}^-$ being the Borel subalgebra containing $f$. Let also $\mathrm{U}^-$ be the unipotent radical of the associated Borel subgroup of $\SO(U')$. By identifying $e+\mathfrak{z}_f\cong \Sigma_{2n}$, we get that $e + \mathfrak{b}^-\cong \mathrm{U}^-\times \Sigma_{2n}$. This isomorphism is $\GG_m$-equiariant, where the action on $e+\mathfrak{b}^-$ is given by $(2, -2\rho^\vee)$, and the action on $\mathrm{U}^-$ is given by $\mathrm{Ad}_{-2\rho^\vee}$. By projecting on the $\mathrm{U}^-$-coordinate and taking inverse, we get a $\GG_m$-equivariant map $\gamma:\Sigma\to \mathrm{U}^-$ such that if we set $\zeta(\sigma) = \gamma(\sigma)\xi(\sigma)$, then $\zeta(\sigma) = (\sigma, \zeta'(\sigma))$ for some $\zeta':\Sigma\to v^* + \ker(f^n)$. The map $\zeta'$ is $\GG_m$-equivariant, where the $\GG_m$-action on $v^*+\ker(f^n)$ is the twisted one, considered above.

Also, $v^*$ must belong to the weight 2 subspace of $h$ on $U'$. This subspace is one dimensional, and we choose $v_0^*$ as in Section \ref{sec: pseudo-slice}, so that $v^* = cv_0^*$ for some $c\in \CC$.

\begin{remark}
    Note that up until now we didn't use any specific feature of $\mathfrak{so}(U')^*\times U'$ (apart from our description of the fixed points for the twisted $\GG_m$-action and of the attractor), and we only need the moment map to $\mathfrak{so}(U')^*$, which is a general feature.
\end{remark}
Using the isomorphism $\Sigma_{e,f,h,v_0^*}\cong \Sigma$ of Section \ref{sec: pseudo-slice}, and denoting $$\omega=\omega(e+y,v_0^*+v):=\zeta'(e+y,v_0^*+v)-v^*\in \ker(f^n)$$ for $(e+y, v_0^*+v)\in \Sigma_{e,f,h,v_0^*}$, we get that for all such $y,v$ and all even $0\leq i \leq 2n-2$, one has
\begin{equation}\label{eq: section invariants}
    \langle v^*+\omega, (e+y)^{i} (v^*+\omega)\rangle = \langle v_0^*+v, (e+y)^{i} (v_0^*+v)\rangle.
\end{equation}
\begin{proposition}
    We must have $c=\pm 1$.
\end{proposition}
\begin{proof}
    First assume $n\geq 3$.
    Restricting to $i=2n-4$ and $v=0$, we take the differential of \eqref{eq: section invariants} at $y=0$. The differential of the left-hand side is given by
    \[
    \begin{aligned}
    \frac{d}{dy}\langle v^*+\omega, (e+y)^{2n-4} (v^*+\omega)\rangle &= \langle v^*, \frac{d}{dy}(e+y)^{2n-4} v^*\rangle + 2\langle e^{2n-4}v^*, \frac{d}{dy}\omega\rangle
    \\&\stackrel{n \geq 3}{=} 2(-1)^n\frac{d}{dy}\langle e^{n-2}v^*, y e^{n-3}v^*\rangle.
    \end{aligned}
    \]
    The last equality follows from $e^{n-1}v^*=0$.
    Comparing with the same expression for $(e+y, v_0^*)$, we get that
    $$\frac{d}{dy}\langle e^{n-2}v^*, y e^{n-3}v^*\rangle = \frac{d}{dy}\langle e^{n-2}v_0^*, y e^{n-3}v_0^*\rangle.$$
    That is, either $c^2=1$ or
    $$\frac{d}{dy}\langle e^{n-2}v_0^*, y e^{n-3}v_0^*\rangle = 0.$$
    By taking $y=\lambda f^{2n-5}$ and varying $\lambda$, we see that this is impossible, hence $c=\pm 1$.

    Now we give a proof for $n=2$.
    Take $y=f, v=0$. Since $(e+f,v_0^*)$ is a fixed point for the (twisted) action by $\sqrt{-1}\in \CC^\times$, as $f$ lies in the weight 4 subspace of the twisted action of $\GG_m$, we get that $(v^* + \omega)$ is a fixed point too. This implies that $\omega=\zeta'(e+f, v_0^*)\in \ker f$ (this is the only eigenspace for the twisted $\GG_m$-action for which the weight is positive and divisible by 4).
    Now writing
    $2\langle v^*, \omega \rangle = \langle v^* + \omega, v^* + \omega\rangle = \langle v_0^*, v_0^*\rangle = 0$
    we see that either $c=0$ or $\omega=0$. If $\omega=0$ we find that 
    $$\langle v^*, (e+f)^2v^*\rangle = \langle v_0^*, (e+f)^2v_0^*\rangle\neq 0,$$
    hence $c^2=1$.
    If $c=0$, choose $v\in \ker f$ so that $\langle v_0^*, v\rangle = \frac{1}{2}$ and take $y=0$.
    Utilizing the same argument as before (again $(e,v_0^*+v)$ is a fixed point for the twisted action by $\sqrt{-1}\in \CC^\times$, hence $\omega\in \ker(f)$), one finds that
    $$1 = \langle v_0^* + v, v_0^* + v\rangle = \langle \omega, \omega\rangle = 0,$$
    a contradiction.
\end{proof}
Potentially replacing our choice of $v_0^*$ by $-v_0^*$, we may assume $c=1$.
\begin{lemma}
    The image of $\zeta'$ is contained in the translation by $v_0^*$ of a linear subspace of $U'$, of dimension $n$.
\end{lemma}
\begin{proof}
Since the twisted $\GG_m$-action has even weights on $\mathfrak{so}(U')\times U'$ it factors through $\GG_m\xrightarrow{2}\GG_m$, i.e. we have another action of $\GG_m$, by half the weights, which is compatible with the map $\zeta'$. This is the $\GG_m$ action we shall use for the rest of the proof of this lemma. This action is contracting, hence it extends to an action of the multiplicative monoid $(\AA^1,\times)$.

Let $(x, v_0^*+v)\in \Sigma$. Taking derivative of the $(\AA^1,\times)$-action at $0$, we get
$$\frac{d}{d\alpha}(\alpha.(x, v_0^*+v))|_{\alpha=0} = (0, \frac{-1}{n^2(n-1)^2}\langle v, fv_0^*\rangle fv_0^*) \in \CC(0, fv_0^*).$$
Hence, if $(x,v_0^*+v)\mapsto (x, v_0^*+\omega)$, then
\begin{equation}\label{eq: linear constraint}
    D_{(e,v_0^*)}\zeta'(0, fv_0^*)\ni \frac{d}{d\alpha}(\alpha.(x, v_0^*+\omega))|_{\alpha=0} = (0, \frac{-1}{n^2(n-1)^2}\langle \omega, fv_0^*\rangle fv_0^* + \pi_J(\omega)).
\end{equation}
Here $\pi_J$ is the orthogonal projection on $J\subset U'$, where $J$ is as in Section \ref{sec: pseudo-slice}.
Note that the above expression appearing in the right-hand side is just the projection of $\omega$ on the weight 0 subspace of $h$ on $U'$.
This is a non-zero linear constraint on $\omega\in \ker(f^n)$. Since $\dim \ker(f^n)=n+1$, the claim follows.
\end{proof}
\begin{corollary}\label{cor: pointwise isomorphism}
    The image of $\zeta'$ is equal to the translation by $v_0^*$ of a linear subspace $S\subset U'$, of dimension $n$. Moreover, for each $x\in \Sigma$, the restriction of $\zeta'$ to $\{x\}\times \Sigma_\mathfrak{gl}$ is an isomorphism onto $v_0^*+S$. In particular, $\zeta$ gives an isomorphism $\Sigma\xrightarrow{\sim} \Sigma_{2n}\times (v_0^*+S)$.
\end{corollary}
\begin{proof}
    Let $S$ be a linear subspace of dimension $n$ as in the Lemma above. The composition $$\{x\}\times \Sigma_\mathfrak{gl}\xrightarrow{\zeta'} v_0^*+S \to \Sigma_{2n+1}$$ is the usual map, and the three of them are of dimension $n$, hence the image of $\zeta'$ is equal to $v_0^*+S$, and the map from $v_0^*+S$ to $\Sigma_{2n+1}$ must be finite. This map is given by $v_0^*+\omega\mapsto \langle v_0^*+\omega, x^{2i} (v_0^*+\omega)\rangle$, and so it is of degree $2^n$. This is also the degree of the projection from $\Sigma_\mathfrak{gl}$ to $\Sigma_{2n+1}$, hence the map $\zeta'|_{\{x\}\times \Sigma_\mathfrak{gl}}$ is of degree 1, i.e. an isomorphism.
\end{proof}
\begin{proposition}
    The above $S$ is equal to $\CC fv_0^*+N$, where $N:=\ker(f^{n-1})$ is as in Section \ref{sec: pseudo-slice}. That is, the image of $\zeta$ is equal to $\Sigma_{(e, f, h, v_0^*)}$.
\end{proposition}
\begin{proof}
    Let $$\epsilon := D_{(e,v_0^*)}\zeta'(0, fv_0^*)\in \CC fv_0^* + J.$$ By \eqref{eq: linear constraint} we have $S = v_0^* + \CC \epsilon + N$. We write $\epsilon = \alpha fv_0^* + \beta j$, with $j\in J$ satisfying $\langle j, j\rangle = -n^2(n-1)^2$.
    The claim of the proposition is equivalent to showing that $\beta=0$.
    
    We start by giving the proof in the case $n\geq 3$.
    Consider $$W:=\left(v_0^*+\mathrm{Span}(\{fv_0^*, f^2v_0^*, f^3v_0^*\})\right)\subset \Sigma_{(e, f, h, v_0^*)}.$$
    We claim that the image of $\{e\}\times W$ under $\zeta'$ is equal to $v_0^* + \mathrm{Span}(\epsilon, f^2v_0^*, f^3v_0^*)$.
    Indeed, let $w\in W$, and write $\zeta'((e, w))=b_0\epsilon + \sum_{i=1}^{n-1}b_i f^{i+1}v_0^*$. Let $i$ be the maximal integer such that $b_i\neq 0$.
    If $i > 2$, we have
    $$0 = \langle w, e^{2i} w\rangle = \langle \zeta'(e, w), e^{2i} \zeta'(e, w)\rangle = (-1)^i b_i^2 \langle e^if^{i+1}v_0^*, e^if^{i+1}v_0^*\rangle \neq 0.$$
    This shows that $\zeta'(\{e\}\times W)\subseteq v_0^* + \mathrm{Span}(\epsilon, f^2v_0^*, f^3v_0^*)$. Both are 3-dimensional affine spaces, and by Corollary \ref{cor: pointwise isomorphism}, $\{e\}\times W$ maps under $\zeta'$ isomorphically to its image. Thus, we get the desired equality.

    Let $w\in W, w = v_0^*+a_0fv_0^* + a_1f^2v_0^* + a_2 f^3v_0^*$, and $\zeta'(e, w) = v_0^* + b_0\epsilon + b_1f^2v_0^* + b_2f^3v_0^*$.
    Setting
    \[\begin{aligned}
        P_0(w) &= \frac{1}{2n^2(n-1)^2}\langle w, w\rangle = a_1-\frac{1}{2}a_0^2\\
        P_1(w) &= \frac{1}{2n^3(n-1)^3}\langle w, e^2 w\rangle = \frac{1}{2}n(n-1)a_1^2 - (n+1)(n-2)a_0a_2\\
        P_2(w) &= -\frac{1}{2n^4(n-1)^4(n+1)^2(n-2)^2}\langle w, e^4 w\rangle = \frac{1}{2}a_2^2\\
        P_0'(\zeta'(e, w)) &= \frac{1}{2n^2(n-1)^2}\langle \zeta'(e, w), \zeta'(e, w)\rangle = b_1-\frac{1}{2}(\alpha^2+\beta^2)b_0^2\\
        P_1'(\zeta'(e, w)) &= \frac{1}{2n^3(n-1)^3}\langle \zeta'(e, w), e^2 \zeta'(e, w)\rangle = \frac{1}{2}n(n-1)b_1^2 - (n+1)(n-2)\alpha b_0b_2\\
        P_2'(\zeta'(e, w)) &= -\frac{1}{2n^4(n-1)^4(n+1)^2(n-2)^2}\langle \zeta'(e, w), e^4 \zeta'(e, w)\rangle = \frac{1}{2}b_2^2
    \end{aligned}\]
    The Jacobian matrix of $(P_0, P_1, P_2)$ is $$\begin{pmatrix}
    -a_0&-(n+1)(n-2)a_2&0\\
    1&n(n-1)a_1&0\\
    0& -(n+1)(n-2)a_0 & a_2
    \end{pmatrix}$$
    and of $(P_0', P_1', P_2')$ is $$\begin{pmatrix}
    -(\alpha^2+\beta^2)b_0&-(n+1)(n-2)\alpha b_2&0\\
    1&n(n-1)b_1&0\\
    0& -(n+1)(n-2)\alpha b_0 & b_2
    \end{pmatrix}.$$
    If we restrict to the locus where $a_2$ is non-zero and the Jacobian matrix of $(P_0, P_1, P_2)$ is singular, i.e.
    $$a_2 = \frac{n(n-1)}{(n+1)(n-2)}a_0a_1,$$
    We get that the Jacobian matrix of $(P_0', P_1', P_2')$, must also be singular and $b_2\neq 0$ (by comparison of $P_2$ to $P_2'$), i.e.
    $$\alpha b_2 = \frac{n(n-1)}{(n+1)(n-2)}(\alpha^2+\beta^2)b_0b_1.$$
    We get that for any $a_0, a_1\in \CC$ which are non-zero, there are $b_0, b_1\in \CC\setminus\{0\}$ such that
    \[
    \begin{aligned}
        (P_0=P_0')&\ \ \ \ b_1-\frac{1}{2}(\alpha^2+\beta^2)b_0^2 = a_1-\frac{1}{2}a_0^2\\
        (P_1=P_1')&\ \ \ \ \frac{1}{2}b_1^2 - (\alpha^2+\beta^2)b_0^2b_1 = \frac{1}{2}a_1^2 - a_0^2a_1\\
        (P_2 = P_2')&\ \ \ \ (\alpha^2+\beta^2)^2b_0^2b_1^2 = \alpha^2 a_0^2a_1^2
    \end{aligned}
    \]

    Choosing $a_0=1, a_1 = 2$ we get that
    \[
    \begin{aligned}
        &b_1-\frac{1}{2}(\alpha^2+\beta^2)b_0^2 = \frac{3}{2}\\
        &\frac{1}{2}b_1 - (\alpha^2+\beta^2)b_0^2 = 0\\
        &(\alpha^2+\beta^2)^2b_0^2b_1^2 = 4\alpha^2
    \end{aligned}
    \]
    From the first two equation we get that $(\alpha^2+\beta^2)b_0^2=1, b_1=2$.
    Hence, the third equation reads $4(\alpha^2+\beta^2)=4\alpha^2$ and so $\beta=0$, as desired.

    Now we assume $n=2$.
    We denote $$\zeta'(e+f, v_0^* + a_0fv_0^* + a_1f^2v_0^*)= v_0^*+b_0\epsilon+b_1f^2v_0^*,$$
    where recall that $\epsilon = \alpha fv_0^* + \beta j$ and $j\in J, \langle j, j\rangle = -4$.
    We have
    \[\begin{aligned}
        P_0(a_0, a_1) &= \frac{1}{8}\langle v_0^* + a_0fv_0^* + a_1f^2v_0^*, v_0^* + a_0fv_0^* + a_1f^2v_0^*\rangle = a_1-\frac{1}{2}a_0^2\\
        P_1(a_0, a_1) &= \frac{1}{4}\langle (e+f)(v_0^* + a_0fv_0^* + a_1f^2v_0^*), (e+f) (v_0^* + a_0fv_0^* + a_1f^2v_0^*)\rangle = -(1+2a_1)^2 + 4a_0^2\\
        P_0'(b_0, b_1) &= \frac{1}{8}\langle v_0^*+b_0\epsilon+b_1f^2v_0^*, v_0^*+b_0\epsilon+b_1f^2v_0^*\rangle = b_1-\frac{1}{2}(\alpha^2+\beta^2)b_0^2\\
        P_1'(b_0, b_1) &= \frac{1}{4}\langle (e+f)(v_0^* + b_0\epsilon + b_1f^2v_0^*), (e+f) (v_0^* + b_0\epsilon + b_1f^2v_0^*)\rangle = -(1+2b_1)^2 + 4\alpha^2b_0^2\\
    \end{aligned}\]
    with $P_0 = P_0'$ and $P_1 = P_1'$.
    The Jacobian matrix of $(P_0, P_1)$ is $$\begin{pmatrix}
        -a_0 & 8a_0\\
        1 & -4(1+2a_1)
    \end{pmatrix}$$
    and its determinant is equal to $4a_0(-1+2a_1)$.
    The Jacobian matrix of $(P_0', P_1')$ is $$\begin{pmatrix}
        -b_0(\alpha^2+\beta^2)&8\alpha^2b_0\\
        1& -4(1+2b_1)
    \end{pmatrix}$$
    and its determinant is equal to $4b_0(-\alpha^2+\beta^2+(\alpha^2+\beta^2)2b_1)$.
    Setting $a_1 = \frac{1}{2}$, we get that we must have:
    \[
    \begin{aligned}
        (P_0=P_0')&\ \ \ \ b_1-\frac{1}{2}(\alpha^2+\beta^2)b_0^2 = \frac{1}{2}-\frac{1}{2}a_0^2\\
        (P_1=P_1')&\ \ \ \ -(1+2b_1)^2+4\alpha^2b_0^2 = -4+4a_0^2\\
        (\mathrm{Jacobian}=0)&\ \ \ \ 4b_0(-\alpha^2+\beta^2+(\alpha^2+\beta^2)2b_1)=0
    \end{aligned}
    \]
    From the first two equations we get
    $$(1-2b_1)^2 + 4\beta^2b_0^2 = 0$$
    We see that if $b_0=0$ then $b_1=\frac{1}{2}$, which is impossible if $a_0\neq 0$. So we must have by the third equation above
    $$2(\alpha^2+\beta^2)b_1 = \alpha^2-\beta^2.$$
    This implies that
    $$\left(\frac{\beta^2-\alpha^2}{\alpha^2+\beta^2}\right)^2 + 4\beta^2b_0^2 = 0.$$
    If $\beta\neq 0$, it would follow that the values of $b_0, b_1$ are fixed, but this is clearly impossible as $a_0$ is allowed to vary.
\end{proof}
\begin{corollary}
The morphism $\zeta$ is equal to a $\GG_m$-equivariant automorphism $\tau$ of $\Sigma$ over $\Sigma_{2n}\times \Sigma_{2n+1}$, composed with the inverse of the isomorphism of Claim \ref{claim: pseudo slice}, and then with the inclusion of $\Sigma_{(e, f, h, v_0^*)}$ into $\mathfrak{so}(U')\times U'$.
\end{corollary}
\begin{proof}
    We showed that $\zeta$ is equal to the composition of an isomorphism $\Sigma\to \Sigma_{(e, f, h, v_0^*)}$ and the inclusion of $\Sigma_{(e, f, h, v_0^*)}$ into $\mathfrak{so}(U')\times U'$. Both $\zeta$ and the above inclusion are $\GG_m$ equivariant (with respect to the twisted $\GG_m$ action on $\mathfrak{so}(U')\times U'$), hence we get that the isomorphism $\Sigma\to \Sigma_{(e, f, h, v_0^*)}$ we have is also $\GG_m$-equivariant.
    We also have compatibility with the maps to $\Sigma_{2n}\times \Sigma_{n+1}$.
    Identifying $\Sigma\cong \Sigma_{(e, f, h, v_0^*)}$ using the isomorphism of Claim \ref{claim: pseudo slice}, we get the desired result.
\end{proof}
Proposition \ref{prop: pseudo inverse} now follows immediately.
\section{Proof of Proposition \ref{prop: psi isom}, $n\geq 2$}\label{sec: conclusion}
Recall that by Claim \ref{claim: loc thm}, the morphism
$$\psi_\mathrm{loc}:\mathfrak{E}^\bullet_\mathrm{loc}\cong\mathfrak{G}^\bullet_\mathrm{loc}$$ is an isomorphism.
Since $\Spec(\mathfrak{G}^\bullet)$ is irreducible, and in particular has no irreducible components mapping non-dominantly to $\Sigma_{2n+1}\times \Sigma_{2n}$, the map
$$\mathfrak{G}^\bullet\hookrightarrow\mathfrak{G}^\bullet_\mathrm{loc}$$ is injective, hence $\psi:\mathfrak{G}^\bullet\hookrightarrow \mathfrak{E}^\bullet$ must also be injective.

To show that $\psi$ is surjective, take $f\in \mathfrak{E}^\bullet\hookrightarrow \mathfrak{G}^\bullet_\mathrm{loc}$, which we may view as a rational function on $\mathrm{Spec}(\mathfrak{G}^\bullet)$, whose pull back to $\Spec(\mathfrak{E}^\bullet)$ via $\psi$ is regular. The pull back of this function back to $SO(U')\times \Sigma$ via the maps in Proposition \ref{prop: pseudo inverse} is regular, hence by Corollary \ref{cor-item: regularity can be checked on Sigma} $f$ must regular too.
\qed
\begin{appendices}
\section{The case $n=1$} \label{appendix: n=1}
In this appendix we treat the case $n=1$. The results of Sections \ref{sec: Hecke actions}, \ref{sec: generation}, \ref{sec: deeq}, \ref{sec: Categorical Quotient}, and \ref{sec: localization} are valid for this case too.

We are able to compute $\mathfrak{E}^{\bullet}$ directly, as follows:
Let $L$ and $L^*$ be the two Lagrangian subspaces of $V'$. They are both $\SO(V')$-invariant, and so we may identify
$\mathcal{C}\cong \operatorname{D-mod}_{\SO(V')(\O)\times \Sp(V)(\O),\mathrm{lc}}(V\otimes L)$. We may also identify $L\cong \CC$ so that $\SO(V')\cong \GG_m$ acts on it by dilations.
Using this identification we will write $V$ for $V\otimes L$, were the $\SO(V')\cong \GG_m$-action is understood to be by dilations.

Let $\omega:\GG_m\xrightarrow{\mathrm{id}}\GG_m$ denote the fundamental representation. Then we clearly have $$\phi_{\GG_m}(\omega^{\otimes k}) * E_0 = \delta_{t^kV(\O)}.$$
Now we can explicitly compute
\begin{equation}\label{eq:n=1 computation}
    \begin{aligned}
        \mathrm{Ext}^\bullet(E_0, \phi_{\GG_m}(\CC[\GG_m]) * E_0) &= \bigoplus_{k\in \ZZ} \mathrm{Ext}^\bullet(\delta_{V(\O)}, \delta_{t^kV(\O)}) \otimes \omega^{-k} \\&= H^\bullet_{\GG_m\times \Sp(V)}(\mathrm{pt})\oplus \bigoplus_{k\in \ZZ_{>0}} H^\bullet_{\GG_m\times \Sp(V)}(\mathrm{pt})\otimes y^k \otimes \omega^{-k} \oplus\\& \bigoplus_{k\in \ZZ_{>0}} H^\bullet_{\GG_m\times \Sp(V)}(\mathrm{pt})\otimes z^k \otimes \omega^{k}
    \end{aligned}
\end{equation}
where $y, z$ are of cohomological degree $2$. Note that the above computes the module structure over $H^\bullet_{\GG_m\times \Sp(V)}(\mathrm{pt})$ but not the algebra structure - namely it does not tell us the product $yz$, which can be computed as an equivariant Euler class.

However, since we want to show that $\psi$ is an isomorphism (rather than existence of an abstract isomorphism), we will use a different argument.

\begin{lemma}
    In the notation of Claim \ref{claim: loc thm}, we have $\tilde{\mathfrak{G}}^\bullet_\mathrm{loc}\cong \CC[\SO(U')]\otimes \CC(\Sigma)$.
\end{lemma}
\begin{proof}
    We choose a coordinate $x$ (of degree 2) on $\mathfrak{so}(U')^*\cong \Sigma_2\cong \AA^1$, and coordinates $y,z$ (of degree 2) on $U'$ which represent a basis of $U'$ in which the quadratic form is given by $2yz$. Then the map to $\Sigma_3\cong \AA^1$ is given by $(x,y,z)\mapsto x^2-yz$. Thus, we get
    $$\tilde{\mathfrak{G}}^\bullet_\mathrm{loc}\cong \CC(x,w)[y,z] / (x^2 - yz-w^2)\cong \CC(\Sigma)[y', z'] / (y'z'-1)\cong \CC[\SO(U')]\otimes \CC(\Sigma),$$
    where $y' = \frac{y}{x+w}$ and $z' = \frac{z}{x-w}$. As the $\SO(U')$-action is given by rescaling $y$ and $z$ by $\lambda$ and $\lambda^{-1}$ respectively, we see that the above isomorphism is $\SO(U')$-equivariant.
\end{proof}
\begin{corollary}
    The map $\psi_\mathrm{loc}:\mathfrak{G}^\bullet_\mathrm{loc}\to \mathfrak{E}^\bullet_\mathrm{loc}$ is an isomorphism.
\end{corollary}
\begin{proof}
    Same as the proof of the case $n\geq 2$ given as part 3 of Claim \ref{claim: loc thm}.
\end{proof}

The map $\psi$ is compatible with the grading coming from the $\GG_m$ action (due to deequivariantization with respect to $\SO(U')\cong\GG_m$). Due to the observation that $\mathfrak{G}^\bullet\to \mathfrak{G}^\bullet_\mathrm{loc}$ is injective (as in Section \ref{sec: conclusion}) along with the above Corollary, $\psi$ is injective too. By \eqref{eq:n=1 computation} we see that in each weight of $\SO(U')\cong \GG_m$, $\mathfrak{E}^\bullet$ and $\mathfrak{G}^\bullet$ are of the same finite dimension (=1) as free modules over $H^\bullet_{\GG_m\times \Sp(V)}(\mathrm{pt})$, hence $\psi$ is an isomorphism, concluding the proof of Proposition \ref{prop: psi isom}, hence of theorem \ref{thm: main} in this case.

\section{Multiplicative line bundles on the affine Grassmannian and invariant bilinear forms}\label{appendix: line bundles and quadratic forms}
Let $G$ be a reductive group.
We give a way to attach to each line bundle $\mathcal{L}$ on the affine Grassmannian of $G$ equipped with multiplicative structure an invariant bilinear form on $\mathfrak{g}^*$. We also define certain endomorphisms on the Satake category attached to $\mathcal{L}$ and to such a bilinear form, and we show that they are the same.
\begin{definition}
    A multiplicative line bundle $\mathcal{L}$ on $\Gr_G$ is a $G_\O$-equivariant line bundle on $\Gr_G$ along with an isomorphism in $\mathrm{Pic}_{G_\O}(G_{\mathcal{K}}\times^{G_\O}\Gr_G)$ of $m^*\mathcal{L}$ and $\mathrm{pr}_1^*\mathcal{L}\otimes \mathrm{pr}_2^*(\mathcal{L})$, satisfying the evident cocycle condition. Here $m:G_{\mathcal{K}}\times^{G_\O}\Gr_G\to \Gr_G$ is the multiplication map, and $\mathrm{pr}_1: G_{\mathcal{K}}\times^{G_\O}\Gr_G \to \Gr_G$ and $\mathrm{pr_2}: G_{\mathcal{K}}\times^{G_\O}\Gr_G \to G_\O \backslash \Gr_G$ are the projections.
\end{definition}
Let $\mathcal{L}$ be a multiplicative line bundle on $\Gr_G$, and let $c_1(\mathcal{L})$ be the $G_\O$-equivariant first Chern class of $\mathcal{L}$. We define $c_\mathcal{L}$ to be the endomorphism (of cohomological degree 2) of identity functor of $\operatorname{D-mod}_{G_\O}(\Gr_G)$ given by cup product with $c_1(\mathcal{L})$.
\begin{proposition}
    We have $c_\mathcal{L}(\F_1 * \F_2) = c_\mathcal{L}(\F_1) * \mathrm{id}_{\F_2} + \mathrm{id}_{\F_1} * c_\mathcal{L}(\F_2)$.
\end{proposition}
\begin{proof}
    By the projection formula, $c_\mathcal{L}(\F_1 * \F_2)$ is the image under $m_*$ of the endomorphism given by cup product with $c_1(m^*\mathcal{L})$. Now the result follows from the multiplicative structure of $\mathcal{L}$.
\end{proof}
Restricting to the image of $\operatorname{D-mod}_{G_\O}(\Gr_G)^\heartsuit$ inside $\operatorname{D-mod}_{G_\O}(\Gr_G)$ and using the derived Satake isomorphism, we get an endomorphism $c_\mathcal{L}$ of the functor $$\mathrm{Rep}(\check{G})\to \mathrm{QCoh}(\check{\mathfrak{g}}^*[2] / \check{G})$$ satisfying $$c_\mathcal{L}(V_1\otimes V_2) = c_\mathcal{L}(V_1) \otimes \mathrm{id}_{V_2} + \mathrm{id}_{V_1} \otimes c_\mathcal{L}(V_2).$$
Applying the same argument as in \cite[Remark 5.4]{YZ} with $\Spec(R)$ replaced by the stack $\check{\mathfrak{g}}^*[2] / \check{G}$ we get that $c_\mathcal{L}$ determines an element in $$(\check{\mathfrak{g}}\otimes \Sym(\check{\mathfrak{g}}[-2]))^{\check{G}}.$$
Since it is of cohomological degree 2, we get an element $e_\mathcal{L}\in (\check{\mathfrak{g}}\otimes \check{\mathfrak{g}})^{\check{G}}$.
Spelling out the above construction explicitly, the endomorphism $c_\mathcal{L}(V)$ is given by the image of $e_\mathcal{L}$ under the map $$(\check{\mathfrak{g}}\otimes \check{\mathfrak{g}})^{\check{G}}\to (\mathrm{End}(V)[-2]\otimes \Sym(\check{\mathfrak{g}}[-2]))^{\check{G}}.$$

\begin{lemma}\label{lemma: line bundle to bilinear form for torus}
    Assume that $G=T$ is a torus, and let $\mathcal{L}$ be a multiplicative line bundle on $\Gr_T$ as above.
    Let $e_\mathcal{L}'$ be the bilinear form on $X_*(T)$ given by sending $\lambda,\mu\in X_*(T)$ to the weight of $\mu$ acting on $\mathcal{L}_\lambda$ (the fiber at $t^\lambda\in \Gr_T$ of $\mathcal{L}$). Then the restriction of $e_\mathcal{L}$ to $X_*(T)\subset \mathfrak{t}\cong \check{\mathfrak{t}}^*$ is equal to $e_\mathcal{L}$.
\end{lemma}
\begin{proof}
    Let $\lambda, \mu$ be as above. Under the derived Satake isomorphism for $T$, $V^\lambda\otimes \O_{\check{\mathfrak{t}}^*}$ corresponds to the skyscraper sheaf at $\Gr_T^\lambda$. Thus, we get that the pairing of $\lambda$ and $e_\mathcal{L}$ is equal to the weight of $T$ on $\mathcal{L}_\lambda$, viewed as an element in $\mathfrak{t}^*\cong \check{\mathfrak{t}}$. Pairing $\mu$ with this element, we get the weight of $\mu$ on $\mathcal{L}_\lambda$, as desired.
\end{proof}

Fix a Borel $\check{B}$ of $\check{G}$ and an embedding of $\check{T}$ into $\check{B}$. Let $\check{N}$ be the unipotent radical of $\check{B}$, and let $\check{B}^-$ and $\check{N}^-$ be the opposite Borel and its unipotent radical.

In order to compute $e_\mathcal{L}$ from $\mathcal{L}$, we will use the following result:
\begin{theorem}[\cite{GR}]
    The derived Satake isomorphisms for $G$ and $T$ intertwine the !-pullback functor along the embedding $T_\O \backslash\Gr_T\to G_\O\backslash \Gr_G$ and the functor $\mathrm{QCoh}(\check{\mathfrak{g}}^*[2] / \check{G})\to \mathrm{QCoh}(\check{\mathfrak{t}}^*[2] / \check{T})$ sending $M$ to $$(M \otimes_{\O_{\check{\mathfrak{g}}^*[2]}} \O^{[]}_{T^*(\check{G}/\check{N})})^{\check{G}},$$ where $\O^{[]}_{T^*(\check{G}/\check{N})}$ is the structure sheaf of $\check{G}\times^{\check{N}} \check{\mathfrak{n}}^\perp[2]$, and $\check{\mathfrak{n}}^\perp\subset \check{\mathfrak{g}}^*$ is the annihilator of $\check{\mathfrak{n}}$.
\end{theorem}
Note that this functor sends $V\otimes \O_{\check{\mathfrak{g}}^*[2]}$ to $(V\otimes \O_{\check{\mathfrak{n}}^\perp[2]})^{\check{N}}$.
\begin{proposition}
    Let $\mathcal{L}_T$ be the restriction of $\mathcal{L}$ to a $T_\O$-equivariant multiplicative line bundle on $\Gr_T$. The bilinear form $e_{\mathcal{L}_T}\in \check{\mathfrak{t}}\otimes \check{\mathfrak{t}}$ associated to it under the previous construction for the group $T$ is the image of $e_\mathcal{L}\in (\check{\mathfrak{g}}\otimes \check{\mathfrak{g}})^{\check{G}}$ under the map $\check{\mathfrak{g}}\to \check{\mathfrak{t}}$ annihilating $\check{\mathfrak{n}}$ and $\check{\mathfrak{n}}^-$.
\end{proposition}
\begin{proof}
    Let $\overline{e_\mathcal{L}}\in\check{\mathfrak{t}}\otimes \check{\mathfrak{t}}$ be the second bilinear form described above, which we wish to show is equal to $e_{\mathcal{L}_T}$.
    We must compare the two endomorphisms of $(V\otimes \O_{\check{\mathfrak{n}}^\perp})^{\check{N}}\in \mathrm{QCoh}(\check{\mathfrak{t}}^*[2]/\check{T})$ (for $V\in \mathrm{Rep}(\check{G})$) resulting from $\overline{e_\mathcal{L}}$ and $e_{\mathcal{L}_T}$. It is enough to compare them on the regular semisimple locus of $\check{\mathfrak{t}}^*$. We have $\check{\mathfrak{n}}^\perp_{\mathrm{rss}} \cong \check{N}\times \check{\mathfrak{t}}^*_{\mathrm{rss}}$, with the projection on $\check{\mathfrak{t}}^*_{\mathrm{rss}}$ given by the composition $\check{\mathfrak{n}}^\perp\subset \check{\mathfrak{g}}^*\to \check{\mathfrak{t}}^*$ (note that this map is indeed $\check{N}$-equivariant) and a section given by $s:\check{\mathfrak{t}}^*\to (\check{\mathfrak{b}}^-)^*\cong \check{\mathfrak{n}}^\perp$. Thus, it is enough to compare the two endomorphisms when restricted to $\check{\mathfrak{t}}^*$ through $s$.
    
    This restriction of the endomorphism induced by $e_{\mathcal{L}_T}$ is given by the image of $e_\mathcal{L}$ under the map $(\check{\mathfrak{g}}\otimes \check{\mathfrak{g}})^{\check{G}}\to (\check{\mathfrak{g}}\otimes \check{\mathfrak{t}})$, given by annihilating $\check{\mathfrak{n}}$ and $\check{\mathfrak{n}}^-$. However, this map factors through $\check{\mathfrak{t}}\otimes \check{\mathfrak{t}}$ (this can be seen by considering the weights of $\check{T}$), and the desired equality $e_{\mathcal{L}_T} = \overline{e_\mathcal{L}}$ follows.
\end{proof}
In the case that $\mathfrak{g}$ is semisimple with Weyl group $W$, all invariant bilinear forms on $\check{\mathfrak{g}}^*$ are symmetric, and the above map $(\check{\mathfrak{g}}\otimes \check{\mathfrak{g}})^{\check{G}}\to \check{\mathfrak{t}}\otimes \check{\mathfrak{t}}$ defines an isomorphism $(\check{\mathfrak{g}}\otimes \check{\mathfrak{g}})^{\check{G}}\cong (\Sym^2\check{\mathfrak{t}})^{W}$.
Combining the above results, we get the following Corollary:
\begin{corollary}\label{cor: derived Satake of canonical endomorphism}
    Assume that $G$ is semisimple, and let $\mathcal{L}$ be a multiplicative line bundle on $\Gr_G$. Let $e_{\mathcal{L}_T}$ be the bilinear form on $X_*(T)$ pairing $\lambda,\mu$ to the weight of $\mu$ on the fiber of $\mathcal{L}$ at $t^\lambda\in \Gr_G$. It is symmetric and $W$-invariant.
    Let $e_\mathcal{L}$ be the invariant symmetric bilinear form on $\check{\mathfrak{g}}^*$ whose image in $(\Sym^2\check{\mathfrak{t}})^{W}$ is given by $e_{\mathcal{L}^T}$. Given $V\in\mathrm{Rep}(\check{G})$, the element $$-\cup c_1(\mathcal{L})\in \Ext^2_{\operatorname{D-mod}_{G_\O}(\Gr_G)}(\phi_G(V), \phi_G(V))$$ corresponds under derived Satake to the image of $$e_\mathcal{L}\in (\Sym^2\check{\mathfrak{g}})^{\check{G}}\to (\check{\mathfrak{g}}\otimes \check{\mathfrak{g}})^{\check{G}}\to (\mathrm{End}(V)\otimes \Sym(\check{\mathfrak{g}}[-2]))^{\check{G}}[2]\cong \mathrm{Ext}^2_{\mathrm{QCoh}_{\check{G}}(\check{\mathfrak{g}}^*[2])}(V\otimes \mathcal{O}_{\check{\mathfrak{g}}^*[2]}, V\otimes \mathcal{O}_{\check{\mathfrak{g}}^*[2]}).$$
\end{corollary}

\subsection{The determinant line bundle}\label{appendix: det line bundle}
The missing ingredient in the above discussion is the construction of multiplicative line bundles. In this section we construct a multiplicative line bundle for $\SL_n$, called the determinant line bundle, and discuss the resulting pullbacks to symplectic and orthogonal groups.

This line bundle is well studied, starting from \cite{BD}, and actually possesses a factorizable structure (which in particular implies a multiplicative structure). The reader is warned that while this line bundle, as well as its pullback from $\SL_{2n}$ to $\Sp_{2n}$, is widely referred to as the determinant line bundle, the pullback from $\SL_n$ to $\SO_{n}$ of this line bundle is usually not called the determinant line bundle. Rather, what is usually called the determinant line bundle (or Pfaffian line bundle, cf. \cite{BD}) on the connected component of $1\in \Gr_{\SO_n}$ is the ample generator of the Picard group, which is a square root of the restriction of the above line bundle.
\subsubsection{Special linear groups}
Let $G=\SL_n$. Recall that points in $\Gr_{\SL_n}$ classify lattices in $\mathcal{K}^n$ with determinant equal to $\O$. Let $\mathcal{D}$ be the determinant line bundle on $\Gr_{\SL_n}$, whose fiber at a lattice $\Lambda$ is given by $\det(\Lambda/\Lambda\cap \O^n)^{-1}\otimes \det(\O^n/\Lambda\cap \O^n)$, i.e. the inverse of the relative determinant of $\Lambda$ with respect to $\O^n$. Due to the transitivity of relative determinants (i.e. if $\Lambda_1,\Lambda_2, \Lambda_3$ are lattices then the relative determinant of $\Lambda_1$ with respect to $\Lambda_3$ is the tensor product of the relative determinant of $\Lambda_1$ with respect to $\Lambda_2$ and the relative determinant of $\Lambda_2$ with respect to $\Lambda_3$), we get that $\mathcal{D}$ has a natural multiplicative structure.

Let $\mu=(1,0,\dots, 0, -1)$ be the dominant minimal cocharacter of $T$.
It acts on the fiber of $\mathcal{D}$ at $t^\mu$ with weight $2$. The associated $W$-invariant quadratic form on $X_*(T)$, in the sense of Corollary \ref{cor: derived Satake of canonical endomorphism} is then given by $\lambda=(\lambda_1, \dots, \lambda_n)\mapsto \sum \lambda_i^2$.

Given any reductive group $H$ mapping to $\SL_n$, we get a multiplicative line bundle on $\Gr_H$ by pulling back $\mathcal{D}$ along the induced map $\Gr_H\to \Gr_{\SL_n}$. The associated quadratic form is then given by the restriction along $X_*(T_H)\to X_*(T)$ of the above quadratic form on $X_*(T)$.

\subsubsection{Orthogonal and symplectic groups}
We now consider the situations of $\Sp(V), \SO(V')$, where we use the notation of Section \ref{sec: notation}.
Start with $G=\SO(V')$. Consider the quadratic form on $\mathrm{End}(U')\cong U'\otimes (U')^*$, which can be also written as $X\mapsto \mathrm{tr}(X^2)$. Its restriction to $\mathfrak{so}(U')$ is non-degenerate, and we consider inverse, which is an invariant quadratic form on $\mathfrak{so}(U')^*$ we denote by $B_{U'}$.
In the standard choice of a basis of $X_*(T)$ (that is, $(0,\dots, 0, 1, 0, \dots, 0)$) the resulting $W$-invariant quadratic form on $X_*(T)$ takes the value $1/2$ on every basis element.
Considering the embedding $\SO(V')\to \SL(V')$, we get the multiplicative line bundle $\mathcal{D}_{\SO(V')}$ on $\Gr_{\SO(V')}$.
It corresponds (in the sense of Corollary \ref{cor: derived Satake of canonical endomorphism}) to the invariant bilinear form on $\mathfrak{so}(U')^*$ given by $4B_{U'}$.

Let $Q\subset \Gr_{\SO(V')}$ be the minuscule orbit, isomorphic to the quadratic of isotropic vectors in $\PP(V')$. Any $\SO(V')$-equivariant line bundle on $Q$ is an integer power of the restriction from $\PP(V')$ of $\O(1)$, with the natural equivariant structure.
Let us compute the power of $\O(1)$ equal to the restriction of $\mathcal{D}_{\SO(V')}$ to $Q$. Letting $\mu$ be the dominant minuscule cocharacter of $T$, we get that $\mu$ acts on the fiber of $\mathcal{D}_{\SO(V')}$ at $t^\mu$ with weight $2$. In comparison, a direct computation shows that the cocharacter $\mu$ acts on the fiber of $\O(1)$, which is the dual to the tautological bundle, with weight $1$. We deduce that as $\SO(V')_\O$-equivariant line bundles, we have $\mathcal{D}_{\SO(V')}|_Q\cong \O(2)$.

For $G=\Sp(V)$ the situation is similar. We define $B_U$ in the same way as before, and again in the standard basis of cocharacters the associated $W$-invariant quadratic form on $X_*(T)$ takes the value $1/2$ on every basis element. Again we have $\mathcal{D}_{\Sp(V)}$, which is the determinant line bundle pulled back from $\SL(V)$, whose associated quadratic form is $4B_U$.
Now considering the minimal orbit $\Gr_{\Sp(V)}^\mathrm{min}\in \Gr_{\Sp(V)}$, and $\mu\in X_*(T)$ the minimal dominant cocharacter, we once again have that all $\Sp(V)_\O$-equivariant line bundles on $\Gr_{\Sp(V)}^\mathrm{min}$ are pulled back from $\Sp(V)$-equivariant line bundles on $\PP(V)$ along the projection $\Gr_{\Sp(V)}^\mathrm{min}\to \PP(V)$ described in Section \ref{sec: Hecke actions}. These are all integral powers of $\O(1)$ with the natural equivariant structure. Again the weight of the action of $\mu$ on the fiber at $t^\mu\in \Gr_{\Sp(V)}^\mathrm{min}$ of $\mathcal{D}_{\Sp(V)}$ is $2$, while the weight of the action of $\mu$ on the pullback of $\O(1)$ is $1$ (by the same kind of direct computation as before). Thus, we get $\mathcal{D}_{\Sp(V)}|_{\Gr_{\Sp(V)}^\mathrm{min}} \cong \O(2)$.

\end{appendices}

\printbibliography
\end{document}